\documentclass[a4paper,10pt,reqno]{amsart}
\usepackage{a4wide,color,array}
\usepackage{cite}
\usepackage{hyperref}
\usepackage{amsmath,amssymb,amsthm,color}
\hypersetup{
	linktocpage,
	colorlinks,
	linkcolor={red},
	citecolor={blue},
	bookmarksdepth=2
}
\usepackage[english]{babel}
\usepackage[T1]{fontenc}
\usepackage[utf8]{inputenc}
\usepackage{dsfont}
\usepackage{enumitem}
\usepackage[a4paper,left=2.6cm,right=2.6cm,top=3cm,bottom=3.5cm]{geometry}%
\usepackage{marginnote}%

\usepackage{float}
\usepackage{subcaption}
\usepackage{accents}
\usepackage{mathrsfs}

\usepackage{xspace}
\usepackage{bm}
\usepackage{bbm,upgreek}
\usepackage[e]{esvect}
\usepackage{esint}
\usepackage{etoolbox}
\usepackage{calc}
\usepackage{eso-pic}

\usepackage{lmodern}
\usepackage{xr}
\numberwithin{equation}{section}

\setlist{nosep}
\setlist{noitemsep}

\newtheorem{theorem}{Theorem}%

\newtheorem{proposition}{Proposition}[section]
\newtheorem{lemma}[proposition]{Lemma}%
\newtheorem{fact}[proposition]{Fact}%
\newtheorem*{question}{Question}%
\newtheorem*{answer}{Answer}%
\newtheorem{remark}{Remark}[section]%

\newtheorem{assumption}[remark]{Assumption}
\newtheorem{definition}[proposition]{Definition}%
\newcommand{\dR}{\mathbb{R}}%
\newcommand{\dN}{\mathbb{N}}%

\newcommand{\ve}{\varepsilon}

\newcommand{\dP}{\mathbb{P}}%
\newcommand{\dE}{\mathbb{E}}%
\newcommand{\Var}{\mathrm{Var}}%
\newcommand{\Cov}{\mathrm{Cov} }%

\newcommand{\supp}{\mathrm{supp}}
\newcommand{\Law}{\mathrm{Law}}

\newcommand{\Ising}{\mathrm{Ising}}

\newcommand{\mc}{\mathcal}

\newcommand{\dd}{\mathrm{d}}
\newcommand{\TAP}{\mathrm{TAP}}

\newcommand{\eq}{\mathrm{eq}}
\newcommand{\Band}{\mathrm{Band}}

\newcommand{\bfm}{\mathbf{m}}
\newcommand{\bfa}{\mathbf{a}}
\newcommand{\bfb}{\mathbf{b}}
\newcommand{\zhat}{\widehat\zeta}
\newcommand{\Shift}{\mathrm{Shift}}% the q-collapse of a full measure to a band measure

\usepackage{graphicx}

\title[Large deviations for the maximum of the generalized TAP free energy]{Large deviations for the maximum of\\ the generalized TAP free energy}

\author{Jeanne Boursier}
\address{Department of Mathematics, Massachusetts Institute of Technology, Cambridge, MA, USA}
\email{boursier@mit.edu}

\begin{document}
	
	\begin{abstract}
		We study large deviations of the maximum of the generalized TAP free
		energy introduced by Chen, Panchenko, and Subag~\cite{ChenPanchenkoSubag}
		for the Ising mixed \(p\)-spin model. The upper bound uses a new
		Guerra-type interpolation that may also be useful for other spin-glass
		models. The lower bound follows the strategy of Huang and
		Sellke~\cite{HuangSellke}. We show that supersymmetric critical points have
		no competitor at positive overlap on their own sphere. As such, they form
		a spherical code, which turns the annealed count into a lower bound on the
		probability of existence. This identifies the supersymmetric formula
		proposed by physicists with the large-deviation exponent for the existence
		of TAP maxima, rather than with the ordinary annealed complexity. The
		resulting rate function is the Legendre
		transform of a constrained Parisi value in the bottom-atom mass. The same
		argument, iterated in bands above TAP ancestors, would give a constructive
		proof of the Parisi formula for the Ising model under a technical strict
		stability assumption.
	\end{abstract}
	
	\maketitle
	
	\setcounter{tocdepth}{1}
	\tableofcontents
	
	%==============================================================================
	\section{Introduction}
	%==============================================================================
	
	\subsection{Setting}\label{sub:setup}
	
	Spin-glass theory has often been driven by conjectural exact formulas
	obtained through nonrigorous methods. The replica method predicted the
	Parisi formula for the free energy, while supersymmetric calculations
	produced Legendre formulas for the complexity of TAP states. The Parisi
	formula has since been proved, but the probabilistic meaning of the
	supersymmetric formulas has remained unclear.
	
	In this paper, we consider the Ising mixed $p$-spin model on
	\(\Sigma_N:=\{-1,1\}^N\). Let
	\((H_N(\bfm))_{\bfm\in[-1,1]^N}\) be the centered Gaussian process with
	covariance
	\[
	\dE\bigl[H_N(\bfm)H_N(\bfm')\bigr]
	=
	N\xi\!\left(\frac{\langle\bfm,\bfm'\rangle}{N}\right),
	\qquad \bfm,\bfm'\in[-1,1]^N,
	\]
	where
	\begin{equation}\label{eq:mixture-assumption}
		\xi(x)=\sum_{p\ge2}\beta_p^2x^p,
		\qquad \beta_p^2\ge0,
	\end{equation}
	has radius of convergence larger than $1$. The restriction of \(H_N\) to
	\(\Sigma_N\) is the Hamiltonian of the model. The normalized partition
	function is
	\begin{equation}\label{eq:normalized-partition-function}
		Z_N:=2^{-N}\sum_{\sigma\in\Sigma_N}e^{H_N(\sigma)}.
	\end{equation}
	
	The limiting free energy is described by the following variational problem.
	Given $\zeta\in\mc P([0,1])$, let
	$\Phi_\zeta:[0,1]\times\dR\to\dR$ be the solution of
	\begin{equation}\label{eq:parisi-pde}
		\partial_s\Phi_\zeta(s,x)
		+
		\frac{\xi''(s)}2
		\left[
		\partial_{xx}\Phi_\zeta(s,x)
		+
		\zeta([0,s])\bigl(\partial_x\Phi_\zeta(s,x)\bigr)^2
		\right]
		=0,
		\qquad
		\Phi_\zeta(1,x)=\log\cosh x.
	\end{equation}
	The corresponding Parisi functional is
	\begin{equation}\label{eq:parisi-value}
		P_\zeta
		:=
		\Phi_\zeta(0,0)
		-
		\frac12\int_0^1s\xi''(s)\zeta([0,s])\,\dd s.
	\end{equation}
	The Parisi formula~\cite[Theorem~3.1]{PanchenkoBook} and Gaussian
	concentration give
	\begin{equation}\label{eq:parisi}
		\lim_{N\to\infty}\frac1N\log Z_N
		=
		\inf_{\zeta\in\mc P([0,1])}P_\zeta
		=:
		f_{\eq}
	\end{equation}
	almost surely.
	
	The Parisi picture behind \eqref{eq:parisi} is that the Gibbs measure splits
	into a hierarchy of clusters whose masses are described by nested Chinese
	restaurant processes. The center of each cluster is called a magnetization.
	
	Chen, Panchenko, and Subag
	\cite[Theorems~1 and~2]{ChenPanchenkoSubag} identified the free-energy
	cost of fixing the barycenter of many replicas at
	\(\bfm\in[-1,1]^N\). Write \(q=\|\bfm\|^2/N\). For
	\(\eta\in\mc P([q,1])\), let
	\[
	\Phi_\eta^*(q,a)
	:=
	\inf_{x\in\dR}\{\Phi_\eta(q,x)-ax\},
	\]
	and set
	\begin{equation}\label{eq:FTAP-eta}
		F_{\TAP,\eta}(\bfm)
		:=
		H_N(\bfm)
		+
		\sum_{i=1}^N\Phi_\eta^*(q,m_i)
		-
		\frac N2\int_q^1s\xi''(s)\eta([0,s])\,\dd s.
	\end{equation}
	The generalized TAP free energy is
	\begin{equation}\label{eq:FTAP}
		F_{\TAP}(\bfm)
		:=
		\inf_{\eta\in\mc P([q,1])}F_{\TAP,\eta}(\bfm).
	\end{equation}
	
	\subsection{Controversies}\label{sub:controversies}
	
	Since the late 1970s, physicists have sought the exponential growth rate of the
	number of critical points of the TAP free energy at level $f$, both on average
	(annealed) and typically (quenched). Here the functional was the classical one,
	i.e.\ \eqref{eq:FTAP} with the infimum over $\eta$ dropped and $\eta$ fixed to
	$\delta_q$. The counting of metastable states goes back to Bray and
	Moore~\cite{AJBray}. The Rome group then recast the Kac--Rice count of critical
	points at a fixed level as a supersymmetric integral: the Hessian determinant
	becomes a fermionic Grassmann integral, and the gradient and energy constraints
	become bosonic integrals with Lagrange
	multipliers~\cite{cavagnaformal,parisi2004supersymmetry,crisanti2003complexity}.
	A Becchi--Rouet--Stora--Tyutin supersymmetry of the action shows that, once the
	absolute value of the determinant is dropped, the dominant saddle solves a set of
	Ward identities~\cite{annibalerole}. The resulting supersymmetric complexity,
	studied further in~\cite{TAPIsing}, is a Legendre transform whose dual parameters
	are the bottom-atom mass $\zeta(\{0\})$ and the level $f$.
	
	This picture is, however, controversial. Dropping the absolute value is
	legitimate only when the Hessian keeps a fixed signature, and this may fail: a
	non-supersymmetric branch could then dominate~\cite{parisi2004supersymmetry}.
	
	\begin{question}
		Is the supersymmetric formula the correct annealed complexity?
	\end{question}
	
	\begin{answer}
		No, not for the unsigned count in general.
		Appendix~\ref{sec:spherical-examples} gives spherical examples in which the
		ordinary annealed complexity is strictly larger than the supersymmetric one.
	\end{answer}
	
	\begin{question}
		How can this be, when the supersymmetric formulas are so clean?
	\end{question}
	
	\begin{answer}
		Because they are counting something else. Also notice that these formulas are valid only for
		$f\ge f_{\eq}$, and they always produce an exponent that is zero or negative.
	\end{answer}
	
	\begin{question}
		What, then, is the supersymmetric quenched complexity? If the states being
		counted have a small probability of existing, would we not get
		$\dE[\log(\mathrm{number})]=\log(0^+)=-\infty$?
	\end{question}
	
	\begin{answer}
		Writing $\log n=\lim_{\ve\downarrow0}\tfrac1\ve(n^\ve-1)$, one sees that the quenched complexity at level $f$ morally corresponds to the probability of existence of a TAP state at level $f$. The supersymmetric formulas thus concern the large deviations of
		the maximum of the TAP free energy.
	\end{answer}
	
	Notice that the annealed and quenched complexities of TAP states at a level
	$f\ge f_{\eq}$ may differ. Indeed, the difference of the two exponents is the rate
	of the conditional expected number of states given that at least one exists,
	\[
	\dE[\mathsf N\mid\mathsf N\ge1]
	=\frac{\dE\mathsf N}{\dP(\mathsf N\ge1)}.
	\]
	Because of the correlations of the Hamiltonian, this conditional count can be exponentially large, so the annealed exponent can strictly exceed the quenched one. 
	
	The supersymmetric formula computes the annealed exponent of isolated local
	maxima of the TAP free energy at high levels. As such, their Hessians are
	negative definite, justifying the supersymmetric manipulation.
	This corrects \cite[Conjectures~1 and~2]{TAPIsing}, where the Legendre
	transform in the bottom-atom mass was conjectured to describe the ordinary
	annealed complexity and the Legendre transform in the top cumulative mass
	the ordinary quenched complexity.
	
	\subsection{Main results}\label{sub:main-results}
	
	Throughout the paper, the mixture satisfies the following
	assumption.
	
	\begin{assumption}[Covariance function]\label{ass:xi}
		The function \(\xi\) is as in \eqref{eq:mixture-assumption} and is
		strictly convex on \([-1,1]\), with \(\xi(0)=\xi'(0)=0\).
	\end{assumption}
	
	For \(f\ge f_{\eq}\) and \(\ve>0\), let \(\mc E_{N,\ve}(f)\) be the
	event that there exist \(\zeta\in\mc P([0,1])\) and
	\(\bfm\in(-1,1)^N\) such that
	\[
	\nabla F_{\TAP,\zeta}(\bfm)=0,
	\qquad
	\left|\frac1N F_{\TAP}(\bfm)-f\right|\le\ve,
	\qquad
	\left|\frac1N F_{\TAP,\zeta}(\bfm)-f\right|\le\ve.
	\]
	For \(0\le\theta<1\), define
	\[
	\Lambda(\theta)
	:=
	\theta
	\inf_{\substack{\zeta\in\mc P([0,1])\\
			\zeta(\{0\})\ge\theta}}
	P_\zeta.
	\]
	Define \(\bar\Lambda\) on \(\mathbb R\) by
	\[
	\bar\Lambda(\theta)
	=
	\begin{cases}
		\Lambda(\theta), & 0\le\theta<1,\\
		\displaystyle\liminf_{u\uparrow1}\Lambda(u), & \theta=1,\\
		+\infty, & \theta\notin[0,1],
	\end{cases}
	\]
	and set
	\[
	\Sigma(f)
	:=
	\inf_{\theta\ge0}
	\bigl[\bar\Lambda(\theta)-\theta f\bigr].
	\]
	Set
	\[
	f_1:=\bar\Lambda'_-(1).
	\]
	
	\begin{definition}[Regular levels]
		\label{def:regular-levels}
		A level \(f\in(f_{\eq},f_1)\) is regular if \(\Sigma\) is
		differentiable at \(f\). Equivalently, the minimizer of
		\(\theta\mapsto\bar\Lambda(\theta)-\theta f\) is unique.
	\end{definition}
	
	The strict-convexity argument of Auffinger and
	Chen~\cite[Theorem~2]{AuffingerChen} applies to the constrained problem
	defining \(\Lambda(\theta)\). Thus its minimizer is unique for every
	\(\theta\in(0,1)\).
	
	Unless an initial condition is specified, the Auffinger--Chen process
	associated with \(\zeta\in\mc P([0,1])\) starts at \(X_0=0\). A positive
	contact point of \(\zeta\) is a minimizer over \((0,1]\) of the Parisi
	obstacle
	\[
	t\longmapsto
	\frac12\int_t^1\xi''(s)
	\left(
	\dE\left[
	\left(\partial_x\Phi_\zeta(s,X_s)\right)^2
	\right]-s
	\right)\,\dd s.
	\]
	When the set of positive contact points has a smallest element, it is
	called the first positive contact point.
	
	\begin{theorem}[Large deviations for the existence of TAP states]
		\label{thm:legendre}
		Fix a regular \(f\in(f_{\eq},f_1)\). Let \(\theta_f\in(0,1)\) be the
		unique minimizer defining \(\Sigma(f)\), let \(\zeta_f\) be the
		corresponding constrained Parisi minimizer, and let \(q_f\) be its first
		positive contact point. Let \(X^f\) be the Auffinger--Chen process
		associated with \(\zeta_f\). Assume the strict Plefka condition
		\begin{equation}\label{eq:onshell-strict-plefka}
			1-\xi''(q_f)\dE\left[
			\left(
			\partial_{xx}\Phi_{\zeta_f}(q_f,X^f_{q_f})
			\right)^2
			\right]>0.
		\end{equation}
		Then
		\begin{equation}\label{eq:existence-tap}
			\begin{aligned}
				\lim_{\ve\downarrow0}\liminf_{N\to\infty}
				\frac1N\log\dP\bigl(\mc E_{N,\ve}(f)\bigr)
				&=
				\lim_{\ve\downarrow0}\limsup_{N\to\infty}
				\frac1N\log\dP\bigl(\mc E_{N,\ve}(f)\bigr)\\
				&=
				\Sigma(f)
				=
				\inf_{\theta\in(0,1)}
				\bigl[\Lambda(\theta)-\theta f\bigr].
			\end{aligned}
		\end{equation}
		Moreover,
		\begin{equation}\label{eq:legendre}
			\lim_{N\to\infty}
			\frac1N
			\log
			\dE\exp\left\{
			\theta_f\max_{\bfm\in[-1,1]^N}F_{\TAP}(\bfm)
			\right\}
			=
			\Lambda(\theta_f).
		\end{equation}
	\end{theorem}
	
	\begin{remark}
		Since \(\zeta_f\) is the constrained Parisi minimizer, the non-strict
		Plefka inequality
		\[
		1-\xi''(q_f)\dE\left[
		\left(
		\partial_{xx}\Phi_{\zeta_f}(q_f,X^f_{q_f})
		\right)^2
		\right]\ge0
		\]
		holds at \(q_f\). Thus \eqref{eq:onshell-strict-plefka} only requires
		this inequality to be strict. This strictness and the strict convexity
		assumption on \(\xi\) are technical. We expect that both can be removed.
	\end{remark}

	\begin{remark}
		Talagrand's large-deviation principle
		\cite[Section~7]{LDTalagrand} identifies
		\(-\Sigma\) as the upper-tail rate function of
		\(N^{-1}\log Z_N\). Equation~\eqref{eq:existence-tap} gives the same
		rate for the existence of TAP states at regular levels in
		\((f_{\eq},f_1)\). This coincidence is natural from the generalized
		TAP representation. An upper deviation of the free energy is realized
		by an upper deviation of the TAP maximum, which requires the existence
		of a TAP state at the corresponding level.
	\end{remark}
	
	\begin{remark}
		The proof also gives a direct lower bound for upper deviations of the
		ordinary free energy. More precisely, for every regular
		\(f\in(f_{\eq},f_1)\) satisfying the strict Plefka condition, we show (see Proposition~\ref{prop:free-energy-ldp-lower-bound}) that
		\[
		\lim_{\rho\downarrow0}\liminf_{N\to\infty}
		\frac1N\log\dP\left(
		\frac1N\log Z_N\ge f-\rho
		\right)
		\ge\Sigma(f).
		\]
		
		The estimates developed here provide ingredients for a constructive proof
		of the Parisi formula for the Ising model following Huang and
		Sellke~\cite{HuangSellke}. Suppose that the
		Parisi contact set (for $f=f_\eq$) is
		\[
		0=q_0<q_1<\cdots<q_D.
		\]
		One can show that for every $i\in \{1,\ldots,D\}$, there exists a TAP state $\bfm^{(i)}$ of self-overlap $q_i$ and free energy level $\approx f_\eq$. This can be proved by induction by studying the maximum of $F_\TAP$ over a band orthogonal to $\bfm^{(i-1)}$, conditionally on the skeleton of ancestors being critical at level $f_\eq$.

		The planted truncated second-moment argument of Huang and
		Sellke~\cite[Proposition~2.7]{HuangSellke} applies almost verbatim to the
		strictly replica-symmetric band above \(q_D\). Combining the two steps
		would give a constructive proof of the lower bound in the Parisi formula
		for the Ising model.
		
		Notice that if the contact set has a continuous part, then a finite iteration through
		contact levels no longer suffices. Subag developed the corresponding
		Hessian-ascent mechanism for spherical models with full replica symmetry
		breaking (FRSB)
		\cite[Theorem~4]{SubagHessianAscent}. Jekel, Sandhu, and Shi established a
		cube analogue for the SK model under an FRSB assumption
		\cite[Theorem~1.5]{JekelSandhuShi}.
	\end{remark}

	\subsection{Some literature}
	
	In 1977, Thouless, Anderson, and Palmer introduced a random free-energy
	functional of the local magnetizations \(m_i=\langle\sigma_i\rangle\), meant
	to govern the thermodynamic behavior of the magnetization vector
	\cite{ThoulessAndersonPalmer1977}. This is the classical TAP free energy,
	named after Thouless, Anderson, and Palmer.
	
	A rigorous approach was developed by Subag for spherical spin glasses and by
	Chen, Panchenko, and Subag for Ising spin glasses
	\cite{Subag2018,ChenPanchenkoSubag,ChenPanchenkoSubag2}. The TAP free energy
	at \(\bfm\) is obtained by constraining many replicas to have barycenter
	\(\bfm\) and computing the corresponding free-energy cost. Thus, after
	normalization, \(f_{\eq}-N^{-1}F_{\TAP}(\bfm)\) is the quenched
	large-deviation rate for the barycenter.
	
	A line of rigorous work studies landscape complexity in regimes
	where critical points are exponentially numerous. This theory is most
	developed for spherical pure \(p\)-spin Hamiltonians. Auffinger, Ben Arous,
	and \v{C}ern\'y computed the annealed exponential growth rate of the number
	of critical points at a prescribed energy using Kac--Rice and random matrix
	theory \cite{AuffingerBenArousCerny}. Subag then used a second-moment argument
	to show that the same exponent describes the typical count, and proved
	concentration relative to the mean in a low-energy interval
	\cite{SubagComplexity}. Subag and Zeitouni extended this concentration to
	arbitrary energies for pure models with \(p\ge32\)
	\cite{SubagZeitouniComplexity}. These results concern positive complexity,
	where the number of critical points grows exponentially with \(N\). At the
	extremal scale, Subag and Zeitouni proved that the critical values of the pure
	\(p\)-spin spherical Hamiltonian converge to a Poisson point process with
	exponential intensity \cite{SubagZeitouniExtremal}. After the corresponding
	exponential reweighting and normalization, the limiting weights have law
	\(\mathrm{PD}(\zeta(\{0\}),0)\), where \(\zeta\) is the equilibrium Parisi
	measure. This matches the Chinese restaurant representation of the Parisi
	hierarchy, in which replicas are customers and TAP states are the first-level
	tables. For the Ising TAP
	landscape, Belius, Concetti, and Genovese proved the determinant
	asymptotics in the Bray--Moore complexity formula at exponential scale and
	computed the prefactor arising from a small spectral outlier
	\cite{BeliusConcettiGenovese}.
	
	The supersymmetric calculations summarized in
	Subsection~\ref{sub:controversies} produce a Legendre relation between a
	distinguished saddle of the TAP Kac--Rice action and the thermodynamic free
	energy. The same variational quantity appears independently in Talagrand's
	fractional-moment formula. Talagrand computed the fractional moments of the
	partition function using Guerra's interpolation and the
	Aizenman--Sims--Starr scheme \cite[Section~7]{LDTalagrand}. In the present notation, for
	\(0<\theta<1\), his formula reads
	\[
	\lim_{N\to\infty}\frac1N\log\dE Z_N^\theta
	=
	\theta
	\inf_{\substack{\zeta\in\mc P([0,1])\\
			\zeta(\{0\})\ge\theta}}
	P_\zeta
	=
	\Lambda(\theta).
	\]
	Under the change of measure with density
	\(Z_N^\theta/\dE Z_N^\theta\), Talagrand's Poisson--Dirichlet cascade
	assigns to the clusters at level \(q\) weights with law
	\(\mathrm{PD}(\alpha,\theta)\), in his convention, where
	\(\alpha=\zeta([0,q])>\theta\) is the cumulative Parisi mass
	\cite[Sections~2 and~5]{LDTalagrand}. It would be interesting to prove,
	following Subag and Zeitouni \cite{SubagZeitouniExtremal}, that under this
	tilted law the ranked normalized weights \(e^{F_{\TAP}}\) of extremal TAP
	states at level \(q\) converge to the same distribution.
	
	A constructive proof of the spherical Parisi formula based on the TAP
	landscape was obtained by Huang and Sellke
	\cite[Theorem~1.5 and Corollary~1.6]{HuangSellke}. The centers of Gibbs
	clusters, which form the nodes of the Parisi ultrametric tree, are expected
	to be critical points of the TAP free energy at the equilibrium level.
	Huang and Sellke realize this picture by constructing exponentially
	branching families of near-optimal points at successive contact levels and
	organizing them into an ultrametric tree. Their main new input is a
	truncated second-moment argument for isolated critical points. Above the
	final contact point, where the remaining band is strictly replica
	symmetric, they instead restrict the partition function to configurations
	that are typical under the planted measure and apply the second-moment
	method to this restricted partition function
	\cite[Proposition~2.7]{HuangSellke}. We adapt their
	isolated-critical-point argument to the Ising TAP landscape.
	
	\subsection{Organization of the paper}\label{sub:organization}
	
	Section~\ref{sec:new5} proves upper bounds for the expected TAP maximum and
	its exponential moments, both in the full cube and conditionally in a band.
	The comparison uses an additive Ruelle probability cascade, Fenchel duality,
	and a maximal estimate along the Parisi flow.
	
	Section~\ref{sec:onshell-lb} proves the matching lower bound and completes
	the proof of Theorem~\ref{thm:legendre}. We compute the one-point annealed
	complexity at a regular level and retain isolated SUSY states. The strict
	Parisi obstacle gap allows us to discard, without changing the exponent,
	points with competitors whose overlap lies in $[\delta,q-\delta]$. The
	strict Plefka condition gives a uniform negative Hessian bound near each
	retained state, which excludes distinct competitors with overlap in
	$[q-\delta,q]$. Distinct retained states therefore have pairwise overlap
	less than $\delta$. A one-sided spherical-code bound then converts the
	retained first moment into an existence probability. The section ends by
	deriving the corresponding lower bound for upper deviations of the
	ordinary free energy.
	
	Appendix~\ref{section:Auffinger} collects the identities for the Parisi PDE and
	the Auffinger--Chen process. Appendix~
	\ref{sec:one-point-susy-computations} contains the supersymmetric Kac--Rice
	computation as well as the Hessian estimates used in
	Section~\ref{sec:onshell-lb}.
	Appendix~\ref{sec:spherical-examples} gives spherical examples comparing the ordinary
	annealed and supersymmetric complexities.
	
	\subsection*{Acknowledgments}
	
	I would like to thank Gérard Ben Arous, Alice Guionnet, and
	Jean-Christophe Mourrat for introducing me to this topic and for numerous
	generous and insightful conversations.
	
	\subsection{Statement of AI use}
	Codex was used to assist in writing the manuscript. GPT Pro was used to help with computations, particularly in Appendices B and C.
	
	\subsection{Notation}\label{sub:notation}
	
	We collect the notation used throughout the paper.
	\begin{itemize}\itemsep2pt
		\item \(\dR\), \(\dE\), and \(\dP\) denote the real line, expectation,
		and probability. The Euclidean inner product and norm on \(\dR^N\) are
		\(\langle\cdot,\cdot\rangle\) and \(\|\cdot\|\). The gradient and
		Hessian are \(\nabla\) and \(\nabla^2\). For symmetric matrices \(A,B\),
		\(A\preceq B\) means that \(B-A\) is positive semidefinite.
		
		\item For \(x,y\in\dR^N\), the overlap and self-overlap are
		\[
		R(x,y):=\frac1N\langle x,y\rangle,
		\qquad
		q_x:=R(x,x)=\frac{\|x\|^2}{N}.
		\]
		
		\item \(\mc P([a,b])\) is the space of probability measures on
		\([a,b]\), and \(\delta_q\) is the Dirac mass at \(q\). We write
		\(\Law(Y)\) for the law of \(Y\) and \(W_p\) for the
		\(p\)-Wasserstein distance, where \(p\in\{1,2\}\).
	\end{itemize}
	
	\begin{itemize}\itemsep2pt
		\item We write
		\[
		\vartheta(x):=x\xi'(x)-\xi(x).
		\]
		
		\item For \(\bfa\in[-1,1]^N\) with \(q_\bfa<1\), the band above
		\(\bfa\) is
		\[
		\Band(\bfa)
		:=
		\left\{
		\bfm\in[-1,1]^N:
		\bfm-\bfa\perp\bfa,\
		q_\bfm\in[q_\bfa,1)
		\right\}.
		\]
	\end{itemize}
	
	\begin{itemize}\itemsep2pt
		\item For \(\zeta\in\mc P([0,1])\), its cumulant is
		\begin{equation}\label{eq:order-parameter-cumulants}
			\widehat\zeta(s)
			:=
			\int_s^1\zeta([0,t])\,\dd t.
		\end{equation}
		
		\item For \(\mu\in\mc P([0,1])\) and \(r\in[0,1]\), set
		\begin{equation}\label{eq:ising-A-mu}
			A_\mu(r)
			:=
			\frac N2\int_0^r s\xi''(s)\mu([0,s])\,\dd s.
		\end{equation}
		
		\item The concave Fenchel conjugate of the Parisi potential is
		\[
		\Phi_\zeta^*(q,m)
		:=
		\inf_{x\in\dR}\{\Phi_\zeta(q,x)-mx\},
		\qquad
		h_\zeta(q,m):=-\Phi_\zeta^*(q,m).
		\]
		Thus \(mx+\Phi_\zeta^*(q,m)\le\Phi_\zeta(q,x)\), with equality when
		\(m=\partial_x\Phi_\zeta(q,x)\), and
		\[
		\partial_m h_\zeta(q,\cdot)
		=
		(\partial_x\Phi_\zeta(q,\cdot))^{-1}.
		\]
		
		\item For \(q\in[0,1)\), the collapse of \(\zeta\) to \(q\) is
		\[
		\Shift_q\zeta
		:=
		\zeta|_{(q,1]}+\zeta([0,q])\delta_q.
		\]
		It satisfies \(\widehat{\Shift_q\zeta}(s)=\widehat\zeta(s)\) for
		\(s\ge q\) and \(V_{\Shift_q\zeta}(q,\mu)=V_\zeta(q,\mu)\).
		
		\item For \(\eta\in\mc P([q,1])\) and
		\(\mu\in\mc P([-1,1])\), the band correction is
		\[
		V_\eta(q,\mu)
		:=
		\int\Phi_\eta^*(q,m)\,\dd\mu(m)
		-
		\frac12\int_q^1s\xi''(s)\eta([0,s])\,\dd s.
		\]
		For \(\bfm\in[-1,1]^N\) and
		\(\eta\in\mc P([q_\bfm,1])\), set
		\[
		F_{\TAP,\eta}(\bfm)
		:=
		H_N(\bfm)
		+
		N V_\eta\left(
		q_\bfm,\frac1N\sum_{i=1}^N\delta_{m_i}
		\right).
		\]
		The generalized TAP free energy and the Parisi value are
		\[
		F_{\TAP}(\bfm)
		=
		\inf_{\eta\in\mc P([q_\bfm,1])}F_{\TAP,\eta}(\bfm),
		\qquad
		P_\zeta=V_\zeta(0,\delta_0).
		\]
		For a full order parameter \(\zeta\), we write
		\[
		F_{\TAP,\zeta}(\bfm)
		:=
		F_{\TAP,\Shift_{q_\bfm}\zeta}(\bfm).
		\]
		
		\item The obstacle of \(\zeta\in\mc P([0,1])\), anchored at \(q\), is
		\begin{equation}\label{eq:ward-obstacle}
			H_{\zeta,q}(t)
			:=
			\frac12\int_t^q\xi''(s)
			\bigl(\dE[(\partial_x\Phi_\zeta(s,X_s))^2]-s\bigr)\,\dd s,
			\qquad 0\le t,q\le1,
		\end{equation}
		where \(X\) is the Auffinger--Chen process in
		\eqref{eq:AuffingerChen}. The integral is signed.
	\end{itemize}
	
	\begin{itemize}\itemsep2pt
		\item The event \(\mc E_{N,\ve}(f)\) records the existence of a critical
		point of \(F_{\TAP,\zeta}\) for some
		\(\zeta\in\mc P([0,1])\) such that both
		\(N^{-1}F_{\TAP}\) and \(N^{-1}F_{\TAP,\zeta}\) lie within \(\ve\) of \(f\).
		
		\item The constrained Parisi functional and the complexity are
		\[
		\Lambda(\theta)
		:=
		\theta
		\inf_{\substack{\zeta\in\mc P([0,1])\\
				\zeta(\{0\})\ge\theta}}
		P_\zeta,
		\qquad
		\Sigma(f)
		:=
		\inf_{\theta\ge0}\bigl[\bar\Lambda(\theta)-\theta f\bigr],
		\]
		where \(\bar\Lambda\) is the extension defined in
		Subsection~\ref{sub:main-results}, and
		\[
		f_1:=\bar\Lambda'_-(1).
		\]
	\end{itemize}
	
	%==============================================================================
	\section{Large deviations upper bound}\label{sec:new5}
	%==============================================================================
	
	In this section we prove upper bounds on the expectation and exponential moments of the maximum of the TAP free energy. We also prove the corresponding conditional bounds for the maximum in the orthogonal band above a prescribed ancestor. The proofs are based on a new Guerra-type comparison with an additive Ruelle probability cascade.

	\subsection{The expected maximum}\label{sub:new5-ising}
	
	We begin by bounding the expectation of the maximum of the TAP free energy.
	
	\begin{proposition}\label{prop:ising-exp}
		The expected maximum of the TAP free energy converges to the Parisi free energy:
		\[
		\lim_{N\to\infty}\frac1N\dE\max_{\bfm\in[-1,1]^N}F_{\TAP}(\bfm)
		=\inf_\zeta P_\zeta\,.
		\]
	\end{proposition}
	
	The lower bound in Proposition~\ref{prop:ising-exp} is immediate, since
	\[
	\frac1N\max_\bfm F_{\TAP}(\bfm)
	\ge
	\frac1N F_{\TAP}(0)
	=
	\inf_\zeta P_\zeta.
	\]
	At $\bfm=0$, the identity
	$\Phi_\zeta^*(0,0)=\Phi_\zeta(0,0)$ follows because
	$\Phi_\zeta(0,\cdot)$ is even and convex. The quantity
	$\inf_\zeta P_\zeta$ is the Ising Parisi free energy~\cite{PanchenkoBook}.
	
	\begin{remark}
		For comparison, the upper bound in Proposition~\ref{prop:ising-exp} follows from the generalized TAP representation of Chen, Panchenko, and Subag and the Parisi formula. Fix $\delta>0$ and an integer $k\ge1$. For $\bfm\in[-1,1]^N$, define
		\[
		\Band_{k,\delta}^{\Ising}(\bfm)
		:=
		\left\{
		(\sigma^1,\ldots,\sigma^k)\in(\{-1,1\}^N)^k:
		\begin{array}{l}
			\bigl|R(\sigma^a,\bfm)-q_\bfm\bigr|\le\delta
			\quad\text{for every }a,\\
			\bigl|R(\sigma^a,\sigma^b)-q_\bfm\bigr|\le\delta
			\quad\text{for }a\ne b
		\end{array}
		\right\}.
		\]
		Then set
		\begin{equation}\label{eq:replicated-band-free-energy}
			U_N^{k,\delta}(\bfm)
			:=\frac1{Nk}\log
			\left[
			2^{-Nk}
			\sum_{(\sigma^1,\ldots,\sigma^k)
				\in\Band_{k,\delta}^{\Ising}(\bfm)}
			\exp\left\{\sum_{a=1}^k
			\bigl(H_N(\sigma^a)-H_N(\bfm)\bigr)\right\}
			\right].
		\end{equation}
		Restricting the replicated partition
		function to this band gives
		\[
		\frac1N\log Z_N
		\ge \frac1N H_N(\bfm)+U_N^{k,\delta}(\bfm)
		\]
		for every $\bfm\in[-1,1]^N$. The expectation and concentration
		estimates in~\cite[Theorem~1]{ChenPanchenkoSubag} give
		\[
		\lim_{\delta\downarrow0}\lim_{k\to\infty}
		\limsup_{N\to\infty}
		\dE\sup_{\bfm\in[-1,1]^N}
		\left|
		U_N^{k,\delta}(\bfm)
		-\frac1N\bigl(F_{\TAP}(\bfm)-H_N(\bfm)\bigr)
		\right|=0.
		\]
		Taking the maximum over $\bfm$, followed by expectation and the
		successive limits, yields
		\[
		\limsup_{N\to\infty}\frac1N
		\dE\max_{\bfm\in[-1,1]^N}F_{\TAP}(\bfm)
		\le
		\lim_{N\to\infty}\frac1N\dE\log Z_N
		=\inf_\zeta P_\zeta
		\]
		by the Parisi formula~\cite{PanchenkoBook}.
	\end{remark}
	
	We next give a direct proof of Proposition~\ref{prop:ising-exp}.
	Fix a finite-step order parameter $\zeta$ satisfying
	$\zeta(\{0\})>0$. Choose radial levels
	\[
	0=q_0<q_1<\cdots<q_L=1,
	\]
	and write
	\[
	\zeta([0,s])=\theta_\ell,
	\qquad q_{\ell-1}<s<q_\ell,
	\]
	where $0<\theta_1<\cdots<\theta_L\le1$. The additive construction below is the logarithmic form of the
	multiplicative RPC.
	
	\begin{definition}[Additive Ruelle probability cascade (RPC)]\label{def:ising-rpc}
		Consider a tree with $L$ levels, rooted at $\varnothing$, in which every vertex before the last level has
		countably many children. We label the children of each vertex by $\dN$, so a vertex at level $\ell$ is a sequence
		$(n_1,\ldots,n_\ell)\in\dN^\ell$.
		
		For every $1\le\ell\le L$, attach to each vertex at level $\ell-1$ an independent Poisson point process on $\dR$
		with intensity
		\[
		\theta_\ell e^{-\theta_\ell u}\,\dd u.
		\]
		Order its points decreasingly and assign them in this order to the edges leading to its children. All these point
		processes are independent.
		
		For a leaf $\alpha=(n_1,\ldots,n_L)$, set $\alpha|0=\varnothing$ and
		$\alpha|\ell=(n_1,\ldots,n_\ell)$. Let $u_{\alpha|\ell}$ denote the label of the edge from
		$\alpha|_{\ell-1}$ to $\alpha|\ell$, and define
		\[
		u_\alpha:=\sum_{\ell=1}^L u_{\alpha|\ell}.
		\]
		For two leaves $\alpha^1$ and $\alpha^2$, let $\alpha^1\wedge\alpha^2$ be the largest
		$\ell\in\{0,\ldots,L\}$ for which $\alpha^1|\ell=\alpha^2|\ell$, and set
		\[
		s_{\alpha^1\wedge\alpha^2}:=q_{\alpha^1\wedge\alpha^2}.
		\]
	\end{definition}
	
	The cascade obeys the marking identity: for a level of parameter $x$ and i.i.d.\ marks $W_i$ independent of
	the labels,
	\begin{equation}\label{eq:ising-marking}
		\{u_i+W_i\}_{i\ge1}\stackrel d=\{u_i+c\}_{i\ge1}\,,\qquad c=\frac1x\log\dE e^{xW},
	\end{equation}
	whenever $c<\infty$. This identity follows from the mapping theorem
	for Poisson point processes. Under the change of variables
	$v=u+W$, the image intensity is
	\[
	x\int e^{-x(v-w)}\,\Law(W)(\dd w)\,\dd v
	=x e^{-xv}\dE e^{xW}\,\dd v
	=x e^{-x(v-c)}\,\dd v,
	\]
	which is exactly the intensity of the original process translated by $c$.

	\begin{definition}[Multiscale cavity field and Onsager field]\label{def:ising-cavity}
		On the tree of Definition~\ref{def:ising-rpc}, let
		$B_s^i(\alpha)$, $1\le i\le N$, and $Y(s,\alpha)$ be centered
		continuous Gaussian martingales indexed by $s\in[0,1]$ and the
		leaves $\alpha$. The two families are independent of each other,
		of $H_N$, and of the Poisson point processes defining the
		cascade. Their covariances are
		\begin{align*}
			\dE B_s^i(\alpha)B_t^j(\beta)
			&=
			\delta_{ij}\,
			\xi'\bigl(s\wedge t\wedge
			s_{\alpha\wedge\beta}\bigr),\\
			\dE Y(s,\alpha)Y(t,\beta)
			&=
			N\vartheta\bigl(s\wedge t\wedge
			s_{\alpha\wedge\beta}\bigr).
		\end{align*}
		For $\bfm\in[-1,1]^N$, with
		\[
		q_\bfm:=\frac1N\sum_{i=1}^N m_i^2,
		\]
		define the multiscale cavity field by
		\[
		Z(\bfm,\alpha)
		:=
		\sum_{i=1}^N m_i B_{q_\bfm}^i(\alpha).
		\]
		We call $Y(q,\alpha)$ the Onsager field at self-overlap $q$.
	\end{definition}
	
	For two marked points $(\bfm^1,\alpha^1)$ and
	$(\bfm^2,\alpha^2)$, set
	\[
	R_{12}:=\frac1N\langle\bfm^1,\bfm^2\rangle,
	\qquad
	Q_{12}:=
	q_{\bfm^1}\wedge q_{\bfm^2}\wedge
	s_{\alpha^1\wedge\alpha^2}.
	\]
	The covariance formulas in Definition~\ref{def:ising-cavity} give
	\begin{equation}\label{eq:ising-covs}
		\dE Z(\bfm^1,\alpha^1)Z(\bfm^2,\alpha^2)
		=
		NR_{12}\xi'(Q_{12}),
		\qquad
		\dE Y(q_{\bfm^1},\alpha^1)Y(q_{\bfm^2},\alpha^2)
		=
		N\vartheta(Q_{12}).
	\end{equation}
	\begin{remark}\label{rem:ising-invariance}
		For every deterministic $q\in[0,1]$,
		\[
		\{u_\alpha+Y(q,\alpha)-A_\zeta(q)\}_\alpha
		\stackrel d=
		\{u_\alpha\}_\alpha,
		\]
		and
		\[
		\left\{
		u_\alpha
		+
		\sum_{i=1}^N\Phi_\zeta(q,B_q^i(\alpha))
		-
		N\Phi_\zeta(0,0)
		\right\}_\alpha
		\stackrel d=
		\{u_\alpha\}_\alpha.
		\]
		
		For the first identity, consider a level $\ell$ such that $q_{\ell-1}<q$. The increment
		\[
		Y(q_\ell\wedge q,\alpha)-Y(q_{\ell-1},\alpha)
		\]
		is a centered Gaussian mark with variance
		\[
		N\int_{q_{\ell-1}}^{q_\ell\wedge q}s\xi''(s)\,\dd s.
		\]
		By the marking identity~\eqref{eq:ising-marking}, adding this increment translates the Poisson points at level
		$\ell$ by
		\[
		\frac{\theta_\ell N}{2}
		\int_{q_{\ell-1}}^{q_\ell\wedge q}s\xi''(s)\,\dd s.
		\]
		Applying the marking identity successively over the levels produces the total translation
		\[
		\frac N2\int_0^q s\xi''(s)\zeta([0,s])\,\dd s=A_\zeta(q),
		\]
		which proves the first identity.
		
		For the second identity, set
		\[
		S_s(\alpha):=\sum_{i=1}^N\Phi_\zeta(s,B_s^i(\alpha)).
		\]
		On $[q_{\ell-1},q_\ell]$, the coefficient in the Parisi equation is $\theta_\ell$. It\^o's formula gives
		\[
		\dd S_s
		=
		\dd M_s-\frac{\theta_\ell}{2}\,\dd\langle M\rangle_s,
		\qquad
		\dd M_s:=
		\sum_{i=1}^N
		\partial_x\Phi_\zeta(s,B_s^i(\alpha))\,\dd B_s^i(\alpha).
		\]
		Since $|\partial_x\Phi_\zeta|\le1$, the process
		\[
		\exp\left\{\theta_\ell\bigl(S_s-S_{q_{\ell-1}}\bigr)\right\},
		\qquad q_{\ell-1}\le s\le q_\ell,
		\]
		is a true martingale. Conditionally on the fields up to $q_{\ell-1}$, the increment of $S$ therefore has zero
		shift in the marking identity. Applying the identity successively over the levels gives
		\[
		\{u_\alpha+S_q(\alpha)-S_0\}_\alpha
		\stackrel d=
		\{u_\alpha\}_\alpha.
		\]
		Since $S_0=N\Phi_\zeta(0,0)$, this proves the second identity.
	\end{remark}
	
	By~\eqref{eq:parisi-value} and~\eqref{eq:ising-A-mu}, we can rewrite $F_{\TAP,\zeta}$ as
	\begin{equation}\label{eq:ising-reduced}
		F_{\TAP,\zeta}(\bfm)-NP_\zeta=T_N(\bfm)\,,\qquad
		T_N(\bfm):=H_N(\bfm)+\sum_{i=1}^N\Phi_\zeta^*(q_\bfm,m_i)-N\Phi_\zeta(0,0)+A_\zeta(q_\bfm)\,,
	\end{equation}
	so the goal is to show that $\limsup_N\tfrac1N\dE\max_\bfm T_N(\bfm)\le0$. 
	
	We compare the Hamiltonian, augmented by the Onsager field $Y$, to the cavity field $Z$. Define
	\[
	\Psi(\bfm,\alpha):=u_\alpha+\sum_{i=1}^N\Phi_\zeta^*(q_\bfm,m_i)-N\Phi_\zeta(0,0)\,,
	\]
	and put 
	\begin{equation}\label{def:X}
		X(\bfm,\alpha):=H_N(\bfm)+Y(q_\bfm,\alpha)+\Psi(\bfm,\alpha)
	\end{equation}
	and 
	\begin{equation}\label{def:tildeX}\widetilde X(\bfm,\alpha):=Z(\bfm,\alpha)+\Psi(\bfm,\alpha).
	\end{equation}
	
	\begin{lemma}[Slepian comparison]\label{lem:ising-slepian}
		The processes defined in~\eqref{def:X} and~\eqref{def:tildeX} satisfy
		\[
		\dE
		\sup_{\substack{\bfm\in[-1,1]^N\\ \alpha\in\dN^L}}
		X(\bfm,\alpha)
		\le
		\dE
		\sup_{\substack{\bfm\in[-1,1]^N\\ \alpha\in\dN^L}}
		\widetilde X(\bfm,\alpha).
		\]
	\end{lemma}
	
	\begin{proof}
		Condition on the cascade labels $(u_\alpha)_\alpha$. The processes
		$X$ and $\widetilde X$ are then Gaussian with the common
		deterministic mean $\Psi$ and centered parts
		\[
		G(\bfm,\alpha)
		:=
		H_N(\bfm)+Y(q_\bfm,\alpha),
		\qquad
		\widetilde G(\bfm,\alpha)
		:=
		Z(\bfm,\alpha).
		\]
		
		For every marked point $(\bfm,\alpha)$, we have
		$Q_{11}=R_{11}=q_\bfm$. Therefore,
		\[
		\begin{aligned}
			\dE G(\bfm,\alpha)^2
			&=
			N\xi(q_\bfm)+N\vartheta(q_\bfm) \\
			&=
			Nq_\bfm\xi'(q_\bfm)
			=
			\dE\widetilde G(\bfm,\alpha)^2.
		\end{aligned}
		\]
		For two marked points,~\eqref{eq:ising-covs} gives
		\[
		\begin{aligned}
			\dE G(\bfm^1,\alpha^1)G(\bfm^2,\alpha^2)
			&-
			\dE\widetilde G(\bfm^1,\alpha^1)
			\widetilde G(\bfm^2,\alpha^2) \\
			&=
			N\left[
			\xi(R_{12})
			+
			\vartheta(Q_{12})
			-
			R_{12}\xi'(Q_{12})
			\right] \\
			&=
			N\left[
			\xi(R_{12})
			-
			\xi(Q_{12})
			-
			\xi'(Q_{12})(R_{12}-Q_{12})
			\right]
			\ge 0,
		\end{aligned}
		\]
		where the inequality follows from the convexity of $\xi$ in
		Assumption~\ref{ass:xi}.
		
		Slepian's inequality now gives the desired comparison on every
		finite subset of $[-1,1]^N\times\dN^L$. By separability and
		monotone convergence, the comparison extends to the full
		suprema. Taking expectation over the cascade labels completes the
		proof.
	\end{proof}
	
	By Lemma~\ref{lem:ising-slepian} it remains to bound the cascade side $\dE\sup_{\bfm,\alpha}\widetilde X$. Recall 
	\begin{equation*}
		\widetilde X(\bfm,\alpha)=u_\alpha+\sum_{i=1}^N m_i B_{q_\bfm}^i(\alpha)+\sum_{i=1}^N \Phi_\zeta^*(q_\bfm,m_i)-N\Phi_\zeta(0,0). 
	\end{equation*}
	Using the Fenchel inequality $m_i B^i_{q_\bfm}(\alpha)+\Phi_\zeta^*(q_\bfm,m_i)\le\Phi_\zeta(q_\bfm,B^i_{q_\bfm}(\alpha))$, we get
	\begin{equation}\label{eq:ising-fenchel}
		\widetilde X(\bfm,\alpha)\le u_\alpha+\sum_{i=1}^N\Phi_\zeta(q_\bfm,B^i_{q_\bfm}(\alpha))-N\Phi_\zeta(0,0)\,.
	\end{equation}
	
	The following lemma bounds the supremum of the right-hand side
	of~\eqref{eq:ising-fenchel}.

	\begin{lemma}\label{lem:parisi-maximal-estimate}
		For the fixed finite-step order parameter $\zeta$, as
		$N\to\infty$,
		\begin{equation}\label{eq:star1}
			\dE
			\sup_{\substack{q\in[0,1]\\ \alpha\in\dN^L}}
			\left\{
			u_\alpha
			+
			\sum_{i=1}^N
			\Phi_\zeta(q,B_q^i(\alpha))
			-
			N\Phi_\zeta(0,0)
			\right\}
			\le
			\dE\max_{\alpha\in\dN^L}u_\alpha
			+
			o(N).
		\end{equation}
	\end{lemma}

	\begin{proof}
		Set
		\[
		S_s(\alpha):=
		\sum_{i=1}^N\Phi_\zeta(s,B_s^i(\alpha)).
		\]
		We first consider $L=1$. In this case $q_1=1$ and
		\[
		\zeta([0,s])=\theta_1=:\theta,
		\qquad 0<s<1.
		\]
		Define
		\[
		W(\alpha):=
		\sup_{0\le s\le1}
		\bigl(S_s(\alpha)-S_0\bigr).
		\]
		The variables $W(\alpha)$ are independent and identically
		distributed over $\alpha\in\dN$ and independent of the cascade
		labels. Therefore,
		\[
		\sup_{\substack{0\le s\le1\\ \alpha\in\dN}}
		\left\{
		u_\alpha+S_s(\alpha)-S_0
		\right\}
		=
		\max_{\alpha\in\dN}
		\left\{
		u_\alpha+W(\alpha)
		\right\}.
		\]
		The marking identity gives
		\[
		\dE
		\max_{\alpha\in\dN}
		\left\{
		u_\alpha+W(\alpha)
		\right\}
		=
		\dE\max_{\alpha\in\dN}u_\alpha+c_N,
		\qquad
		c_N:=
		\frac1\theta\log\dE e^{\theta W},
		\]
		where $W$ has the common law of the variables $W(\alpha)$.
		
		On $[0,1]$, It\^o's formula and the Parisi equation give
		\[
		S_s(\alpha)-S_0
		=
		M_s(\alpha)
		-\frac \theta2\langle M(\alpha)\rangle_s,
		\]
		where
		\[
		M_s(\alpha)
		:=
		\sum_{i=1}^N
		\int_0^s
		\partial_x\Phi_\zeta(r,B_r^i(\alpha))
		\,\dd B_r^i(\alpha).
		\]
		Since $|\partial_x\Phi_\zeta|\le1$,
		\[
		\langle M(\alpha)\rangle_s
		=
		\int_0^s
		\xi''(r)
		\sum_{i=1}^N
		\bigl(
		\partial_x\Phi_\zeta(r,B_r^i(\alpha))
		\bigr)^2
		\,\dd r
		\le
		N\bigl(\xi'(s)-\xi'(0)\bigr).
		\]
		Consequently,
		\[
		\mathcal M_s(\alpha)
		:=
		\exp\left\{
		\theta\bigl(S_s(\alpha)-S_0\bigr)
		\right\}
		\]
		is a true positive martingale, and
		\[
		e^{\theta W(\alpha)}
		=
		\sup_{0\le s\le1}\mathcal M_s(\alpha).
		\]
		
		Set $p=1+N^{-1}$. Doob's inequality gives
		\[
		\dE e^{\theta W}
		\le
		\frac p{p-1}
		\left(
		\dE\mathcal M_1^p
		\right)^{1/p}.
		\]
		Moreover,
		\[
		\begin{aligned}
			\mathcal M_1^p
			&=
			\exp\left\{
			p\theta M_1
			-
			\frac{p^2\theta^2}{2}
			\langle M\rangle_1
			\right\}
			\\
			&\qquad\times
			\exp\left\{
			\frac{(p^2-p)\theta^2}{2}
			\langle M\rangle_1
			\right\}.
		\end{aligned}
		\]
		The first factor has expectation one, while the bracket bound gives
		\[
		\dE\mathcal M_1^p
		\le
		\exp\left\{
		\frac{(p^2-p)\theta^2}{2}
		N\bigl(\xi'(1)-\xi'(0)\bigr)
		\right\}.
		\]
		Since $p/(p-1)=N+1$ and $p^2-p=p/N$, we obtain
		\[
		c_N
		\le
		\frac{\log(N+1)}{\theta}
		+
		\frac{\theta}{2}
		\bigl(\xi'(1)-\xi'(0)\bigr)
		=
		o(N).
		\]
		It follows that
		\[
		\dE
		\sup_{\substack{0\le s\le1\\ \alpha\in\dN}}
		\left\{
		u_\alpha+S_s(\alpha)-S_0
		\right\}
		\le
		\dE\max_{\alpha\in\dN}u_\alpha+o(N).
		\]
		
		For a general finite-step order parameter and each
		$1\le\ell\le L$, define
		\[
		\Delta_\ell(\alpha)
		:=
		\sup_{q_{\ell-1}\le s\le q_\ell}
		\bigl(
		S_s(\alpha)-S_{q_{\ell-1}}(\alpha)
		\bigr).
		\]
		If $q\in[q_{k-1},q_k]$, then
		\begin{equation}\label{eq:splitting}
			\begin{aligned}
				S_q(\alpha)-S_0
				&=
				\sum_{\ell<k}
				\bigl(
				S_{q_\ell}(\alpha)-S_{q_{\ell-1}}(\alpha)
				\bigr)
				+
				S_q(\alpha)-S_{q_{k-1}}(\alpha)
				\\
				&\le
				\sum_{\ell=1}^L \Delta_\ell(\alpha).
			\end{aligned}
		\end{equation}
		Consequently,
		\[
		\sup_{\substack{q\in[0,1]\\ \alpha\in\dN^L}}
		\left\{
		u_\alpha+S_q(\alpha)-S_0
		\right\}
		\le
		\max_{\alpha\in\dN^L}
		\left\{
		u_\alpha+\sum_{\ell=1}^L \Delta_\ell(\alpha)
		\right\}.
		\]
		
		The preceding one-level calculation applies conditionally on
		every interval $[q_{\ell-1},q_\ell]$, uniformly in the Gaussian
		field at the parent vertex. It bounds the corresponding marking
		shift by
		\[
		d_\ell(N)
		:=
		\frac{\log(N+1)}{\theta_\ell}
		+
		\frac{\theta_\ell}{2}
		\bigl(
		\xi'(q_\ell)-\xi'(q_{\ell-1})
		\bigr).
		\]
		
		We spell out the recursion for $L=2$. For
		$\alpha=(n_1,n_2)$, write
		$u_\alpha=u_{(n_1)}+u_{(n_1,n_2)\mid2}$. The paths with the
		same first coordinate agree up to time $q_1$, so
		$\Delta_1(\alpha)$ depends only on $n_1$. We denote it by
		$\Delta_1(n_1)$. Therefore,
		\[
		\begin{aligned}
			&\max_{\alpha}
			\left\{
			u_\alpha+\Delta_1(\alpha)+\Delta_2(\alpha)
			\right\}
			\\
			&\qquad=
			\max_{n_1}
			\left\{
			u_{(n_1)}+\Delta_1(n_1)
			+
			\max_{n_2}
			\left(
			u_{(n_1,n_2)\mid2}+\Delta_2(n_1,n_2)
			\right)
			\right\}.
		\end{aligned}
		\]
		
		Let $\mathcal G_1$ be the $\sigma$-field generated by the
		first-level cascade labels and the Gaussian paths up to time
		$q_1$. Conditionally on $\mathcal G_1$, for each fixed $n_1$,
		the variables $\Delta_2(n_1,n_2)$ are independent and identically
		distributed over $n_2$ and independent of the second-level
		labels. The collections corresponding to distinct values of
		$n_1$ are conditionally independent. The conditional one-level
		estimate gives
		\[
		c_2(n_1)
		:=
		\frac1{\theta_2}
		\log
		\dE\left[
		e^{\theta_2\Delta_2(n_1,n_2)}
		\,\middle|\,
		\mathcal G_1
		\right]
		\le d_2(N).
		\]
		Set $V(n_1):=\max_{n_2}u_{(n_1,n_2)\mid2}$. Applying the
		marking identity conditionally at every first-level vertex gives,
		jointly in $n_1$,
		\[
		\max_{n_2}
		\left\{
		u_{(n_1,n_2)\mid2}+\Delta_2(n_1,n_2)
		\right\}
		\stackrel d=
		c_2(n_1)+V(n_1).
		\]
		Since $c_2(n_1)\le d_2(N)$, it follows that
		\[
		\begin{aligned}
			&\dE\max_{\alpha}
			\left\{
			u_\alpha+\Delta_1(\alpha)+\Delta_2(\alpha)
			\right\}
			\\
			&\qquad\le
			d_2(N)
			+
			\dE\max_{n_1}
			\left\{
			u_{(n_1)}+\Delta_1(n_1)+V(n_1)
			\right\}.
		\end{aligned}
		\]
		
		Each of the families $\{\Delta_1(n_1)\}_{n_1\ge1}$ and
		$\{V(n_1)\}_{n_1\ge1}$ consists of independent and identically
		distributed variables. The two families are mutually independent
		and both are independent of the first-level labels. Moreover,
		$\dE e^{\theta_1V(n_1)}<\infty$ because
		$\theta_1<\theta_2$. Writing
		$c_1:=\theta_1^{-1}\log\dE e^{\theta_1\Delta_1(n_1)}
		\le d_1(N)$, the marking identity at the first level gives
		\[
		\begin{aligned}
			\dE\max_{n_1}
			\left\{
			u_{(n_1)}+\Delta_1(n_1)+V(n_1)
			\right\}
			&=
			c_1+
			\dE\max_{n_1}
			\left\{
			u_{(n_1)}+V(n_1)
			\right\}
			\\
			&\le
			d_1(N)+\dE\max_\alpha u_\alpha.
		\end{aligned}
		\]
		Combining the last two estimates yields
		\[
		\dE\max_{\alpha}
		\left\{
		u_\alpha+\Delta_1(\alpha)+\Delta_2(\alpha)
		\right\}
		\le
		\dE\max_\alpha u_\alpha+d_1(N)+d_2(N).
		\]
		
		For general $L$, the same conditional argument is applied from
		level $L$ down to level $1$. At level $\ell$, the maximum of the
		remaining subtree is independent of $\Delta_\ell$, and the strict
		ordering of the cascade parameters guarantees the exponential
		moments required by the marking identity. We obtain
		\[
		\dE
		\max_{\alpha\in\dN^L}
		\left\{
		u_\alpha+\sum_{\ell=1}^L \Delta_\ell(\alpha)
		\right\}
		\le
		\dE\max_{\alpha\in\dN^L}u_\alpha
		+
		\sum_{\ell=1}^L d_\ell(N).
		\]
		Since $L$ and $\zeta$ are fixed,
		$\sum_{\ell=1}^L d_\ell(N)=O(\log N)=o(N)$. Together with~\eqref{eq:splitting}
		and $S_0=N\Phi_\zeta(0,0)$, this
		proves~\eqref{eq:star1}.
	\end{proof}

	\begin{lemma}\label{lem:ising-lower}
		Let $X(\bfm,\alpha)$ and $T_N(\bfm)$ be defined by~\eqref{def:X}
		and~\eqref{eq:ising-reduced}. Then
		\[
		\dE
		\sup_{\substack{\bfm\in[-1,1]^N\\ \alpha\in\dN^L}}
		X(\bfm,\alpha)
		\ge
		\dE\max_{\bfm\in[-1,1]^N}T_N(\bfm)
		+
		\dE\max_{\alpha\in\dN^L}u_\alpha.
		\]
	\end{lemma}
	
	\begin{proof}
		By~\eqref{def:X} and~\eqref{eq:ising-reduced},
		\[
		X(\bfm,\alpha)
		=
		T_N(\bfm)+u_\alpha+Y(q_\bfm,\alpha)-A_\zeta(q_\bfm).
		\]
		Let $\bfm_\star=\bfm_\star(H_N)$ be a measurable maximizer of
		$T_N$ on $[-1,1]^N$, and set $q_\star=q_{\bfm_\star}$. Since
		$T_N$ depends on the randomness only through $H_N$, the self-overlap
		$q_\star$ is independent of $(u,Y)$. Therefore,
		\[
		\sup_{\bfm,\alpha}X(\bfm,\alpha)
		\ge
		\max_{\bfm\in[-1,1]^N}T_N(\bfm)
		+
		\max_{\alpha\in\dN^L}
		\{u_\alpha+Y(q_\star,\alpha)-A_\zeta(q_\star)\}.
		\]
		Conditionally on $H_N$, the self-overlap $q_\star$ is deterministic, so
		Remark~\ref{rem:ising-invariance} gives
		\[
		\dE\left[
		\max_{\alpha\in\dN^L}
		\{u_\alpha+Y(q_\star,\alpha)-A_\zeta(q_\star)\}
		\,\middle|\,H_N
		\right]
		=
		\dE\max_{\alpha\in\dN^L}u_\alpha.
		\]
		Taking expectations proves the claim.
	\end{proof}
	
	We conclude the proof of Proposition~\ref{prop:ising-exp}.

	\begin{proof}[Proof of Proposition~\ref{prop:ising-exp}]
		By Lemma~\ref{lem:ising-slepian}, the Fenchel bound~\eqref{eq:ising-fenchel}, and
		Lemma~\ref{lem:parisi-maximal-estimate},
		\[
		\dE\sup_{\bfm,\alpha}X(\bfm,\alpha)
		\le
		\dE\sup_{\bfm,\alpha}\widetilde X(\bfm,\alpha)
		\le
		\dE\max_\alpha u_\alpha+o(N).
		\]
		Combining this with Lemma~\ref{lem:ising-lower} yields
		\[
		\limsup_{N\to\infty}
		\frac1N\dE\max_{\bfm\in[-1,1]^N}T_N(\bfm)
		\le 0.
		\]
		The preceding bound holds for every finite-step order parameter
		$\zeta$ satisfying $\zeta(\{0\})>0$. Indeed,
		$F_{\TAP}\le F_{\TAP,\zeta}$ and~\eqref{eq:ising-reduced} give
		\[
		\limsup_{N\to\infty}
		\frac1N\dE\max_{\bfm\in[-1,1]^N}F_{\TAP}(\bfm)
		\le P_\zeta.
		\]
		
		Now let $\zeta$ be an arbitrary finite-step order parameter. For
		$\varepsilon\in(0,1)$, set
		\[
		\zeta_\varepsilon
		:=
		(1-\varepsilon)\zeta+\varepsilon\delta_0.
		\]
		Since $\zeta_\varepsilon(\{0\})>0$, the preceding bound gives
		\[
		\limsup_{N\to\infty}
		\frac1N\dE\max_{\bfm\in[-1,1]^N}F_{\TAP}(\bfm)
		\le P_{\zeta_\varepsilon}.
		\]
		Letting $\varepsilon\downarrow0$ and using continuity of the Parisi
		functional gives the same bound with $P_\zeta$ on the right-hand
		side. Taking the infimum over finite-step order parameters and using
		their density proves
		\[
		\limsup_{N\to\infty}
		\frac1N\dE\max_{\bfm\in[-1,1]^N}F_{\TAP}(\bfm)
		\le\inf_\zeta P_\zeta.
		\qedhere
		\]
	\end{proof}

	\subsection{The exponential moments}
	We extend the comparison of Subsection~\ref{sub:new5-ising} from the expected maximum to its fractional
	exponential moments. The only new ingredient is an outer Poisson level of parameter $\theta$. This one-level Poisson--Dirichlet device is the same mechanism employed by Chen, Guionnet, Ko,
	Lacroix-A-Chez-Toine, and Mourrat~\cite{ChenGuionnetKoLacroixMourrat} to compute $\dE[Z_N^\theta]$.
	
	\begin{proposition}\label{prop:direct-tilt-TAP}
		For every $\theta\in(0,1)$,
		\[
		\limsup_{N\to\infty}
		\frac1N
		\log\dE\exp\left\{
		\theta\max_{\bfm\in[-1,1]^N}F_{\TAP}(\bfm)
		\right\}
		\le
		\Lambda(\theta)
		=
		\theta
		\inf_{\substack{\zeta\in\mc P([0,1])\\
				\zeta(\{0\})\ge\theta}}
		P_\zeta.
		\]
	\end{proposition}
	
	\begin{proof}
		Fix $\theta\in(0,1)$. We first consider a finite-step order
		parameter $\zeta$ satisfying $\zeta(\{0\})>\theta$. Write
		$x_1<\cdots<x_L$ for its cumulative masses, so that
		$\theta<x_1$.
		
		By~\eqref{eq:ising-reduced},
		$F_{\TAP,\zeta}=NP_\zeta+T_N$, so it remains to prove
		\begin{equation}\label{eq:ising-tilt-reduced}
			\limsup_{N\to\infty}
			\frac1N
			\log\dE\exp\left\{
			\theta\max_{\bfm\in[-1,1]^N}T_N(\bfm)
			\right\}
			\le0.
		\end{equation}
		
		\medskip\noindent\emph{Step 1 (marking identity).}\;
		Let $(w_b)_{b\ge1}$ be a Poisson point process on $\dR$ with
		intensity $\theta e^{-\theta w}\,\dd w$. If $(V_b)_b$ are
		independent and identically distributed marks, independent of
		$(w_b)_b$, then~\eqref{eq:ising-marking} gives
		\begin{equation}\label{eq:outer-theta-marking}
			\dE\max_b\{w_b+V_b\}
			=
			\dE\max_b w_b
			+\frac1\theta\log\dE e^{\theta V},
		\end{equation}
		whenever the exponential moment is finite, where $V$ has the
		common law of the marks.
		
		\medskip\noindent\emph{Step 2 (adding the outer level).}\;
		For every $b$, take an independent copy
		$(H_N^b,u^b,B^b,Y^b)$ of the variables used in the proof of
		Proposition~\ref{prop:ising-exp}, with all the copies independent
		of $(w_b)_b$. Let $T_N^b$ denote $T_N$ with $H_N$ replaced by
		$H_N^b$, let $Z^b$ be the corresponding cavity field, and set
		\[
		U_b:=\max_{\alpha\in\dN^L}u_\alpha^b.
		\]
		The labels $w_b+u_\alpha^b$ form a cascade with parameters
		$\theta<x_1<\cdots<x_L$. The level of parameter $\theta$ has a
		radial interval of length zero. Define
		\[
		\begin{aligned}
			X_b(\bfm,\alpha)
			&:=
			w_b+H_N^b(\bfm)+Y^b(q_\bfm,\alpha)+u_\alpha^b
			+\sum_{i=1}^N\Phi_\zeta^*(q_\bfm,m_i)
			-N\Phi_\zeta(0,0),\\
			\widetilde X_b(\bfm,\alpha)
			&:=
			w_b+Z^b(\bfm,\alpha)+u_\alpha^b
			+\sum_{i=1}^N\Phi_\zeta^*(q_\bfm,m_i)
			-N\Phi_\zeta(0,0).
		\end{aligned}
		\]
		
		\medskip\noindent\emph{Step 3 (Slepian comparison).}\;
		Conditionally on $(w_b,u^b)_b$, the two processes are Gaussian with
		the same mean. Within each branch
		$b$, their covariance comparison is exactly that of
		Lemma~\ref{lem:ising-slepian}. For distinct branches, both
		covariances vanish. Slepian's inequality therefore gives
		\[
		\dE\sup_{b,\bfm,\alpha}X_b(\bfm,\alpha)
		\le
		\dE\sup_{b,\bfm,\alpha}\widetilde X_b(\bfm,\alpha).
		\]
		
		\medskip\noindent\emph{Step 4 (Fenchel bound).}\;
		The Fenchel bound~\eqref{eq:ising-fenchel} gives
		\[
		\widetilde X_b(\bfm,\alpha)
		\le
		w_b+u_\alpha^b
		+\sum_{i=1}^N
		\Phi_\zeta(q_\bfm,B_{q_\bfm}^{b,i}(\alpha))
		-N\Phi_\zeta(0,0).
		\]
		
		\medskip\noindent\emph{Step 5 (maximal estimate).}\;
		The level-by-level argument in
		Lemma~\ref{lem:parisi-maximal-estimate} applies to the cascade
		from Step~2. The outer level contributes no Gaussian increment,
		while $\theta<x_1$ gives the exponential moment required by the
		marking identity at that level. Combining Steps~3 and~4, we obtain
		\begin{equation}\label{eq:tilt-upper}
			\dE\sup_{b,\bfm,\alpha}X_b(\bfm,\alpha)
			\le
			\dE\max_b\{w_b+U_b\}+o(N).
		\end{equation}
		
		\medskip\noindent\emph{Step 6 (lower bound).}\;
		For each $b$, let $\bfm_b$ be a measurable maximizer of $T_N^b$
		on $[-1,1]^N$ and set $q_b=q_{\bfm_b}$. Then
		\[
		\sup_{\bfm,\alpha}X_b(\bfm,\alpha)
		\ge
		w_b+\max_{\bfm\in[-1,1]^N}T_N^b(\bfm)
		+V_b,
		\]
		where
		\[
		V_b
		:=
		\max_{\alpha\in\dN^L}
		\{u_\alpha^b+Y^b(q_b,\alpha)-A_\zeta(q_b)\}.
		\]
		Conditionally on $(H_N^b)_b$, the radii $(q_b)_b$ are
		deterministic. Remark~\ref{rem:ising-invariance}, applied
		independently in every branch, shows that the variables $(V_b)_b$
		are independent, have the same law as $(U_b)_b$, and are
		independent of $(\max_\bfm T_N^b)_b$. Consequently,
		\begin{equation}\label{eq:tilt-lower}
			\dE\sup_{b,\bfm,\alpha}X_b(\bfm,\alpha)
			\ge
			\dE\max_b
			\left\{
			w_b+\max_{\bfm\in[-1,1]^N}T_N^b(\bfm)+U_b
			\right\}.
		\end{equation}
		
		\medskip\noindent\emph{Step 7 (conclusion and approximation).}\;
		The marking identity and the strict ordering
		$x_1<\cdots<x_L$ show that $U_b$ has exponential moments of every
		order below $x_1$. In particular, $\dE e^{\theta U_b}<\infty$.
		Moreover, Borell--TIS and the boundedness of the deterministic part
		of $T_N$ on the cube give
		\[
		\dE\exp\left\{
		\theta\max_{\bfm\in[-1,1]^N}T_N(\bfm)
		\right\}<\infty.
		\]
		Applying~\eqref{eq:outer-theta-marking} and using the independence
		of $U_b$ and $T_N^b$, we obtain
		\[
		\begin{aligned}
			\dE\max_b\{w_b+U_b\}
			&=
			\dE\max_b w_b
			+\frac1\theta\log\dE e^{\theta U_b},\\
			\dE\max_b
			\left\{
			w_b+\max_\bfm T_N^b(\bfm)+U_b
			\right\}
			&=
			\dE\max_b w_b
			+\frac1\theta
			\log\dE e^{\theta\max_\bfm T_N(\bfm)}
			+\frac1\theta\log\dE e^{\theta U_b}.
		\end{aligned}
		\]
		Comparing~\eqref{eq:tilt-upper} and~\eqref{eq:tilt-lower} now
		proves~\eqref{eq:ising-tilt-reduced}. Since
		$F_{\TAP}\le F_{\TAP,\zeta}$,
		\[
		\limsup_{N\to\infty}
		\frac1N
		\log\dE\exp\left\{
		\theta\max_{\bfm\in[-1,1]^N}F_{\TAP}(\bfm)
		\right\}
		\le\theta P_\zeta.
		\]
		
		Finally, let $\zeta\in\mc P([0,1])$ satisfy
		$\zeta(\{0\})\ge\theta$ and set
		\[
		\zeta_\varepsilon
		:=
		(1-\varepsilon)\zeta+\varepsilon\delta_0.
		\]
		Then $\zeta_\varepsilon(\{0\})>\theta$. Approximate
		$\zeta_\varepsilon$ by finite-step order parameters that retain
		their mass at zero. Continuity of the Parisi functional, followed
		by $\varepsilon\downarrow0$, gives the same bound with $P_\zeta$
		on the right-hand side. Taking the infimum over such $\zeta$ proves
		the proposition.
	\end{proof}
	
	Markov's inequality gives the corresponding bound on the upper tail of the maximum. For every level
	$f$,
	\[
	\limsup_{N\to\infty}\frac1N\log\dP\Bigl(\tfrac1N\max_{\bfm\in[-1,1]^N}F_{\TAP}(\bfm)\ge f\Bigr)
	\le\inf_{\theta\in(0,1)}\Bigl[\theta\inf_{\zeta:\,\zeta(\{0\})\ge\theta}P_\zeta-\theta f\Bigr]\,.
	\]
	
	%------------------------------------------------------------------------------
	\subsection{The expected maximum in a band}
	
	The results of this subsection and the next are not needed for the proof of
	Theorem~\ref{thm:legendre}. They show how the same argument applies in a
	band above a TAP ancestor.
	
	We prove the conditional counterpart of
	Proposition~\ref{prop:ising-exp} in the band above a prescribed
	ancestor. Fix $q\in[0,1)$ and a finite-step order parameter
	$\eta\in\mc P([q,1])$. For $\bfa\in(-1,1)^N$ with $q_\bfa=q$,
	assume that $\eta$ attains the TAP correction at $\bfa$, so that
	\[
	F_{\TAP}(\bfa)=F_{\TAP,\eta}(\bfa).
	\]
	For $f\in\dR$ and $g\in\bfa^\perp$, set
	\begin{equation}\label{eq:cond-band-conditioning}
		\mc J_\bfa(f,g)
		:=
		\bigl\{
		F_{\TAP,\eta}(\bfa)=Nf,
		\ P_{\bfa^\perp}\nabla F_{\TAP,\eta}(\bfa)=g
		\bigr\}.
	\end{equation}
	When $q>0$, this conditioning is defined for every $f\in\dR$ and
	$g\in\bfa^\perp$. When $q=0$, we take $\bfa=0$, $f=P_\eta$, and
	$g=0$, in which case the two quantities in
	\eqref{eq:cond-band-conditioning} are deterministic.
	
	Since $a_i\in(-1,1)$, there is a unique $b_i\in\dR$ such that
	\[
	a_i=\partial_x\Phi_\eta(q,b_i),
	\qquad\text{equivalently}\qquad
	b_i=\partial_m h_\eta(q,a_i).
	\]
	In particular,
	\[
	a_i b_i+\Phi_\eta^*(q,a_i)=\Phi_\eta(q,b_i).
	\]
	
	The descendants considered below lie in the band
	$\Band(\bfa)$ defined in Subsection~\ref{sub:notation}. Since
	$q_\bfa=q$, in the present setting
	\[
	\Band(\bfa)
	=
	\left\{
	\bfm\in[-1,1]^N:
	\bfm-\bfa\perp\bfa,
	\ q_\bfm\in[q,1)
	\right\}.
	\]
	Writing $\bfm=\bfa+\tau$, one has
	$q_\bfm=q+N^{-1}\|\tau\|^2$. For $\ve>0$, set
	\[
	\Band_{\le1-\ve}(\bfa)
	:=
	\{\bfm\in\Band(\bfa):q_\bfm\le1-\ve\}.
	\]
	
	\begin{definition}
		For $g\in\bfa^\perp$, define
		\begin{equation}\label{eq:conditional-band-D}
			\mathscr D_\eta(\bfa,g)
			:=
			\inf_{h:\,P_{\bfa^\perp}h=g}
			\frac1N\sum_{i=1}^N
			\left[
			\Phi_\eta(q,b_i+h_i)
			-\Phi_\eta(q,b_i)
			-a_i h_i
			\right].
		\end{equation}
	\end{definition}
	
	The summands in~\eqref{eq:conditional-band-D} are nonnegative by
	convexity, and therefore
	\[
	\mathscr D_\eta(\bfa,0)=0.
	\]
	
	\begin{proposition}\label{prop:cond-band-noncrit}
		Fix $q\in[0,1)$ and a finite-step order parameter
		$\eta\in\mc P([q,1])$. For each \(N\), let
		\(\bfa_N\in(-1,1)^N\) satisfy
		\[
		q_{\bfa_N}=q,
		\qquad
		F_{\TAP}(\bfa_N)=F_{\TAP,\eta}(\bfa_N).
		\]
		Let \(f_N\in\dR\) and \(g_N\in\bfa_N^\perp\). When \(q=0\), take
		\(\bfa_N=0\), \(f_N=P_\eta\), and \(g_N=0\). Condition on
		\[
		F_{\TAP,\eta}(\bfa_N)=Nf_N,
		\qquad
		P_{\bfa_N^\perp}\nabla F_{\TAP,\eta}(\bfa_N)=g_N.
		\]
		Then
		\[
		\limsup_{\ve\downarrow0}
		\limsup_{N\to\infty}
		\left[
		\frac1N
		\dE\left[
		\sup_{\bfm\in\Band_{\le1-\ve}(\bfa_N)}
		\bigl(
		F_{\TAP}(\bfm)-F_{\TAP}(\bfa_N)
		\bigr)
		\,\middle|\,
		\mc J_{\bfa_N}(f_N,g_N)
		\right]
		-
		\mathscr D_\eta(\bfa_N,g_N)
		\right]
		\le0.
		\]
	\end{proposition}
	
	In particular, if \(g_N=0\) for every \(N\), then
	\(\mathscr D_\eta(\bfa_N,g_N)=0\). Since \(\bfa_N\) belongs to the
	band, the normalized conditional maximum converges to zero.

	\begin{proof}
		We suppress the subscript \(N\) from \(\bfa_N\), \(f_N\), and \(g_N\).
		For $\bfm\in\Band(\bfa)$ of self-overlap $r$,
		$F_{\TAP,\eta}(\bfm)$ below means
		$F_{\TAP,\Shift_r\eta}(\bfm)$. Since
		\[
		F_{\TAP}(\bfa)=F_{\TAP,\eta}(\bfa),
		\qquad
		F_{\TAP}(\bfm)\le F_{\TAP,\eta}(\bfm),
		\]
		it is enough to bound
		\[
		T_{\bfa,\eta}(\bfm)
		:=
		F_{\TAP,\eta}(\bfm)-F_{\TAP,\eta}(\bfa).
		\]
		
		Let $h\in\dR^N$ satisfy $P_{\bfa^\perp}h=g$. If
		$\bfm=\bfa+\tau\in\Band(\bfa)$, then
		\[
		g\cdot\tau=h\cdot\tau.
		\]
		
		\medskip\noindent\emph{Step 1 (Gaussian regression).}\;
		Define
		\[
		G_{\bfa}(\bfm)
		:=
		H_N(\bfm)-H_N(\bfa)-DH_N(\bfa)[\bfm-\bfa].
		\]
		Differentiating the covariance kernel shows that $G_{\bfa}$ is
		independent of $(H_N(\bfa),\nabla H_N(\bfa))$ and that, for
		$R_{12}=R(\bfm^1,\bfm^2)$,
		\[
		\frac1N
		\Cov\bigl(
		G_{\bfa}(\bfm^1),G_{\bfa}(\bfm^2)
		\bigr)
		=
		\xi(R_{12})-\xi(q)-\xi'(q)(R_{12}-q).
		\]
		
		The conditioned value and transverse gradient of $F_{\TAP,\eta}$ are affine
		functions of $H_N(\bfa)$ and $\nabla H_N(\bfa)$. Since
		$Dq_{\bfa}[\tau]=0$ and
		\[
		\partial_m\Phi_\eta^*(q,a_i)=-b_i,
		\]
		the gradient condition gives
		\[
		DH_N(\bfa)[\tau]
		=
		g\cdot\tau+\sum_i b_i\tau_i
		=
		\sum_i(b_i+h_i)\tau_i.
		\]
		Consequently, under the conditioning $\mc J_{\bfa}(f,g)$,
		\[
		H_N(\bfm)-H_N(\bfa)
		=
		G_{\bfa}(\bfm)
		+
		\sum_i(b_i+h_i)(m_i-a_i).
		\]
		
		Since $\eta$ is supported on $[q,1]$,
		formula~\eqref{eq:ising-A-mu} gives $A_\eta(q)=0$. Using
		\[
		\Phi_\eta^*(q,a_i)=\Phi_\eta(q,b_i)-a_i b_i,
		\]
		we obtain
		\[
		\begin{aligned}
			T_{\bfa,\eta}(\bfm)
			={}&
			G_{\bfa}(\bfm)+A_\eta(q_\bfm)\\
			&+
			\sum_i
			\left[
			m_i b_i
			+\Phi_\eta^*(q_\bfm,m_i)
			-\Phi_\eta(q,b_i)
			+h_i(m_i-a_i)
			\right].
		\end{aligned}
		\]
		
		\medskip\noindent\emph{Step 2 (Guerra-type comparison).}\;
		Write $x(r)=\eta([q,r])$. If $\eta\ne\delta_1$, choose
		\[
		q\le r_0<r_1<\cdots<r_L=1,
		\qquad
		0<\theta_1<\cdots<\theta_L\le1,
		\]
		so that $x=0$ on $[q,r_0)$ and
		\[
		x(r)=\theta_\ell,
		\qquad r_{\ell-1}<r<r_\ell.
		\]
		Let $(u_\alpha)_\alpha$ be the additive RPC with parameters
		$\theta_1,\ldots,\theta_L$ and overlap levels
		$r_0,\ldots,r_L$. When $\eta=\delta_1$, take a single index
		$\alpha$, set $u_\alpha=0$, and set
		$s_{\alpha\wedge\alpha}=1$.
		
		Independently, let $B_r^i(\alpha)$ and $Y(r,\alpha)$,
		$r\in[q,1]$, be centered continuous Gaussian fields with
		\[
		\dE B_s^i(\alpha)B_t^j(\beta)
		=
		\delta_{ij}
		\left[
		\xi'(s\wedge t\wedge s_{\alpha\wedge\beta})
		-\xi'(q)
		\right]
		\]
		and
		\[
		\dE Y(s,\alpha)Y(t,\beta)
		=
		N\left[
		\vartheta(s\wedge t\wedge s_{\alpha\wedge\beta})
		-\vartheta(q)
		\right].
		\]
		If $\eta\ne\delta_1$ and $r_0>q$, these fields are common to all leaves on
		$[q,r_0]$.
		
		The marking identity on the positive-mass intervals, together
		with the centering of the common increment on $[q,r_0]$, gives
		\[
		\dE\max_\alpha
		\left\{
		u_\alpha+Y(r,\alpha)-A_\eta(r)
		\right\}
		=
		\dE\max_\alpha u_\alpha
		\]
		for every deterministic $r\in[q,1]$.
		
		Set
		\[
		X_{\bfa}(\bfm,\alpha)
		=
		T_{\bfa,\eta}(\bfm)
		+
		u_\alpha+Y(q_\bfm,\alpha)-A_\eta(q_\bfm)
		\]
		and
		\[
		\begin{aligned}
			\widetilde X_{\bfa,h}(\bfm,\alpha)
			=
			u_\alpha+\sum_i
			\bigl[
			&m_i\bigl(b_i+h_i+B_{q_\bfm}^i(\alpha)\bigr)
			+\Phi_\eta^*(q_\bfm,m_i)\\
			&-\Phi_\eta(q,b_i)-h_i a_i
			\bigr].
		\end{aligned}
		\]
		
		Conditionally on the cascade labels, these processes have the
		same mean. For two marked points, set
		\[
		Q_{12}
		=
		q_{\bfm^1}\wedge q_{\bfm^2}
		\wedge s_{\alpha^1\wedge\alpha^2}.
		\]
		The covariances of their centered parts, divided by $N$, are
		\[
		\xi(R_{12})-\xi(q)-\xi'(q)(R_{12}-q)
		+\vartheta(Q_{12})-\vartheta(q)
		\]
		and
		\[
		R_{12}\bigl[\xi'(Q_{12})-\xi'(q)\bigr],
		\]
		respectively. Their variances agree, while the first covariance
		minus the second is
		\[
		\xi(R_{12})-\xi(Q_{12})
		-\xi'(Q_{12})(R_{12}-Q_{12})\ge0
		\]
		by Assumption~\ref{ass:xi}. Writing \(\dE_{\bfa,f,g}\) for expectation
		under the conditioning $\mc J_\bfa(f,g)$, Slepian's inequality,
		first on finite subsets and then by separability, gives
		\[
		\dE_{\bfa,f,g}
		\sup_{\substack{
				\bfm\in\Band_{\le1-\ve}(\bfa)\\
				\alpha
		}}
		X_{\bfa}(\bfm,\alpha)
		\le
		\dE_{\bfa,f,g}
		\sup_{\substack{
				\bfm\in\Band_{\le1-\ve}(\bfa)\\
				\alpha
		}}
		\widetilde X_{\bfa,h}(\bfm,\alpha).
		\]
		
		\medskip\noindent\emph{Step 3 (Fenchel inequality and the maximal estimate).}\;
		Fenchel's inequality gives
		\[
		\begin{aligned}
			\widetilde X_{\bfa,h}(\bfm,\alpha)
			\le{}&
			\sum_i
			\left[
			\Phi_\eta(q,b_i+h_i)
			-\Phi_\eta(q,b_i)-a_i h_i
			\right]\\
			&+
			u_\alpha+
			\sum_i
			\left[
			\Phi_\eta\bigl(
			q_\bfm,b_i+h_i+B_{q_\bfm}^i(\alpha)
			\bigr)
			-\Phi_\eta(q,b_i+h_i)
			\right].
		\end{aligned}
		\]
		The proof of Lemma~\ref{lem:parisi-maximal-estimate}, applied on
		$[q,1]$, yields
		\[
		\begin{aligned}
			\dE
			\sup_{\substack{r\in[q,1-\ve]\\\alpha}}
			\left\{
			u_\alpha+
			\sum_i
			\left[
			\Phi_\eta(r,b_i+h_i+B_r^i(\alpha))
			-\Phi_\eta(q,b_i+h_i)
			\right]
			\right\}
			\le
			\dE\max_\alpha u_\alpha+o(N).
		\end{aligned}
		\]
		The estimate is uniform in $(b_i)_{i=1}^N$ and $h$. On a possible initial
		interval with zero mass, the process is a martingale common to all
		leaves and Doob's inequality gives an $O(\sqrt N)$ contribution.
		The remaining intervals are exactly those treated in
		Lemma~\ref{lem:parisi-maximal-estimate}. Hence
		\[
		\begin{aligned}
			\dE_{\bfa,f,g}
			\sup_{\substack{
					\bfm\in\Band_{\le1-\ve}(\bfa)\\
					\alpha
			}}
			X_{\bfa}(\bfm,\alpha)
			\le{}&
			\dE\max_\alpha u_\alpha\\
			&+
			\sum_i
			\left[
			\Phi_\eta(q,b_i+h_i)
			-\Phi_\eta(q,b_i)-a_i h_i
			\right]
			+o(N).
		\end{aligned}
		\]
		
		\medskip\noindent\emph{Step 4 (removing the auxiliary field).}\;
		Let $\bfm_\star$ be a measurable maximizer of $T_{\bfa,\eta}$ on
		$\Band_{\le1-\ve}(\bfa)$. Under the conditional law, $\bfm_\star$
		is measurable with respect to $H_N$ and is independent of the RPC
		and the auxiliary Gaussian fields. Conditioning on $\bfm_\star$ and applying
		the preceding identity at fixed self-overlap gives
		\[
		\begin{aligned}
			\dE_{\bfa,f,g}
			\sup_{\substack{
					\bfm\in\Band_{\le1-\ve}(\bfa)\\
					\alpha
			}}
			X_{\bfa}(\bfm,\alpha)
			\ge
			\dE_{\bfa,f,g}
			\max_{\bfm\in\Band_{\le1-\ve}(\bfa)}
			T_{\bfa,\eta}(\bfm)
			+
			\dE\max_\alpha u_\alpha.
		\end{aligned}
		\]
		Comparison with the preceding upper bound yields, for every
		$h$ satisfying $P_{\bfa^\perp}h=g$,
		\[
		\begin{aligned}
			\dE_{\bfa,f,g}
			\max_{\bfm\in\Band_{\le1-\ve}(\bfa)}
			T_{\bfa,\eta}(\bfm)
			\le
			\sum_i
			\left[
			\Phi_\eta(q,b_i+h_i)
			-\Phi_\eta(q,b_i)-a_i h_i
			\right]
			+o(N).
		\end{aligned}
		\]
		The error is uniform in the conditioning data, $h$, and $\ve$,
		for the fixed finite-step order parameter $\eta$. Dividing by
		$N$, optimizing over $h$, and using the initial reduction gives
		\[
		\frac1N
		\dE_{\bfa,f,g}
		\sup_{\bfm\in\Band_{\le1-\ve}(\bfa)}
		\left[
		F_{\TAP}(\bfm)-F_{\TAP}(\bfa)
		\right]
		-
		\mathscr D_\eta(\bfa,g)
		\le o(1).
		\]
		Taking $N\to\infty$ and then $\ve\downarrow0$ proves the
		proposition.
	\end{proof}
	
	%------------------------------------------------------------------------------
	\subsection{The exponential moments in a band}
	
	Adding an outer Poisson level to the conditional comparison gives a
	bound on the conditional exponential moments.
	
	\begin{proposition}
		Let $q$, $\eta$, $\bfa$, $f$, and $g$ satisfy the assumptions of
		Proposition~\ref{prop:cond-band-noncrit}. Let $\theta\in(0,1)$ satisfy
		\[
		\theta<\eta(\{q\}).
		\]
		Then
		\[
		\limsup_{\ve\downarrow0}
		\limsup_{N\to\infty}
		\frac1{N\theta}
		\log
		\dE\left[
		\exp\left\{
		\theta
		\sup_{\bfm\in\Band_{\le1-\ve}(\bfa)}
		\bigl(F_{\TAP}(\bfm)-F_{\TAP}(\bfa)\bigr)
		\right\}
		\,\middle|\,
		\mc J_\bfa(f,g)
		\right]
		\le
		\mathscr D_\eta(\bfa,g).
		\]
	\end{proposition}
	
	\begin{proof}
		Apply the proof of Proposition~\ref{prop:cond-band-noncrit} after
		adding an outer Poisson level of parameter $\theta$, as in the proof
		of Proposition~\ref{prop:direct-tilt-TAP}.
	\end{proof}
	
	%==============================================================================
	\section{Large deviations lower bound}\label{sec:onshell-lb}
	
	We prove Theorem~\ref{thm:legendre} by establishing the large-deviation
	lower bound for TAP states at regular levels in \((f_{\eq},f_1)\).
	
	%==============================================================================
	\subsection{Preliminaries on selected levels}
	%==============================================================================
	
	We first identify the levels for which the Legendre problem is attained in
	the interior and record the first-contact properties of the corresponding
	constrained Parisi minimizer.
	
	\begin{definition}[Selected levels]\label{def:selected-levels}
		A level \(f>f_{\eq}\) is selected if the infimum defining
		\(\Sigma(f)\) is attained at some \(\theta_f\in(0,1)\) and the
		corresponding constrained Parisi minimizer \(\zeta_f\) satisfies
		\[
		\zeta_f(\{0\})=\theta_f.
		\]
	\end{definition}
	
	\begin{lemma}\label{lem:selected-first-contact}
		\begin{enumerate}[label=\textup{(\arabic*)}]
			\item The function \(\bar\Lambda\) is closed and convex,
			\[
			\bar\Lambda'_+(0)=f_{\eq},
			\qquad
			f_1=\bar\Lambda'_-(1)<\infty.
			\]
			
			\item Every \(f\in(f_{\eq},f_1)\) is selected.
			
			\item Suppose that \(f>f_{\eq}\) and that some
			\(\theta\in(0,1)\) attains the infimum defining \(\Sigma(f)\).
			Every corresponding constrained Parisi minimizer \(\zeta\) satisfies
			\(\zeta(\{0\})=\theta\). In particular, \(f\) is selected. Moreover,
			\(\zeta\) has a first positive
			contact point \(q\in(0,1)\) such that
			\[
			\zeta((0,q))=0,
			\qquad
			H_{\zeta,q}(r)>0
			\quad\text{for every \(r\in(0,q)\)}.
			\]
			For the Auffinger--Chen process associated with \(\zeta\),
			\[
			\dE\bigl[(\partial_x\Phi_\zeta(q,X_q))^2\bigr]=q.
			\]
		\end{enumerate}
	\end{lemma}
	
	\begin{proof}
		We first prove \textup{(1)}. Let
		\(\theta=t\theta_1+(1-t)\theta_2\), and choose \(\zeta_i\) arbitrarily
		close to the constrained infimum defining \(\Lambda(\theta_i)\). The
		probability measure
		\[
		\zeta
		:=
		\frac{t\theta_1\zeta_1+(1-t)\theta_2\zeta_2}{\theta}
		\]
		satisfies
		\[
		\zeta(\{0\})
		\ge
		\frac{t\theta_1^2+(1-t)\theta_2^2}{\theta}
		\ge\theta.
		\]
		The convexity of the Parisi functional therefore gives
		\[
		\Lambda(\theta)
		\le
		t\theta_1P_{\zeta_1}+(1-t)\theta_2P_{\zeta_2}.
		\]
		Taking the constrained infima proves that \(\Lambda\) is convex on
		\((0,1)\).
		
		Let \(\zeta_*\) be the Parisi minimizer and set
		\(\zeta_s=(1-s)\zeta_*+s\delta_0\). Then
		\[
		f_{\eq}
		\le\frac{\Lambda(s)}s
		\le P_{\zeta_s}
		\longrightarrow f_{\eq}.
		\]
		Thus \(\bar\Lambda\) is continuous at zero and
		\begin{equation}\label{eq:selected-right-slope-zero}
			\bar\Lambda'_+(0)=f_{\eq}.
		\end{equation}
		
		For \(\mu,\nu\in\mc P([0,1])\), interpolation in the Parisi equation
		and \(|\partial_x\Phi|\le1\) give
		\[
		|P_\mu-P_\nu|
		\le
		\|\xi''\|_\infty
		\int_0^1
		|\mu([0,s])-\nu([0,s])|\,\dd s
		=
		\|\xi''\|_\infty W_1(\mu,\nu).
		\]
		If \(\mu(\{0\})\ge1-h\), then
		\[
		W_1(\mu,\delta_0)=\int_0^1s\,\mu(\dd s)\le h.
		\]
		Consequently,
		\[
		(1-h)\bigl(P_{\delta_0}-\|\xi''\|_\infty h\bigr)
		\le\Lambda(1-h)
		\le(1-h)P_{\delta_0}.
		\]
		It follows that \(\bar\Lambda(1)=P_{\delta_0}\) and
		\[
		\frac{\bar\Lambda(1)-\bar\Lambda(1-h)}h
		\le P_{\delta_0}+\|\xi''\|_\infty.
		\]
		The endpoint definitions now show that \(\bar\Lambda\) is closed and
		convex, and that \(f_1<\infty\).
		
		We record the saturation argument used to prove \textup{(2)} and
		\textup{(3)}. Suppose that \(f>f_{\eq}\), that
		\(\theta\in(0,1)\) minimizes
		\[
		G_f(s):=\bar\Lambda(s)-sf,
		\]
		and let \(\zeta\) minimize \(P_\nu\) under the constraint
		\(\nu(\{0\})\ge\theta\). Such a minimizer exists by compactness and
		continuity.
		
		Suppose that \(a:=\zeta(\{0\})>\theta\). For every
		\(\nu\in\mc P([0,1])\), the measure
		\((1-t)\zeta+t\nu\) remains admissible for all sufficiently small
		\(t>0\). Hence
		\[
		P_\zeta
		\le P_{(1-t)\zeta+t\nu}
		\le(1-t)P_\zeta+tP_\nu.
		\]
		Thus
		\begin{equation}\label{eq:selected-slack-global}
			\zeta(\{0\})>\theta
			\quad\Longrightarrow\quad
			P_\zeta=f_{\eq}.
		\end{equation}
		For \(\theta<s<a\), the same measure is
		admissible in the constrained problem at \(s\), and therefore
		\(\Lambda(s)=sf_{\eq}\). Since \(f>f_{\eq}\),
		\[
		G_f(s)=s(f_{\eq}-f)
		<\theta(f_{\eq}-f)=G_f(\theta),
		\]
		a contradiction. Thus \(\zeta(\{0\})=\theta\).
		
		We now prove \textup{(2)}. Fix \(f\in(f_{\eq},f_1)\). The function
		\(G_f\) attains its minimum on \([0,1]\). By
		\eqref{eq:selected-right-slope-zero},
		\[
		G'_{f,+}(0)=f_{\eq}-f<0,
		\qquad
		G'_{f,-}(1)=f_1-f>0.
		\]
		Every minimizer therefore lies in \((0,1)\), and the saturation argument
		above shows that \(f\) is selected.
		
		It remains to prove \textup{(3)}. Fix \(f>f_{\eq}\), let
		\(\theta\in(0,1)\) minimize \(G_f\), and let \(\zeta\) be a
		corresponding constrained Parisi minimizer. The preceding
		saturation argument gives \(\zeta(\{0\})=\theta\).
		
		Write \(M_t=\partial_x\Phi_\zeta(t,X_t)\). The first-variation obstacle is
		\[
		H(t)
		:=
		\frac12\int_t^1\xi''(s)\bigl(\dE M_s^2-s\bigr)\,\dd s.
		\]
		If \(c=\min_{t\in(0,1]}H(t)\), the first-variation formula
		\eqref{eq:TAP-first-variation} gives
		\[
		\supp(\zeta)\cap(0,1]\subset\{t:H(t)=c\},
		\qquad
		H(0)\ge c.
		\]
		Suppose that the positive contact points accumulate at zero. Continuity
		of \(H\) then gives \(H(0)=c\). The first-variation condition is now the
		unconstrained Parisi optimality condition, so \(P_\zeta=f_{\eq}\).
		For \(0<h<1-\theta\), set
		\[
		\zeta_h
		:=
		(\theta+h)\delta_0
		+
		\frac{1-\theta-h}{1-\theta}\,
		\zeta|_{(0,1]}.
		\]
		The signed measure
		\[
		\frac{\zeta_h-\zeta}{h}
		=
		\delta_0-\frac1{1-\theta}\zeta|_{(0,1]}
		\]
		has total mass zero. Since \(H(0)=c\) and
		\(H=c\) on \(\supp(\zeta)\cap(0,1]\), the first-variation formula
		gives
		\[
		P_{\zeta_h}=P_\zeta+o(h)=f_{\eq}+o(h).
		\]
		As \(\zeta_h(\{0\})=\theta+h\),
		\[
		\begin{aligned}
			G_f(\theta+h)
			&\le(\theta+h)\bigl(P_{\zeta_h}-f\bigr)\\
			&=G_f(\theta)+h(f_{\eq}-f)+o(h)
			<G_f(\theta)
		\end{aligned}
		\]
		for all sufficiently small \(h>0\), a contradiction.
		
		Since \(M_1=\tanh X_1\), we have \(\dE M_1^2<1\). Hence
		\(H(t)<H(1)\) for \(t<1\) sufficiently close to \(1\), so
		\(\zeta(\{1\})=0\). The positive contact set is therefore nonempty and
		contains a point below \(1\), because \(\theta<1\) and the positive
		support of \(\zeta\) is contained in the contact set. It has a smallest
		element \(q\in(0,1)\). As \(q\) is an interior minimum of \(H\), the
		contact equation gives \(\dE M_q^2=q\). For \(0<r<q\),
		\[
		H_{\zeta,q}(r)=H(r)-H(q)>0.
		\]
		The support inclusion also gives \(\zeta((0,q))=0\).
	\end{proof}
	
	%==============================================================================
	\subsection{Main statement and strategy}\label{sub:strategy}
	%==============================================================================
	
	Recall the selected levels from Definition~\ref{def:selected-levels} and the
	regular levels from Definition~\ref{def:regular-levels}. For a selected level,
	\begin{equation}\label{eq:selected-level-variational-formula}
		\Sigma(f)
		=
		\inf_{\theta\in(0,1)}
		\inf_{\substack{\zeta\in\mc P([0,1])\\
				\zeta(\{0\})=\theta}}
		\theta\bigl(P_\zeta-f\bigr).
	\end{equation}
	If the level is regular, then \(\Sigma\) is differentiable at \(f\), and
	\[
	\Sigma'(f)=-\theta_f.
	\]
	
	For the Kac--Rice arguments below, first assume that \(\xi\) has at least
	two nonzero coefficients. If
	\(\xi(x)=\beta_p^2x^p\), choose an even \(\ell\ne p\), let
	\(H_{N,\ell}\) be an independent pure \(\ell\)-spin Hamiltonian, and set
	\[
	H_N^\gamma:=H_N+\gamma H_{N,\ell},
	\qquad
	\xi_\gamma(x):=\xi(x)+\gamma^2x^\ell.
	\]
	For every \(\gamma>0\), the joint law of
	\((\nabla H_N^\gamma(\bfm),H_N^\gamma(\bfm))\) is nondegenerate when
	\(q_\bfm>0\). We apply this section and
	Appendix~\ref{sec:one-point-susy-computations} with all Parisi and TAP
	quantities formed from \(\xi_\gamma\). At fixed \(\gamma\), the dependence
	on \(\gamma\) is suppressed. Continuity of the Parisi and TAP quantities
	preserves the strict Plefka inequality for small \(\gamma\). The local
	negative Hessian estimate in
	Proposition~\ref{prop:one-point-conditional-hessian} makes the existence
	of the retained critical points stable under this \(C^2\) perturbation.
	Letting \(\gamma\downarrow0\) after \(N\to\infty\) proves the pure case.
	
	The main result of this section is the following lower bound.
	
	\begin{proposition}\label{prop:onshell-lower-bound}
		Let \(f\in(f_{\eq},f_1)\) be regular. Let \(q\) be the first
		positive contact point of \(\zeta_f\), and let \(X\) be the Auffinger--Chen process
		associated with \(\zeta_f\).
		Assume that
		\[
		1-\xi''(q)\dE\left[
		\left(
		\partial_{xx}\Phi_{\zeta_f}(q,X_q)
		\right)^2
		\right]>0.
		\]
		Then
		\[
		\lim_{\ve\downarrow0}
		\liminf_{N\to\infty}
		\frac1N\log\dP\bigl(\mc E_{N,\ve}(f)\bigr)
		\ge\Sigma(f).
		\]
	\end{proposition}
	
	Together with the upper bound obtained from
	Proposition~\ref{prop:direct-tilt-TAP} and Markov's inequality,
	Proposition~\ref{prop:onshell-lower-bound} yields the existence
	probability~\eqref{eq:existence-tap} in
	Theorem~\ref{thm:legendre}.
	
	\begin{definition}[SUSY states]\label{def:susy-state}
		Fix a regular \(f\in(f_{\eq},f_1)\), write \(\theta=\theta_f\) and
		\(\zeta=\zeta_f\), and let \(q\) be the first positive contact point of
		\(\zeta\). Let \(X\) be the Auffinger--Chen process associated with
		\(\zeta\).
		
		For \(\ve>0\), an \(\ve\)-SUSY state at level \(f\) is a
		critical point \(\bfm\in(-1,1)^N\) of \(F_{\TAP,\zeta}\) satisfying
		\[
		|q_\bfm-q|\le\ve,
		\qquad
		\left|\frac1N F_{\TAP,\zeta}(\bfm)-f\right|\le\ve,
		\qquad
		\left|\frac1N F_{\TAP}(\bfm)-f\right|\le\ve.
		\]
		Define \(b_i\) by
		\[
		m_i=\partial_x\Phi_\zeta(q,b_i),
		\qquad 1\le i\le N,
		\]
		and require
		\[
		W_2\left(
		\frac1N\sum_{i=1}^N\delta_{b_i},
		\Law(X_q)
		\right)\le\ve.
		\]
		Let \(\mathsf N_\ve\) denote the number of
		\(\ve\)-SUSY states at level \(f\).
	\end{definition}
	
	For \(\sigma\in\{-,+\}\), let
	\(\nabla_\sigma^2F_{\TAP,\zeta}(x)\) denote the one-sided Hessian obtained
	by approaching \(q_x\) from below or above, respectively. The two Hessians
	coincide when \(q_x\) is not an atom of \(\zeta\). At an atom, the unsigned
	notation \(\nabla^2F_{\TAP,\zeta}(x)\) means
	\(\nabla_-^2F_{\TAP,\zeta}(x)\).
	
	\begin{definition}[Isolated SUSY states]\label{def:isolated-susy-state}
		For \(\ve,\kappa>0\) and \(0<\delta<q/4\), an
		\(\ve\)-SUSY state \(\bfm\) is isolated at parameters
		\((\delta,\kappa)\) if:
		\begin{enumerate}[label=\textup{(\roman*)}]
			\item no distinct \(\ve\)-SUSY state \(\bfm'\)
			satisfies
			\[
			R(\bfm,\bfm')\in[\delta,q-\delta].
			\]
			\item
			\[
			\sup_{\substack{
					x\in(-1,1)^N,\\
					\|x-\bfm\|_2\le2\sqrt{\delta N}
			}}
			\lambda_{\max}\!\left(
			\nabla^2F_{\TAP,\zeta}(x)
			\right)
			\le-\kappa.
			\]
		\end{enumerate}
		Let
		\[
		\widehat{\mathsf N}_{\ve,\delta,\kappa}
		:=
		\#\{\text{\(\ve\)-SUSY states isolated at parameters
			\((\delta,\kappa)\)}\}.
		\]
	\end{definition}
	Set
	\begin{equation}\label{eq:susy-intermediate-gap}
		g_\delta
		:=
		\frac{\theta}{8}
		\inf_{r\in[\delta,q-\delta]}H_{\zeta,q}(r)>0.
	\end{equation}
	
	The strategy is as follows.
	\begin{enumerate}[label=\textup{(\arabic*)}]
		\item The one-point Kac--Rice formula gives the annealed exponent
		\(\Sigma(f)\) of \(\mathsf N_\ve\). For fixed \(\delta\) and
		\(\kappa\), the truncation below preserves this exponent.
		On \([\delta,q-\delta]\), this follows from the strict gap
		\eqref{eq:susy-intermediate-gap}.
		On \([q-\delta,q]\), it follows from the
		conditional local spectral-gap estimate and the determinant
		lower-tail estimate.
		\item The two isolation conditions imply that distinct isolated states
		have overlap less than \(\delta\). After radial normalization they form
		a one-sided spherical code, and therefore
		\[
		\dE\widehat{\mathsf N}_{\ve,\delta,\kappa}
		\le
		A_N\left(\frac{\delta}{q-\ve}\right)
		\dP\left(\widehat{\mathsf N}_{\ve,\delta,\kappa}\ne0\right).
		\]
		The spherical-code estimate and Step~\textup{(1)} prove
		Proposition~\ref{prop:onshell-lower-bound} after letting
		\(\delta\downarrow0\).
	\end{enumerate}
	
	%==============================================================================
	\subsection{The SUSY annealed complexity}
	%==============================================================================
	
	We first compute the annealed exponent of the SUSY count.
	
	\begin{proposition}[Annealed complexity of SUSY states]\label{prop:susy-annealed}
		Let \(f\in(f_{\eq},f_1)\) be regular, write
		\(\theta=\theta_f\) and \(\zeta=\zeta_f\), and let \(q\) be the first
		positive contact point of \(\zeta\). Let \(X\) be the
		Auffinger--Chen process associated with \(\zeta\). Assume that
		\[
		1-\xi''(q)\dE\left[
		\left(
		\partial_{xx}\Phi_\zeta(q,X_q)
		\right)^2
		\right]>0.
		\]
		Then the counts \(\mathsf N_\ve\) from
		Definition~\ref{def:susy-state} satisfy
		\[
		\liminf_{\ve\downarrow0}
		\liminf_{N\to\infty}
		\frac1N\log\dE\mathsf N_\ve
		=
		\limsup_{\ve\downarrow0}
		\limsup_{N\to\infty}
		\frac1N\log\dE\mathsf N_\ve
		=
		\theta\bigl(P_\zeta-f\bigr)
		=
		\Sigma(f).
		\]
	\end{proposition}

	%==============================================================================
	\subsection{Conditional Gaussian regression}
	%==============================================================================
	
	Conditioning on the value and gradient at a fixed point gives the following
	regression formula.
	
	\begin{lemma}\label{lem:reduced-same-shell-increment}
		Fix a regular \(f\in(f_{\eq},f_1)\), write
		\(\theta=\theta_f\) and \(\zeta=\zeta_f\), and let \(q\) be the first
		positive contact point of \(\zeta\). Let \(X\) be the
		Auffinger--Chen process associated with \(\zeta\). For
		\(\ve>0\), fix \(\bfm\in(-1,1)^N\), and define \(b_i\) by
		\[
		m_i=\partial_x\Phi_\zeta(q,b_i).
		\]
		Write \(\bfb=(b_i)_{i\le N}\).
		Assume that
		\[
		|q_\bfm-q|\le\ve,
		\qquad
		W_2\left(
		\frac1N\sum_{i=1}^N\delta_{b_i},
		\Law(X_q)
		\right)\le\ve.
		\]
		For \(f'\in[f-\ve,f+\ve]\), condition on
		\[
		F_{\TAP,\zeta}(\bfm)=Nf',
		\qquad
		\nabla F_{\TAP,\zeta}(\bfm)=0.
		\]
		Assume also that
		\[
		\left|
		f'
		+\frac1N\left(
		F_{\TAP}(\bfm)-F_{\TAP,\zeta}(\bfm)
		\right)
		-f
		\right|
		\le\ve.
		\]
		Denote the corresponding conditional expectation by
		\(\dE_{\bfm,f'}\).
		There is a centered Gaussian field \(G_\bfm\), independent of the
		conditioned value and gradient, such that, for every fixed
		\(r\in(-q,q)\),
		\begin{equation}\label{eq:ising-reduced-sibling-increment}
			\begin{aligned}
				F_{\TAP,\zeta}(\bfm')-F_{\TAP,\zeta}(\bfm)
				&=G_\bfm(\bfm')
				+\sum_{i=1}^N\left[
				\frac{\xi'(r)}{\xi'(q)}m'_i b_i
				+\Phi_\zeta^*(q,m'_i)-\Phi_\zeta(q,b_i)
				\right]\\
				&\quad+N\theta\bigl[\xi(r)-r\xi'(r)-\xi(q)+q\xi'(q)\bigr]
				+N o_\ve(1)+N o_\delta(1)+o_N(N),
			\end{aligned}
		\end{equation}
		uniformly when \(\lvert q_{\bfm'}-q\rvert\le\delta\) and
		\(\lvert R(\bfm,\bfm')-r\rvert\le\delta\). Here
		\(o_N(N)/N\to0\) for fixed \(\ve,\delta\), while
		\(o_\ve(1)\to0\) and \(o_\delta(1)\to0\) as their respective
		parameters tend to zero. All errors are uniform over the displayed
		\(\ve\)-window conditions. For two points
		\(\bfm'^1,\bfm'^2\) in this slice,
		\[
		\frac1N\dE_{\bfm,f'}\left[
		G_\bfm(\bfm'^1)G_\bfm(\bfm'^2)
		\right]
		=
		\xi(R_{12})-
		\frac{\xi'(r)^2}{\xi'(q)}R_{12}+C_r
		+o_\ve(1)+o_\delta(1)+o_N(1),
		\]
		where \(R_{12}=N^{-1}\langle\bfm'^1,\bfm'^2\rangle\) and \(C_r\) depends
		only on \(q\) and \(r\).
	\end{lemma}
	
	\begin{proof}
		Let \(X\) be the Auffinger--Chen process associated with \(\zeta\), set
		\(M_t=\partial_x\Phi_\zeta(t,X_t)\), and recall the definition of
		\(\widehat\zeta\) from \eqref{eq:order-parameter-cumulants}.
		By Lemmas~\ref{lem:selected-first-contact} and
		\ref{lem:first-variation-contact}, together with stationarity of
		\eqref{eq:selected-level-variational-formula} at \(\theta\), we have
		\begin{align}
			\dE M_q^2&=q,
			\label{eq:ising-contact}\\
			\dE\Phi_\zeta(q,X_q)
			&=
			f+\theta\bigl(q\xi'(q)-\xi(q)\bigr)
			+\frac12\int_q^1t\xi''(t)\zeta([0,t])\,\dd t.
			\label{eq:ising-selected-potential}
		\end{align}
		Optimality of \(\zeta\) implies that \(\Shift_q\zeta\) minimizes
		\(V_\eta(q,\Law(M_q))\) over \(\eta\in\mc P([q,1])\). The
		Auffinger--Chen process in Lemma~\ref{lem:first-variation-contact}
		then agrees in law with \(X\) from time \(q\) onward. Its contact
		equations verify the hypotheses of
		Lemma~\ref{lem:AC-susceptibility}. Hence
		\[
		\dE\,\partial_{xx}\Phi_\zeta(q,X_q)=\widehat\zeta(q).
		\]
		Together with~\eqref{eq:ising-contact} and Gaussian integration by
		parts for the Auffinger--Chen law, this yields
		\[
		\dE[M_qX_q]
		=
		\xi'(q)\left(
		\theta\dE M_q^2
		+\dE\,\partial_{xx}\Phi_\zeta(q,X_q)
		\right)
		=
		\xi'(q)\bigl(\theta q+\widehat\zeta(q)\bigr).
		\]
		The empirical-law window therefore yields
		\begin{equation}\label{eq:ising-ward-inner}
			\frac1N\sum_{i=1}^N m_i b_i
			=
			\xi'(q)\bigl(\theta q+\widehat\zeta(q)\bigr)
			+o_\ve(1).
		\end{equation}
		The definition of \(b_i\) gives the Fenchel equality
		\begin{equation}\label{eq:ising-fenchel-identity}
			\Phi_\zeta^*(q,m_i)+m_i b_i=\Phi_\zeta(q,b_i),
			\qquad 1\le i\le N.
		\end{equation}
		Together with~\eqref{eq:ising-selected-potential}, the
		empirical-law window gives
		\begin{equation}\label{eq:ising-ward-energy}
			-\sum_{i=1}^N\bigl(\Phi_\zeta^*(q,m_i)+m_i b_i\bigr)
			+N\theta\bigl(q\xi'(q)-\xi(q)\bigr)
			+\frac N2\int_q^1t\xi''(t)\zeta([0,t])\,\dd t
			=-Nf+N o_\ve(1).
		\end{equation}
		
		Differentiating the deterministic TAP correction and using the conditioned
		gradient and the finite-window assumptions gives
		\[
		\nabla H_N(\bfm)
		=
		\bfb+\xi''(q)\widehat\zeta(q)\bfm
		+\sqrt N\,o_\ve(1),
		\]
		where the error is measured in Euclidean norm.
		
		Recall that
		\[
		\Cov(H_N(\bfm'),H_N(\bfm))=N\xi(R(\bfm,\bfm'))\,,\qquad
		\Cov(H_N(\bfm'),\nabla H_N(\bfm))=\xi'(R(\bfm,\bfm'))\,\bfm'\,,
		\]
		together with
		\[
		\Cov(\nabla H_N(\bfm),\nabla H_N(\bfm))
		=
		\xi'(q_\bfm)I+\frac{\xi''(q_\bfm)}N\bfm\bfm^\top.
		\]
		Set
		\[
		x:=\frac{\bfb-\theta\xi'(q)\bfm}{\xi'(q)}.
		\]
		By~\eqref{eq:ising-ward-inner}, we have
		\[
		\frac1N\langle x,\bfm\rangle
		=\widehat\zeta(q)+o_\ve(1).
		\]
		Set
		\[
		L_\bfm:=x\cdot\nabla H_N(\bfm)+\theta H_N(\bfm).
		\]
		Tilt the original Gaussian law by the density
		\[
		\exp\left\{L_\bfm-\frac12\dE L_\bfm^2\right\}
		\].
		Under the tilted law, the covariance is unchanged and every centered
		Gaussian variable \(Z\) has mean
		\(\Cov(Z,L_\bfm)\). We verify below that the resulting means of
		\(H_N(\bfm)\) and \(\nabla H_N(\bfm)\) agree with their conditioned values,
		up to the prescribed errors. For the gradient,
		\[
		\Cov\bigl(
		\nabla H_N(\bfm),
		L_\bfm
		\bigr)
		=
		\bfb+\xi''(q)\widehat\zeta(q)\bfm
		+\sqrt N\,o_\ve(1)
		=
		\nabla H_N(\bfm)+\sqrt N\,o_\ve(1),
		\]
		where the vector errors are measured in Euclidean norm. For the value,
		\eqref{eq:ising-ward-energy} and the conditioned value of
		\(F_{\TAP,\zeta}(\bfm)\) give
		\[
		\Cov\bigl(
		H_N(\bfm),
		L_\bfm
		\bigr)
		=
		N\xi'(q)\widehat\zeta(q)+\theta N\xi(q)
		+N o_\ve(1)
		=
		H_N(\bfm)+N o_\ve(1).
		\]
		Since \(L_\bfm\) is a linear combination of the conditioned value and
		gradient, matching their tilted means with the prescribed values identifies
		the conditional mean by Gaussian regression.
		Therefore, if
		\(\lvert R(\bfm,\bfm')-r\rvert\le\delta\),
		\[
		\begin{aligned}
			\dE_{\bfm,f'} H_N(\bfm')
			&=\Cov\bigl(H_N(\bfm'),L_\bfm\bigr)
			+N o_\ve(1)+o_N(N)\\
			&=\xi'(r)\langle\bfm',x\rangle+\theta N\xi(r)
			+N o_\ve(1)+N o_\delta(1)+o_N(N)\\
			&=\frac{\xi'(r)}{\xi'(q)}\sum_i m'_i b_i
			+N\theta\bigl(\xi(r)-r\xi'(r)\bigr)
			+N o_\ve(1)+N o_\delta(1)+o_N(N).
		\end{aligned}
		\]
		Hence,
		\[
		\begin{aligned}
			\dE_{\bfm,f'}\bigl[H_N(\bfm')-H_N(\bfm)\bigr]
			&=
			\frac{\xi'(r)}{\xi'(q)}\sum_i m'_i b_i-\sum_i m_i b_i\\
			&\quad+N\theta\bigl[\xi(r)-r\xi'(r)-\xi(q)+q\xi'(q)\bigr]\\
			&\quad+N o_\ve(1)+N o_\delta(1)+o_N(N).
		\end{aligned}
		\]
		
		Since \(q_\bfm\) and \(q_{\bfm'}\) lie in the respective
		\(\ve\)- and \(\delta\)-windows around \(q\), both TAP corrections
		may be evaluated at \(q\) with an
		\(N o_\ve(1)+N o_\delta(1)\) error. Hence,
		\[
		F_{\TAP,\zeta}(\bfm')-F_{\TAP,\zeta}(\bfm)
		=
		H_N(\bfm')-H_N(\bfm)
		+\sum_i\bigl[\Phi_\zeta^*(q,m'_i)-\Phi_\zeta^*(q,m_i)\bigr]
		+N o_\ve(1)+N o_\delta(1).
		\]
		Set
		\[
		G_\bfm(\bfm')
		:=
		H_N(\bfm')
		-
		\dE\left[
		H_N(\bfm')
		\,\middle|\,
		F_{\TAP,\zeta}(\bfm),
		\nabla F_{\TAP,\zeta}(\bfm)
		\right].
		\]
		By joint Gaussianity, \(G_\bfm\) is centered and independent of the
		conditioned value and gradient.
		Combining the last two displays with~\eqref{eq:ising-fenchel-identity} proves
		\eqref{eq:ising-reduced-sibling-increment}.
		Smoothness of \(\xi\) and continuity of the TAP correction in the self-overlap
		make the calculation uniform on the stated \(\delta\)-window, with error
		\(N o_\ve(1)+N o_\delta(1)+o_N(N)\).
		
		Write
		\[
		\bfm'^a
		=
		\frac{R(\bfm,\bfm'^a)}{q_\bfm}\bfm+\tau^a,
		\qquad \tau^a\perp\bfm.
		\]
		Conditioning on the transverse gradient subtracts
		\[
		\frac{
			\xi'(R(\bfm,\bfm'^1))\xi'(R(\bfm,\bfm'^2))
		}{\xi'(q_\bfm)}
		\langle\tau^1,\tau^2\rangle
		=
		N\frac{\xi'(r)^2}{\xi'(q)}
		\left(R_{12}-\frac{r^2}{q}\right)
		+N o_\ve(1)+N o_\delta(1)+o_N(N).
		\]
		All remaining terms in the Schur complement involve \(H_N(\bfm)\) or
		\(\langle\bfm,\nabla H_N(\bfm)\rangle\). On the slice, their covariances
		with \(H_N(\bfm'^a)\) depend only on \(q\) and \(r\), up to the stated
		errors. Absorbing these contributions and the \(R_{12}\)-independent part
		of the preceding display into \(NC_r\) gives the covariance in the lemma.
	\end{proof}
	
	%==============================================================================
	\subsection{Conditional Guerra bound at fixed overlap}
	%==============================================================================
	
	For the remainder of this subsection, fix a regular
	\(f\in(f_{\eq},f_1)\), write
	\(\theta=\theta_f\) and \(\zeta=\zeta_f\), and let \(q\) be the first
	positive contact point of \(\zeta\). Let \(X\) be the
	Auffinger--Chen process associated with \(\zeta\).
	We bound the TAP increment over configurations at overlap \(r\) with \(\bfm\).
	For \(\bfm\in[-1,1]^N\), \(r\in(0,q)\), and \(\delta>0\), let
	\[
	\mc C_N(\bfm,r,\delta)
	:=
	\left\{
	\bfm'\in[-1,1]^N:
	\lvert R(\bfm,\bfm')-r\rvert\le\delta,
	\qquad
	\lvert q_{\bfm'}-q\rvert\le\delta
	\right\}.
	\]
	
	\begin{definition}
		Fix \(r\in(0,q)\). For \(\bfm\in(-1,1)^N\), define \(b_i\) by
		\[
		m_i=\partial_x\Phi_\zeta(q,b_i),
		\qquad 1\le i\le N.
		\]
		Let \(\rho\) be a finite-step probability measure on \([r,q]\), and write
		\[
		\bar x_\rho(t)
		=
		\theta+(1-\theta)\rho([r,t]),
		\qquad r\le t<q.
		\]
		Let \(\phi_\rho\) solve on \([r,q]\)
		\begin{equation}\label{eq:sibling-reduced-PDE}
			\begin{aligned}
				\partial_t\phi_\rho
				+\frac{\xi''(t)}2
				\left(
				\partial_{xx}\phi_\rho
				+\bar x_\rho(t)(\partial_x\phi_\rho)^2
				\right)&=0,\\
				\phi_\rho(q,x)&=\Phi_\zeta(q,x).
			\end{aligned}
		\end{equation}
		For
		\[
		Z\sim\mc N\left(
		0,\xi'(r)-\frac{\xi'(r)^2}{\xi'(q)}
		\right),
		\]
		define
		\[
		\begin{aligned}
			\mc J_{N,\bfm,q,r}(\rho)
			:={}&
			\theta\bigl[\xi(r)-r\xi'(r)-\xi(q)+q\xi'(q)\bigr]
			-\frac1N\sum_i\Phi_\zeta(q,b_i)\\
			&+\frac1N\sum_i
			\dE_Z\phi_\rho\left(
			r,Z+\frac{\xi'(r)}{\xi'(q)}b_i
			\right)
			-\frac12\int_r^q t\xi''(t)\bar x_\rho(t)\,\dd t.
		\end{aligned}
		\]
	\end{definition}
	
	\begin{proposition}\label{prop:sibling-guerra-bound}
		Fix a regular \(f\in(f_{\eq},f_1)\), write
		\(\theta=\theta_f\) and \(\zeta=\zeta_f\), and let \(q\) be the first
		positive contact point of \(\zeta\). Let \(X\) be the
		Auffinger--Chen process associated with \(\zeta\).
		For \(\ve>0\), fix \(\bfm\in(-1,1)^N\), and define \(b_i\) by
		\[
		m_i=\partial_x\Phi_\zeta(q,b_i).
		\]
		Assume that
		\[
		|q_\bfm-q|\le\ve,
		\qquad
		W_2\left(
		\frac1N\sum_{i=1}^N\delta_{b_i},
		\Law(X_q)
		\right)\le\ve.
		\]
		Let \(f'\in[f-\ve,f+\ve]\) satisfy
		\[
		\left|
		f'
		+\frac1N\left(
		F_{\TAP}(\bfm)-F_{\TAP,\zeta}(\bfm)
		\right)
		-f
		\right|
		\le\ve.
		\]
		Let \(\dE_{\bfm,f'}\) denote conditional expectation given
		\[
		F_{\TAP,\zeta}(\bfm)=Nf',
		\qquad
		\nabla F_{\TAP,\zeta}(\bfm)=0.
		\]
		For every \(r\in(0,q)\),
		every \(\delta>0\), and every fixed finite-step probability measure
		\(\rho\) on \([r,q]\),
		\begin{equation}\label{eq:sibling-guerra-bound}
			\frac1N\dE_{\bfm,f'}
			\max_{\bfm'\in\mc C_N(\bfm,r,\delta)}
			\left[
			F_{\TAP,\zeta}(\bfm')-F_{\TAP,\zeta}(\bfm)
			\right]
			\le
			\mc J_{N,\bfm,q,r}(\rho)
			+o_\ve(1)+o_\delta(1)+o_N(1).
		\end{equation}
		The errors are uniform over the displayed \(\ve\)-window data,
		with \(N\to\infty\) taken before
		\(\ve,\delta\downarrow0\).
	\end{proposition}
	
	\begin{proof}
		Fix a finite-step probability measure \(\rho\) on \([r,q]\), and write
		\[
		I_N:=\mc C_N(\bfm,r,\delta).
		\]
		The bound is immediate when \(I_N\) is empty.
		
		\medskip\noindent\emph{Step 1. Reduction.}
		Define
		\[
		\Psi_r(\bfm')
		:=
		\sum_{i=1}^N
		\left[
		\frac{\xi'(r)}{\xi'(q)}b_im_i'
		+\Phi_\zeta^*(q,m_i')
		\right].
		\]
		Lemma~\ref{lem:reduced-same-shell-increment} gives, uniformly over
		\(\bfm'\in I_N\),
		\[
		\begin{aligned}
			F_{\TAP,\zeta}(\bfm')-F_{\TAP,\zeta}(\bfm)
			={}&
			G_\bfm(\bfm')+\Psi_r(\bfm')
			-\sum_{i=1}^N\Phi_\zeta(q,b_i)\\
			&+
			N\theta
			\bigl[
			\xi(r)-r\xi'(r)-\xi(q)+q\xi'(q)
			\bigr]
			+N o_\ve(1)+N o_\delta(1)+o_N(N).
		\end{aligned}
		\]
		
		\medskip\noindent\emph{Step 2. Gaussian comparison.}
		Set
		\[
		K_r(s)
		:=
		\xi(s)-\frac{\xi'(r)^2}{\xi'(q)}s+r\xi'(r)-\xi(r).
		\]
		Then
		\[
		K_r''=\xi'',
		\qquad
		K_r'(r)
		=
		\xi'(r)-\frac{\xi'(r)^2}{\xi'(q)}
		\ge0,
		\qquad
		K_r(r)=rK_r'(r).
		\]
		In particular,
		\[
		\int_r^t v\xi''(v)\,\dd v
		=
		tK_r'(t)-K_r(t),
		\qquad r\le t\le q.
		\]
		
		For \(\bfm'^1,\bfm'^2\in I_N\), write
		\[
		R_{12}:=R(\bfm'^1,\bfm'^2).
		\]
		The covariance formula in
		Lemma~\ref{lem:reduced-same-shell-increment} gives
		\[
		\frac1{2N}\dE_{\bfm,f'}
		\left[
		G_\bfm(\bfm'^1)
		-G_\bfm(\bfm'^2)
		\right]^2
		=
		K_r(q)-K_r(R_{12})
		+o_\ve(1)+o_\delta(1)+o_N(1).
		\]
		
		Let \((U_\alpha)_\alpha\) be the additive RPC with cumulative
		\(\bar x_\rho\). Let \(B_t^i(\alpha)\) and \(Y_t(\alpha)\),
		\(r\le t\le q\), be centered Gaussian fields with
		\[
		\begin{aligned}
			\dE B_t^i(\alpha^1)B_s^j(\alpha^2)
			&=
			\delta_{ij}
			K_r'\bigl(
			t\wedge s\wedge q_{\alpha^1\wedge\alpha^2}
			\bigr),\\
			\dE Y_t(\alpha^1)Y_s(\alpha^2)
			&=
			N\int_r^{t\wedge s\wedge q_{\alpha^1\wedge\alpha^2}}
			v\xi''(v)\,\dd v.
		\end{aligned}
		\]
		The fields \(U,B,Y\) are mutually independent and independent of
		\(G_\bfm\).
		
		For \(\bfm'\in I_N\), set
		\[
		\begin{aligned}
			\mathcal X(\bfm',\alpha)
			&:=
			G_\bfm(\bfm')
			+Y_q(\alpha)+U_\alpha+\Psi_r(\bfm'),\\
			\widetilde{\mathcal X}(\bfm',\alpha)
			&:=
			\sum_{i=1}^N m_i'B_q^i(\alpha)
			+U_\alpha+\Psi_r(\bfm').
		\end{aligned}
		\]
		For two indices \((\bfm'^1,\alpha^1)\) and
		\((\bfm'^2,\alpha^2)\), write
		\[
		Q_{12}:=q_{\alpha^1\wedge\alpha^2}.
		\]
		The difference between the squared increments of the centered Gaussian
		parts of \(\widetilde{\mathcal X}\) and \(\mathcal X\) is
		\[
		2N\bigl[
		K_r(R_{12})-K_r(Q_{12})
		-K_r'(Q_{12})(R_{12}-Q_{12})
		\bigr]
		+N o_\ve(1)+N o_\delta(1)+o_N(N).
		\]
		The term in brackets is nonnegative by
		Assumption~\ref{ass:xi}. Gaussian comparison, first on finite subsets
		and then by continuity, therefore gives
		\[
		\dE_{\bfm,f'}
		\sup_{\bfm'\in I_N,\alpha}\mathcal X(\bfm',\alpha)
		\le
		\dE_{\bfm,f'}
		\sup_{\bfm'\in I_N,\alpha}\widetilde{\mathcal X}(\bfm',\alpha)
		+N o_\ve(1)+N o_\delta(1)+o_N(N).
		\]
		
		\medskip\noindent\emph{Step 3. Fenchel bound.}
		For every \(i\),
		\[
		m_i'
		\left[
		B_q^i(\alpha)+\frac{\xi'(r)}{\xi'(q)}b_i
		\right]
		+\Phi_\zeta^*(q,m_i')
		\le
		\Phi_\zeta\left(
		q,
		B_q^i(\alpha)+\frac{\xi'(r)}{\xi'(q)}b_i
		\right).
		\]
		Enlarging the supremum to the cube and applying the RPC recursion gives
		\[
		\begin{aligned}
			\dE_{\bfm,f'}
			\sup_{\bfm'\in I_N,\alpha}\widetilde{\mathcal X}(\bfm',\alpha)
			\le{}&
			\dE\max_\alpha U_\alpha\\
			&+
			\sum_{i=1}^N
			\dE_Z\phi_\rho\left(
			r,Z+\frac{\xi'(r)}{\xi'(q)}b_i
			\right)
			+o(N).
		\end{aligned}
		\]
		
		\medskip\noindent\emph{Step 4. Evaluation at the original maximizer.}
		Let \(\bfm_*\) be a measurable maximizer of
		\[
		G_\bfm(\bfm')+\Psi_r(\bfm'),
		\qquad \bfm'\in I_N.
		\]
		Since \(\bfm_*\) is independent of \((U,Y)\),
		\[
		\begin{aligned}
			\dE_{\bfm,f'}
			\sup_{\bfm'\in I_N,\alpha}\mathcal X(\bfm',\alpha)
			\ge{}&
			\dE_{\bfm,f'}
			\max_{\bfm'\in I_N}
			\left[
			G_\bfm(\bfm')+\Psi_r(\bfm')
			\right]\\
			&+
			\dE\max_\alpha
			\left[
			U_\alpha+Y_q(\alpha)
			\right].
		\end{aligned}
		\]
		The RPC recursion gives
		\[
		\dE\max_\alpha
		\left[
		U_\alpha+Y_q(\alpha)
		\right]
		=
		\dE\max_\alpha U_\alpha
		+\frac N2
		\int_r^q t\xi''(t)\bar x_\rho(t)\,\dd t.
		\]
		Combining the last three estimates,
		\[
		\begin{aligned}
			\dE_{\bfm,f'}
			\max_{\bfm'\in I_N}
			\left[
			G_\bfm(\bfm')+\Psi_r(\bfm')
			\right]
			\le{}&
			\sum_{i=1}^N
			\dE_Z\phi_\rho\left(
			r,Z+\frac{\xi'(r)}{\xi'(q)}b_i
			\right)\\
			&-
			\frac N2
			\int_r^q t\xi''(t)\bar x_\rho(t)\,\dd t
			+N o_\ve(1)+N o_\delta(1)+o_N(N).
		\end{aligned}
		\]
		Returning to the first display and dividing by \(N\) gives
		\[
		\frac1N\dE_{\bfm,f'}
		\max_{\bfm'\in I_N}
		\left[
		F_{\TAP,\zeta}(\bfm')-F_{\TAP,\zeta}(\bfm)
		\right]
		\le
		\mc J_{N,\bfm,q,r}(\rho)
		+o_\ve(1)+o_\delta(1)+o_N(1).
		\]
		This is \eqref{eq:sibling-guerra-bound}.
	\end{proof}
	
	\begin{remark}
		The upper bound in Proposition~\ref{prop:sibling-guerra-bound} is not
		sharp in general. In the spherical analogue
		\cite[Proposition~3.13]{HuangSellke}, the final cancellation in the proof
		leaves, in the notation of that paper, the term
		\[
		-\frac{
			\left(
			q\xi'(1)-(1+(1-q)z)\xi'(q)
			\right)^2
		}{
			2\xi'(1)(1+(1-q)z)
		}.
		\]
		It is strictly negative unless the square vanishes.
	\end{remark}
	
	Let \(\mu=\Law(X_q)\). The empirical-law window in
	Proposition~\ref{prop:sibling-guerra-bound} gives, for every fixed
	finite-step \(\rho\) and uniformly over its conditional data,
	\[
	\mc J_{N,\bfm,q,r}(\rho)
	=
	\mc J_{q,r}^{\mu}(\rho)+o_\ve(1),
	\]
	where
	\begin{equation}\label{eq:sibling-functional-limit}
		\begin{aligned}
			\mc J_{q,r}^{\mu}(\rho)
			:={}&
			\theta\bigl[\xi(r)-r\xi'(r)-\xi(q)+q\xi'(q)\bigr]
			-\int\Phi_\zeta(q,b)\,\dd\mu(b)\\
			&+\int\dE_Z\,\phi_\rho\left(
			r,Z+\frac{\xi'(r)}{\xi'(q)}b
			\right)\,\dd\mu(b)
			-\frac12\int_r^q t\xi''(t)\bar x_\rho(t)\,\dd t.
		\end{aligned}
	\end{equation}
	Taking \(\rho=\delta_q\) in
	Proposition~\ref{prop:sibling-guerra-bound}, we obtain, uniformly over its
	conditional data,
	\begin{equation}\label{eq:sibling-guerra-bound-limit}
		\frac1N\dE_{\bfm,f'}
		\max_{\bfm'\in\mc C_N(\bfm,r,\delta)}
		\left[
		F_{\TAP,\zeta}(\bfm')-F_{\TAP,\zeta}(\bfm)
		\right]
		\le
		\mc J_{q,r}^{\mu}(\delta_q)
		+o_\ve(1)+o_\delta(1)+o_N(1).
	\end{equation}
	
	%==============================================================================
	\subsection{The replica-symmetric value}
	%==============================================================================
	
	The replica-symmetric choice \(\rho=\delta_q\) gives the strict
	fixed-overlap bound used below.
	
	\begin{lemma}
		Fix a regular \(f\in(f_{\eq},f_1)\), write
		\(\theta=\theta_f\) and \(\zeta=\zeta_f\), and let \(q\) be the first
		positive contact point of \(\zeta\). Let \(X\) be the
		Auffinger--Chen process associated with \(\zeta\), and let
		\(\mu=\Law(X_q)\). For every \(r\in(0,q)\),
		\begin{equation}\label{eq:sibling-rs-value}
			\mc J_{q,r}^{\mu}(\delta_q)
			=
			-\theta H_{\zeta,q}(r).
		\end{equation}
	\end{lemma}
	
	\begin{proof}
		For \(\rho=\delta_q\), we have
		\(\bar x_{\delta_q}(t)=\theta\) on \([r,q)\). Since
		\(\zeta((0,q))=0\), the reduced equation
		\eqref{eq:sibling-reduced-PDE} agrees with the Parisi equation on this
		interval, and therefore
		\[
		\phi_{\delta_q}(t,x)=\Phi_\zeta(t,x),
		\qquad r\le t\le q.
		\]
		The transition density in
		Lemma~\ref{lem:cole-hopf-plateau} gives
		\[
		X_r\mid X_q=b
		\stackrel{d}{=}
		\frac{\xi'(r)}{\xi'(q)}b+Z,
		\qquad
		Z\sim\mc N\left(
		0,
		\xi'(r)-\frac{\xi'(r)^2}{\xi'(q)}
		\right).
		\]
		Hence,
		\[
		\int\dE_Z\Phi_\zeta\left(
		r,
		Z+\frac{\xi'(r)}{\xi'(q)}b
		\right)\,\dd\mu(b)
		=
		\dE\Phi_\zeta(r,X_r).
		\]
		
		Write \(M_t=\partial_x\Phi_\zeta(t,X_t)\).
		Lemma~\ref{lem:AC-identities} gives
		\[
		\dE\Phi_\zeta(q,X_q)-\dE\Phi_\zeta(r,X_r)
		=
		\frac{\theta}{2}
		\int_r^q\xi''(t)\dE[M_t^2]\,\dd t.
		\]
		Substituting this identity into
		\eqref{eq:sibling-functional-limit} and using
		\[
		\xi(r)-r\xi'(r)-\xi(q)+q\xi'(q)
		=
		\int_r^q t\xi''(t)\,\dd t
		\]
		gives
		\[
		\mc J_{q,r}^{\mu}(\delta_q)
		=
		-\frac{\theta}{2}
		\int_r^q\xi''(t)
		\left(\dE[M_t^2]-t\right)\,\dd t
		=
		-\theta H_{\zeta,q}(r).
		\]
	\end{proof}
	
	For the conditional data in
	Proposition~\ref{prop:sibling-guerra-bound},
	\eqref{eq:sibling-guerra-bound-limit} and
	\eqref{eq:sibling-rs-value} give, uniformly,
	\begin{equation}\label{eq:sibling-nonpositive-bound}
		\frac1N\dE_{\bfm,f'}
		\max_{\bfm'\in\mc C_N(\bfm,r,\delta)}
		\left[
		F_{\TAP,\zeta}(\bfm')-F_{\TAP,\zeta}(\bfm)
		\right]
		\le
		-\theta H_{\zeta,q}(r)
		+o_\ve(1)+o_\delta(1)+o_N(1),
	\end{equation}
	where \(H_{\zeta,q}(r)>0\).
	The finite-tolerance conditions in that proposition give
	\[
	\left|
	F_{\TAP}(\bfm)-F_{\TAP,\zeta}(\bfm)
	\right|
	\le2N\ve.
	\]
	Since \(F_{\TAP}\le F_{\TAP,\zeta}\) pointwise, for every \(\bfm'\),
	\[
	F_{\TAP}(\bfm')-F_{\TAP}(\bfm)
	\le
	F_{\TAP,\zeta}(\bfm')-F_{\TAP,\zeta}(\bfm)
	+2N\ve.
	\]
	Thus \eqref{eq:sibling-nonpositive-bound} remains valid with
	\(F_{\TAP}\) inside the maximum after enlarging
	\(o_\ve(1)\) by \(2\ve\).
	
	%==============================================================================
	\subsection{From annealed complexity to existence}
	%==============================================================================
	
	We now carry out the strategy of Subsection~\ref{sub:strategy}. The
	deterministic input is the following spherical-code bound.
	For \(a\in(0,1)\), let \(A_N(a)\) denote the largest cardinality of a set
	\(\mathcal A\) on the sphere of radius \(\sqrt N\) such that
	\[
	\frac1N\langle u,v\rangle\le a
	\qquad
	\text{for all distinct \(u,v\in\mathcal A\)}.
	\]
	
	\begin{lemma}
		\label{lem:one-sided-spherical-code}
		For every \(a\in(0,1)\),
		\begin{equation}\label{eq:one-sided-spherical-code-explicit}
			\limsup_{N\to\infty}
			\frac1N\log A_N(a)
			\le-\frac12\log(1-a).
		\end{equation}
		In particular,
		\begin{equation}\label{eq:one-sided-spherical-code}
			\lim_{a\downarrow0}
			\limsup_{N\to\infty}
			\frac1N\log A_N(a)=0.
		\end{equation}
	\end{lemma}
	
	\begin{proof}
		Set \(x_u=u/\sqrt N\) for \(u\in\mathcal A\). Let \(z\) be uniform on
		\(S^{N-1}\), and define
		\[
		\mathcal A_z
		:=\{u\in\mathcal A:\langle x_u,z\rangle\ge\sqrt a\}.
		\]
		If \(x\in S^{N-1}\) is fixed and
		\[
		p_N(a):=\dP\bigl(\langle x,z\rangle\ge\sqrt a\bigr),
		\]
		then rotational invariance gives
		\[
		\dE_z|\mathcal A_z|=|\mathcal A|p_N(a).
		\]
		Hence some \(z\) satisfies
		\(|\mathcal A_z|\ge|\mathcal A|p_N(a)\).
		
		Let \(P\) be the orthogonal projection onto \(z^\perp\). For distinct
		\(u,v\in\mathcal A_z\),
		\[
		\langle Px_u,Px_v\rangle
		=
		\langle x_u,x_v\rangle
		-\langle x_u,z\rangle\langle x_v,z\rangle
		\le0.
		\]
		By Rankin's theorem~\cite{Rankin1955}, a family of nonzero vectors in an
		\((N-1)\)-dimensional space with pairwise nonpositive inner products has
		cardinality at most \(2(N-1)\). A projection can vanish only when
		\(x_u=\pm z\), and the definition of \(\mathcal A_z\) excludes
		\(x_u=-z\). Thus at most one projection vanishes, and
		\[
		|\mathcal A_z|\le2(N-1)+1\le2N.
		\]
		Since \(|\mathcal A_z|\ge|\mathcal A|p_N(a)\), it follows that
		\[
		A_N(a)\le\frac{2N}{p_N(a)}.
		\]
		
		Since the first coordinate of a uniform point on \(S^{N-1}\) has density
		proportional to \((1-t^2)^{(N-3)/2}\) on \([-1,1]\), with a
		subexponential normalizing constant, Laplace's method gives
		\[
		\lim_{N\to\infty}\frac1N\log p_N(a)
		=\frac12\log(1-a).
		\]
		This proves \eqref{eq:one-sided-spherical-code-explicit}. Letting
		\(a\downarrow0\), and using \(A_N(a)\ge1\), proves
		\eqref{eq:one-sided-spherical-code}.
	\end{proof}
	
	We can now prove the main result of this section.
	
	\begin{proof}[Proof of Proposition~\ref{prop:onshell-lower-bound}]
		Write \(\theta=\theta_f\) and \(\zeta=\zeta_f\).
		Let \(\mathsf N_\ve\) be the SUSY count from
		Definition~\ref{def:susy-state}.
		Proposition~\ref{prop:one-point-conditional-hessian} gives
		\(\delta_0,\kappa>0\). Fix
		\[
		0<\delta<\min\left\{\frac q4,\delta_0\right\}.
		\]
		Fix \(\kappa_0>0\) and choose
		\[
		L>
		\dE\exp\{(8+\kappa_0)|X_q|\}.
		\]
		Let \(\ve_{\mathrm H}(\delta)>0\) and
		\(\ve_{\det}(L)>0\) be thresholds for the
		local spectral-gap and determinant lower-tail estimates.
		
		For a deterministic \(\bfm\) and \(f'\in\dR\), let
		\(\dP_{\bfm,f'}\) and \(\dE_{\bfm,f'}\) denote probability and
		expectation conditioned on
		\[
		F_{\TAP,\zeta}(\bfm)=Nf',\quad
		\nabla F_{\TAP,\zeta}(\bfm)=0.
		\]
		
		\medskip\noindent\emph{Step 1. Stable truncation.}
		The regression and comparison errors are locally uniform in
		\(r\in[\delta,q-\delta]\). Indeed, the covariance coefficients,
		deterministic terms, and reduced Parisi equation in
		Proposition~\ref{prop:sibling-guerra-bound} depend continuously on \(r\)
		and remain uniformly bounded on this interval.
		For each \(r\in[\delta,q-\delta]\),
		\eqref{eq:sibling-nonpositive-bound} and
		\(\theta H_{\zeta,q}(r)\ge8g_\delta\) therefore give
		\(h_r,\ve_r>0\) such that, for all sufficiently large \(N\),
		\[
		\frac1N\dE_{\bfm,f'}
		\max_{\bfm'\in\mc C_N(\bfm,r,h_r)}
		\left[
		F_{\TAP}(\bfm')-F_{\TAP}(\bfm)
		\right]
		\le-6g_\delta
		\]
		uniformly over the one-point windows of width at most
		\(\ve_r\).
		
		Choose \(r_1,\ldots,r_J\) such that the intervals
		\[
		\left(r_j-\frac{h_{r_j}}2,r_j+\frac{h_{r_j}}2\right),
		\qquad 1\le j\le J,
		\]
		cover \([\delta,q-\delta]\). Choose
		\[
		0<\ve_{\mathrm I}
		\le
		\min_{1\le j\le J}
		\left\{\ve_{r_j},h_{r_j}\right\}.
		\]
		If \(0<\ve<\ve_{\mathrm I}\), every point in the maximum
		below belongs to at least one slice
		\(\mc C_N(\bfm,r_j,h_{r_j})\), and the expected maximum on each slice is at
		most \(-6Ng_\delta\). The conditional variance of every Hamiltonian
		increment on the cube is bounded by \(CN\). Gaussian concentration and
		a union bound over the finite cover give \(c_\delta^{\mathrm I}>0\) such
		that
		\begin{equation}\label{eq:sibling-uniform-intermediate}
			\dP_{\bfm,f'}\left(
			\max_{\substack{
					\bfm'\in[-1,1]^N,\\
					|q_{\bfm'}-q|\le\ve,\\
					R(\bfm,\bfm')\in[\delta,q-\delta]
			}}
			\left[F_{\TAP}(\bfm')-F_{\TAP}(\bfm)\right]
			>-2Ng_\delta
			\right)
			\le e^{-c_\delta^{\mathrm I}N}
		\end{equation}
		for \(0<\ve<\ve_{\mathrm I}\), uniformly over the
		one-point Kac--Rice windows.
		
		Choose \(\ve>0\) such that
		\begin{equation}\label{eq:susy-geometric-parameter-choice}
			\ve
			<
			\min\left\{
			\ve_{\mathrm H}(\delta),\ve_{\det}(L),
			\ve_{\mathrm I},
			\delta,g_\delta
			\right\}.
		\end{equation}
		
		For a point \(\bfm\) counted by \(\mathsf N_\ve\), let
		\(\mathcal G_{\bfm,\delta}\) be the event that \(\bfm\) is isolated in
		the sense of Definition~\ref{def:isolated-susy-state}.
		
		If both \(\bfm\) and \(\bfm'\) are counted by
		\(\mathsf N_\ve\), their TAP values differ by at most
		\(2N\ve\). Hence, by
		\eqref{eq:susy-geometric-parameter-choice}, failure of condition
		\textup{(i)} in Definition~\ref{def:isolated-susy-state} is contained
		in the event in \eqref{eq:sibling-uniform-intermediate}. Therefore
		\begin{equation}\label{eq:geometric-intermediate-failure}
			\dP_{\bfm,f'}\left(
			\bfm\text{ fails condition \textup{(i)} in
				Definition~\ref{def:isolated-susy-state}}
			\right)
			\le e^{-c_\delta^{\mathrm I}N}.
		\end{equation}
		By \eqref{eq:one-point-local-spectral-gap}, for all sufficiently
		large \(N\), we use the local spectral-gap estimate only in the form
		\begin{equation}\label{eq:geometric-local-gap-probability}
			\dP_{\bfm,f'}\left(
			\bfm\text{ satisfies condition \textup{(ii)} in
				Definition~\ref{def:isolated-susy-state}}
			\right)
			\ge\frac12.
		\end{equation}
		\eqref{eq:geometric-intermediate-failure} and
		\eqref{eq:geometric-local-gap-probability} give
		\begin{equation}\label{eq:geometric-good-point-probability}
			\dP_{\bfm,f'}\left(
			\mathcal G_{\bfm,\delta}
			\right)
			\ge\frac12-e^{-c_\delta^{\mathrm I}N}.
		\end{equation}
		
		For \(f'\in[f-\ve,f+\ve]\), let
		\[
		\mathcal W_{\ve,f'}
		:=
		\left\{
		\begin{array}{l}
			\bfm\in(-1,1)^N:\ |q_\bfm-q|\le\ve,\\[1mm]
			W_2\left(N^{-1}\sum_{i=1}^N\delta_{b_i},\Law(X_q)\right)
			\le\ve,\\[1mm]
			\left|f'+N^{-1}\left(F_{\TAP}(\bfm)
			-F_{\TAP,\zeta}(\bfm)\right)-f\right|\le\ve
		\end{array}
		\right\}.
		\]
		Set
		\[
		\mathcal W_{\ve,f'}^{(L)}
		:=
		\left\{
		\bfm\in\mathcal W_{\ve,f'}:
		\frac1N\sum_{i=1}^N
		\exp\{(8+\kappa_0)|b_i|\}\le L
		\right\}.
		\]
		By the Kac--Rice formula,
		\[
		\begin{aligned}
			\dE\widehat{\mathsf N}_{\ve,\delta,\kappa}
			\ge{}&
			N\int_{f-\ve}^{f+\ve}
			\int_{\mathcal W_{\ve,f'}^{(L)}}\\
			&\quad\dE_{\bfm,f'}\!\left[
			\left|\det\nabla^2F_{\TAP,\zeta}(\bfm)\right|
			\boldsymbol 1_{\mathcal G_{\bfm,\delta}}
			\right]\\
			&\quad\times
			p_{\left(
				\nabla F_{\TAP,\zeta}(\bfm),
				F_{\TAP,\zeta}(\bfm)
				\right)}(0,Nf')
			\,\dd\bfm\,\dd f'.
		\end{aligned}
		\]
		
		Set
		\[
		D_\bfm
		:=
		\det\nabla^2F_{\TAP,\zeta}(\bfm).
		\]
		Let \(r_{N,\ve}=o_N(N)\) be the uniform remainder from
		Lemma~\ref{lem:one-point-conditional-determinant}\textup{(ii)}, and let
		\[
		\mathcal D_\bfm
		:=
		\left\{
		|D_\bfm|
		\ge
		\exp\{-N\rho_{\det}(\ve)-r_{N,\ve}\}
		\dE_{\bfm,f'}|D_\bfm|
		\right\}.
		\]
		Lemma~\ref{lem:one-point-conditional-determinant}\textup{(ii)} gives
		\[
		\dP_{\bfm,f'}(\mathcal D_\bfm^{\,c})
		\le e^{-c_LN}
		\]
		uniformly on \(\mathcal W_{\ve,f'}^{(L)}\). Together with
		\eqref{eq:geometric-good-point-probability}, this gives
		\[
		\dP_{\bfm,f'}\left(
		\mathcal G_{\bfm,\delta}\cap\mathcal D_\bfm
		\right)
		\ge
		\frac12-e^{-c_\delta^{\mathrm I}N}-e^{-c_LN}
		\ge\frac14
		\]
		for all sufficiently large \(N\). Therefore
		\begin{equation}\label{eq:geometric-local-stability-input}
			\dE_{\bfm,f'}\left[
			|D_\bfm|
			\boldsymbol 1_{\mathcal G_{\bfm,\delta}}
			\right]
			\ge
			\frac14
			\exp\{-N\rho_{\det}(\ve)-r_{N,\ve}\}
			\dE_{\bfm,f'}|D_\bfm|
		\end{equation}
		uniformly on \(\mathcal W_{\ve,f'}^{(L)}\).
		
		The Auffinger--Chen drift is bounded, so \(X_q\) has all exponential
		moments. With \(L\) chosen above,
		the law of large numbers shows that the cutoff has probability tending to
		one under the product measure used in the lower-bound proof of
		Proposition~\ref{prop:susy-annealed}. Hence the restricted
		Kac--Rice integral has exponent \(\Sigma(f)\). Together with
		\eqref{eq:geometric-local-stability-input} and
		\(\rho_{\det}(\ve)\to0\), this gives the lower bound. The upper
		bound follows from
		\(\widehat{\mathsf N}_{\ve,\delta,\kappa}\le\mathsf N_\ve\) and
		Proposition~\ref{prop:susy-annealed}. Therefore
		\begin{equation}\label{eq:geometric-truncated-first-moment}
			\lim_{\ve\downarrow0}
			\liminf_{N\to\infty}
			\frac1N
			\log
			\dE\widehat{\mathsf N}_{\ve,\delta,\kappa}
			=
			\Sigma(f).
		\end{equation}
		
		\medskip\noindent\emph{Step 2. Exclusion at high overlap.}
		A point \(\bfm\) counted by
		\(\widehat{\mathsf N}_{\ve,\delta,\kappa}\) has no distinct competitor
		\(\bfm'\) counted by \(\mathsf N_\ve\) with
		\[
		R(\bfm,\bfm')\ge q-\delta.
		\]
		To prove this, suppose that such a competitor exists. Since
		\(\ve\le\delta\),
		\[
		\begin{aligned}
			\frac1N\|\bfm'-\bfm\|_2^2
			&=
			q_\bfm+q_{\bfm'}-2R(\bfm,\bfm')\\
			&\le
			2(q+\ve)-2(q-\delta)
			\le4\delta.
		\end{aligned}
		\]
		With \(v=\bfm'-\bfm\), the segment
		\[
		\{\bfm+tv:0\le t\le1\}
		\]
		is contained in the ball appearing in condition \textup{(ii)} of
		Definition~\ref{def:isolated-susy-state}.
		
		Since both endpoints are critical,
		\[
		\nabla F_{\TAP,\zeta}(\bfm)
		=
		\nabla F_{\TAP,\zeta}(\bfm')
		=0.
		\]
		Taylor's formula for the gradient gives
		\[
		\begin{aligned}
			0
			&=
			\left\langle
			v,
			\nabla F_{\TAP,\zeta}(\bfm')
			-
			\nabla F_{\TAP,\zeta}(\bfm)
			\right\rangle\\
			&=
			\int_0^1
			v^{\mathsf T}
			\nabla^2F_{\TAP,\zeta}(\bfm+tv)
			v\,\dd t\\
			&\le
			-\kappa\|v\|_2^2.
		\end{aligned}
		\]
		It follows that \(v=0\), contradicting \(\bfm'\ne\bfm\).
		
		\medskip\noindent\emph{Step 3. Spherical-code bound.}
		Let \(\bfm\) and \(\bfm'\) be two distinct points counted by
		\(\widehat{\mathsf N}_{\ve,\delta,\kappa}\). By condition
		\textup{(i)} of Definition~\ref{def:isolated-susy-state}, their overlap
		cannot belong to
		\([\delta,q-\delta]\), and Step~2 excludes overlaps at least
		\(q-\delta\). Therefore
		\[
		R(\bfm,\bfm')<\delta.
		\]
		
		Normalize the retained points by setting
		\[
		u_\bfm:=\frac{\bfm}{\sqrt{q_\bfm}}.
		\]
		Then
		\[
		\|u_\bfm\|_2=\sqrt N,
		\]
		and, for two distinct retained points,
		\[
		\frac1N
		\langle u_\bfm,u_{\bfm'}\rangle
		=
		\frac{R(\bfm,\bfm')}
		{\sqrt{q_\bfm q_{\bfm'}}}
		\le
		\frac{\delta}{q-\ve}.
		\]
		Thus the normalized retained points form a code to which
		Lemma~\ref{lem:one-sided-spherical-code} applies, and
		\begin{equation}\label{eq:geometric-code-count}
			\widehat{\mathsf N}_{\ve,\delta,\kappa}
			\le
			A_N\left(
			\frac{\delta}{q-\ve}
			\right).
		\end{equation}
		
		\medskip\noindent\emph{Step 4. Existence probability.}
		By \eqref{eq:geometric-code-count},
		\[
		\widehat{\mathsf N}_{\ve,\delta,\kappa}
		\le
		A_N\left(
		\frac{\delta}{q-\ve}
		\right)
		\boldsymbol 1_{
			\{\widehat{\mathsf N}_{\ve,\delta,\kappa}\ge1\}
		}.
		\]
		Taking expectations gives
		\[
		\dP\left(
		\widehat{\mathsf N}_{\ve,\delta,\kappa}\ge1
		\right)
		\ge
		\frac{
			\dE\widehat{\mathsf N}_{\ve,\delta,\kappa}
		}{
			A_N\left(\delta/(q-\ve)\right)
		}.
		\]
		Since
		\(\mathsf N_\ve
		\ge\widehat{\mathsf N}_{\ve,\delta,\kappa}\),
		\[
		\dP\bigl(\mc E_{N,\ve}(f)\bigr)
		\ge
		\dP\left(
		\widehat{\mathsf N}_{\ve,\delta,\kappa}\ge1
		\right).
		\]
		Lemma~\ref{lem:one-sided-spherical-code} now gives
		\[
		\begin{aligned}
			\lim_{\ve\downarrow0}
			\liminf_{N\to\infty}
			\frac1N
			\log
			\dP\bigl(\mc E_{N,\ve}(f)\bigr)
			&\ge
			\lim_{\ve\downarrow0}
			\left[
			\liminf_{N\to\infty}
			\frac1N
			\log\dE\widehat{\mathsf N}_{\ve,\delta,\kappa}
			+
			\frac12\log\left(
			1-\frac{\delta}{q-\ve}
			\right)
			\right]\\
			&=
			\Sigma(f)
			+
			\frac12\log\left(1-\frac{\delta}{q}\right),
		\end{aligned}
		\]
		where the equality follows from
		\eqref{eq:geometric-truncated-first-moment}. Letting \(\delta\downarrow0\)
		gives
		\[
		\lim_{\ve\downarrow0}
		\liminf_{N\to\infty}
		\frac1N
		\log
		\dP\bigl(\mc E_{N,\ve}(f)\bigr)
		\ge
		\Sigma(f),
		\]
		which proves Proposition~\ref{prop:onshell-lower-bound}.
	\end{proof}
	
	\subsection{Proof of Theorem~\ref{thm:legendre}}
	
	We combine the upper and lower bounds obtained above.
	
	\begin{proof}
		Set
		\[
		M_N:=\max_{\bfm\in[-1,1]^N}F_{\TAP}(\bfm).
		\]
		Proposition~\ref{prop:direct-tilt-TAP} gives, for every
		\(s\in(0,1)\),
		\[
		\limsup_{N\to\infty}
		\frac1N\log\dE e^{sM_N}
		\le\Lambda(s).
		\]
		
		For every \(\ve>0\) and \(s\in(0,1)\), Markov's inequality gives
		\[
		\dP\bigl(\mc E_{N,\ve}(f)\bigr)
		\le
		\dP\bigl(M_N\ge N(f-\ve)\bigr)
		\le
		e^{-Ns(f-\ve)}\dE e^{sM_N}.
		\]
		Therefore,
		\[
		\lim_{\ve\downarrow0}
		\limsup_{N\to\infty}
		\frac1N\log\dP\bigl(\mc E_{N,\ve}(f)\bigr)
		\le
		\inf_{s\in(0,1)}
		\bigl[\Lambda(s)-sf\bigr]
		=
		\Sigma(f).
		\]
		Proposition~\ref{prop:onshell-lower-bound} gives the reverse
		inequality and proves~\eqref{eq:existence-tap}.
		
		For every \(\ve>0\),
		\[
		\dE e^{\theta_fM_N}
		\ge
		e^{N\theta_f(f-\ve)}
		\dP\bigl(\mc E_{N,\ve}(f)\bigr).
		\]
		Since \(\theta_f\) minimizes
		\(\theta\mapsto\Lambda(\theta)-\theta f\),
		\[
		\theta_f f+\Sigma(f)=\Lambda(\theta_f).
		\]
		Thus \eqref{eq:existence-tap} gives
		\[
		\liminf_{N\to\infty}
		\frac1N\log\dE e^{\theta_fM_N}
		\ge\Lambda(\theta_f).
		\]
		The reverse inequality follows from
		Proposition~\ref{prop:direct-tilt-TAP}, proving~\eqref{eq:legendre}.
	\end{proof}
	
	\subsection{Free-energy large-deviation lower bound}
	
	We conclude by transferring the TAP-state lower bound to the ordinary free
	energy.
	
	\begin{proposition}[Free-energy large-deviation lower bound]
		\label{prop:free-energy-ldp-lower-bound}
		Let \(f\in(f_{\eq},f_1)\) be regular. Let \(q\) be the first
		positive contact point of \(\zeta_f\), and let \(X\) be the
		Auffinger--Chen process
		associated with \(\zeta_f\). Assume that
		\[
		1-\xi''(q)\dE\left[
		\left(
		\partial_{xx}\Phi_{\zeta_f}(q,X_q)
		\right)^2
		\right]>0.
		\]
		Then
		\[
		\lim_{\rho\downarrow0}\liminf_{N\to\infty}
		\frac1N\log\dP\left(
		\frac1N\log Z_N\ge f-\rho
		\right)
		\ge \Sigma(f).
		\]
	\end{proposition}
	
	\begin{proof}
		Recall \(U_N^{k,\delta}\) from
		\eqref{eq:replicated-band-free-energy}. The quantitative uniform band
		estimate of Chen, Panchenko, and
		Subag~\cite[Theorem~1]{ChenPanchenkoSubag} implies that, for every
		\(c,t>0\), one can choose \(k\) sufficiently large and \(\delta>0\)
		sufficiently small so that, for all sufficiently large \(N\),
		\begin{equation}\label{eq:CPS-uniform-band-bound}
			\dP\left(
			U_N^{k,\delta}(\bfm)
			\ge
			\frac{F_{\TAP}(\bfm)-H_N(\bfm)}N-t
			\quad\text{for every }\bfm\in[-1,1]^N
			\right)
			\ge1-e^{-cN}.
		\end{equation}
		The normalization in \eqref{eq:replicated-band-free-energy} subtracts
		\(\log2\) from the band free energy. Our terminal condition
		\(\log\cosh x\), in place of \(\log(2\cosh x)\), subtracts the same
		constant from the TAP correction. Hence no additional constant appears
		in~\eqref{eq:CPS-uniform-band-bound}.
		
		Fix \(\gamma,\rho>0\). By Proposition~\ref{prop:onshell-lower-bound},
		we may choose
		\(0<\ve<\rho/2\) such that
		\[
		\liminf_{N\to\infty}\frac1N
		\log\dP\bigl(\mc E_{N,\ve}(f)\bigr)
		\ge\Sigma(f)-\gamma.
		\]
		Consequently, for all sufficiently large \(N\),
		\[
		\dP\bigl(\mc E_{N,\ve}(f)\bigr)
		\ge \exp\bigl\{N(\Sigma(f)-2\gamma)\bigr\}.
		\]
		Choose \(c>-\Sigma(f)+2\gamma\) and apply
		\eqref{eq:CPS-uniform-band-bound} with \(t=\rho/2\). The probability
		that \(\mc E_{N,\ve}(f)\) and the uniform bound both hold is then at
		least
		\[
		\frac12
		\exp\bigl\{N(\Sigma(f)-2\gamma)\bigr\}
		\]
		for all sufficiently large \(N\).
		
		On this event, let \(\bfm\) be a witness for
		\(\mc E_{N,\ve}(f)\). Since the band estimate is uniform in \(\bfm\),
		no measurable choice is needed. Restricting the replicated partition
		function to \(\Band_{k,\delta}^{\Ising}(\bfm)\) gives
		\begin{align*}
			\frac1N\log Z_N
			&\ge
			\frac1N H_N(\bfm)+U_N^{k,\delta}(\bfm)\\
			&\ge
			\frac1N F_{\TAP}(\bfm)-\frac\rho2\\
			&\ge f-\ve-\frac\rho2
			\ge f-\rho.
		\end{align*}
		Therefore
		\[
		\liminf_{N\to\infty}\frac1N\log\dP\left(
		\frac1N\log Z_N\ge f-\rho
		\right)
		\ge\Sigma(f)-2\gamma.
		\]
		Finally let \(\gamma\downarrow0\) and then \(\rho\downarrow0\).
	\end{proof}
	
	%==============================================================================
	\appendix
	
	\section{The Parisi PDE and its Auffinger--Chen flow}\label{section:Auffinger}
	%==============================================================================
	
	We collect the Parisi-flow identities used in the Kac--Rice and band
	arguments. All statements in this appendix are quoted without proof. They
	are standard consequences of the Auffinger--Chen stochastic control
	representation of the Parisi PDE and of the generalized TAP variational
	calculus, as recalled in
	\cite{AuffingerChen,AuffingerChenControl,ChenPanchenkoSubag} and
	\cite[Appendix~B]{TAPIsing}.
	Throughout, if $\zeta\in\mc P([q,1])$, we keep the convention used in the
	body of the paper and write $\zeta([0,t])$ for $\zeta([q,t])$ when $t\ge q$.
	The same convention is used for signed perturbations of band measures. We
	use the cumulant notation of \eqref{eq:order-parameter-cumulants} with this
	convention.
	
	%------------------------------------------------------------------------------
	\subsection{The Auffinger--Chen flow}
	%------------------------------------------------------------------------------
	
	For $\zeta\in\mc P([0,1])$, the Parisi PDE is
	\begin{equation}\label{eq:ParisiPDE}
		\partial_t\Phi_\zeta(t,x)
		+\frac{\xi''(t)}2
		\left(
		\partial_{xx}\Phi_\zeta(t,x)
		+\zeta([0,t])\bigl(\partial_x\Phi_\zeta(t,x)\bigr)^2
		\right)=0,
		\qquad
		\Phi_\zeta(1,x)=\log\cosh x .
	\end{equation}
	The Parisi flow satisfies
	\[
	0\le\partial_{xx}\Phi_\zeta(t,x)\le1.
	\]
	The lower and upper bounds follow from
	\cite[Lemmas~B.5 and~B.2]{TAPIsing}, respectively. On compact subsets of
	\([0,1)\times\mathbb R\), the spatial derivatives of
	\(\Phi_\zeta\) are bounded and continuous. Their \(t\)-derivatives have
	one-sided limits at atoms of \(\zeta\). These are standard interior
	regularity properties of \eqref{eq:ParisiPDE}.
	Given $s\in[0,1]$ and $x\in\dR$, the associated Auffinger--Chen process on $[s,1]$ is the solution of
	\begin{equation}\label{eq:AuffingerChen}
		\dd X_t
		=\zeta([0,t])\xi''(t)\partial_x\Phi_\zeta(t,X_t)\,\dd t
		+\sqrt{\xi''(t)}\,\dd B_t,
		\qquad X_s=x .
	\end{equation}
	We write
	\[
	M_t:=\partial_x\Phi_\zeta(t,X_t).
	\]
	When the initial condition is clear, it is suppressed from the notation.
	
	The following identities combine
	\cite[Lemmas~B.1 and~B.4]{TAPIsing}.
	\begin{lemma}[Auffinger--Chen identities]\label{lem:AC-identities}
		For the process~\eqref{eq:AuffingerChen}, the following identities hold.
		\begin{enumerate}[label=\textup{(\roman*)}]
			\item $M_t$ is a martingale and
			\begin{equation}\label{eq:AC-martingale-differential}
				\dd M_t
				=\sqrt{\xi''(t)}\,\partial_{xx}\Phi_\zeta(t,X_t)\,\dd B_t .
			\end{equation}
			
			\item The Parisi value along the optimal trajectory satisfies
			\begin{equation}\label{eq:AC-ito-phi}
				\dd\Phi_\zeta(t,X_t)
				=\frac12\zeta([0,t])\xi''(t)M_t^2\,\dd t
				+\sqrt{\xi''(t)}M_t\,\dd B_t .
			\end{equation}
			Equivalently, for $s<t$,
			\begin{equation}\label{eq:AC-ito-phi-integrated}
				\Phi_\zeta(t,X_t)-\Phi_\zeta(s,X_s)
				=\frac12\int_s^t\zeta([0,r])\xi''(r)M_r^2\,\dd r
				+\int_s^t\sqrt{\xi''(r)}M_r\,\dd B_r .
			\end{equation}
			
			\item For $s<t$,
			\begin{equation}\label{eq:AC-basic-ward-a}
				\dE\bigl[(X_t-X_s)M_s\bigr]
				=\dE[M_s^2]\int_s^t\zeta([0,r])\xi''(r)\,\dd r,
			\end{equation}
			and
			\begin{equation}\label{eq:AC-basic-ward-b}
				\dE\bigl[(M_t-M_s)X_s\bigr]=0 .
			\end{equation}
			
			\item Suppose that
			\(\dE[M_r^2]=r\) for \(\zeta\)-a.e.\ \(r\in[0,1]\).
			If \(s,t\in\supp(\zeta)\cap[0,1)\) and
			\(\zeta((s,t))=0\), then
			\begin{equation}\label{eq:AC-deltaM-deltaX}
				\dE\bigl[(M_t-M_s)(X_t-X_s)\bigr]
				=\bigl(\xi'(t)-\xi'(s)\bigr)\int_s^1\zeta([0,r])\,\dd r .
			\end{equation}
			Consequently,
			\begin{equation}\label{eq:AC-ward-contact-form}
				\dE\bigl[M_t(X_t-X_s)\bigr]
				=\bigl(\xi'(t)-\xi'(s)\bigr)
				\left(\zeta([0,s])\,t+\widehat\zeta(t)\right),
			\end{equation}
			where in the last display we used $\zeta((s,t))=0$.
		\end{enumerate}
	\end{lemma}
	
	The next identity is \cite[Lemma~B.2]{TAPIsing}.
	\begin{lemma}\label{lem:AC-susceptibility}
		Let \(\zeta\in\mc P([0,1])\), fix \(q\in[0,1]\), and let \(X\) be
		the associated Auffinger--Chen process started at \(X_0=0\). Write
		\(M_t=\partial_x\Phi_\zeta(t,X_t)\). If
		\[
		\dE M_q^2=q,
		\qquad
		\dE M_t^2=t
		\quad\text{for \(\zeta\)-a.e. \(t\in[q,1]\)},
		\]
		then
		\[
		\dE\,\partial_{xx}\Phi_\zeta(q,X_q)
		=
		\widehat\zeta(q).
		\]
	\end{lemma}
	
	%------------------------------------------------------------------------------
	\subsection{Constant-mass intervals and Cole--Hopf}
	%------------------------------------------------------------------------------
	
	The following plateau formula is \cite[Lemma~B.3]{TAPIsing}.
	\begin{lemma}[Cole--Hopf identity on a plateau]\label{lem:cole-hopf-plateau}
		Let $0\le s<t\le1$, assume $\zeta((s,t))=0$, and set $\alpha:=\zeta([0,s])$.  Then
		\[
		U(r,x):=\exp\{\alpha\Phi_\zeta(r,x)\}
		\]
		solves the backward heat equation on $[s,t]$,
		\[
		\partial_r U(r,x)+\frac12\xi''(r)\partial_{xx}U(r,x)=0 .
		\]
		Hence, with
		\[
		p^{\mathrm{BM}}_{s,t}(x,y)
		:=\frac{1}{\sqrt{2\pi(\xi'(t)-\xi'(s))}}
		\exp\left\{-\frac{(y-x)^2}{2(\xi'(t)-\xi'(s))}\right\},
		\]
		we have
		\begin{equation}\label{eq:cole-hopf-integral}
			e^{\alpha\Phi_\zeta(s,x)}
			=\int_\dR e^{\alpha\Phi_\zeta(t,y)}p^{\mathrm{BM}}_{s,t}(x,y)\,\dd y .
		\end{equation}
		Equivalently, conditionally on $X_s=x$, the transition density of the Auffinger--Chen process on the plateau is
		\begin{equation}\label{eq:AC-plateau-transition}
			p_{s,t}(x,y)
			=\frac{e^{\alpha\Phi_\zeta(t,y)}}{e^{\alpha\Phi_\zeta(s,x)}}
			p^{\mathrm{BM}}_{s,t}(x,y).
		\end{equation}
	\end{lemma}
	
	In the parent-band computation this is applied with $s=q_0$, $t=q$, $\alpha=u=\zeta([0,q_0])$, and $\xi'(q)-\xi'(q_0)=\xi_0'(q)$.  Thus
	\begin{equation}\label{eq:band-cole-hopf-collapse}
		(2\pi\xi_0'(q))^{-1/2}
		\int_\dR
		\exp\left\{
		u\Phi_\zeta(q,c)-\frac{(c-b)^2}{2\xi_0'(q)}
		\right\}\,\dd c
		=e^{u\Phi_\zeta(q_0,b)} .
	\end{equation}
	The corresponding tilted one-coordinate channel is
	\begin{equation}\label{eq:band-AC-channel}
		\nu_b(\dd c)
		=\frac{
			\exp\left\{
			u\Phi_\zeta(q,c)-\frac{(c-b)^2}{2\xi_0'(q)}
			\right\}\,\dd c
		}{
			\int_\dR
			\exp\left\{
			u\Phi_\zeta(q,y)-\frac{(y-b)^2}{2\xi_0'(q)}
			\right\}\,\dd y
		}.
	\end{equation}
	If $c\sim\nu_b$, $m=\partial_x\Phi_\zeta(q,c)$, and
	$a=\partial_x\Phi_\zeta(q_0,b)$, then
	\begin{equation}\label{eq:band-AC-martingale-channel}
		\dE[m\mid b]=a .
	\end{equation}

	%------------------------------------------------------------------------------
	\subsection{First variation of the band TAP correction}
	%------------------------------------------------------------------------------
	
	For $q\in[0,1)$, $\mu\in\mc P((-1,1))$ with $\int a^2\,\mu(\dd a)=q$, and $\zeta\in\mc P([q,1])$, recall the band correction
	\begin{equation}\label{eq:band-correction-V}
		V_\zeta(q,\mu)
		:=\int\Phi_\zeta^*(q,a)\,\mu(\dd a)
		-\frac12\int_q^1 r\xi''(r)\zeta([0,r])\,\dd r .
	\end{equation}
	Let $X^\mu$ be the Auffinger--Chen process on $[q,1]$ with the law-matching condition
	\begin{equation}\label{eq:AC-law-matching}
		\partial_x\Phi_\zeta(q,X_q^\mu)\sim\mu,
	\end{equation}
	and write $M_r^\mu:=\partial_x\Phi_\zeta(r,X_r^\mu)$.
	
	\begin{lemma}[First variation and contact equations]\label{lem:first-variation-contact}
		The map $\zeta\mapsto V_\zeta(q,\mu)$ is strictly convex.  If $\kappa$ is a finite signed measure on $[q,1]$ with $\kappa([q,1])=0$, then
		\begin{equation}\label{eq:TAP-first-variation}
			D_\zeta V_\zeta(q,\mu)[\kappa]
			=\frac12\int_q^1\xi''(r)
			\left(\dE[(M_r^\mu)^2]-r\right)
			\kappa([0,r])\,\dd r .
		\end{equation}
		Define the obstacle
		\begin{equation}\label{eq:TAP-obstacle}
			H_\zeta^\mu(s)
			:=\frac12\int_s^1\xi''(r)
			\left(\dE[(M_r^\mu)^2]-r\right)\,\dd r,
			\qquad s\in[q,1].
		\end{equation}
		The unique minimizer $\zeta_\mu$ of $V_\zeta(q,\mu)$ over $\mc P([q,1])$ is characterized by
		\begin{equation}\label{eq:TAP-support-obstacle}
			\supp(\zeta_\mu)
			\subset
			\operatorname*{argmin}_{s\in[q,1]}H_{\zeta_\mu}^\mu(s).
		\end{equation}
		Moreover, $\zeta_\mu(\{1\})=0$. At every interior contact point
		$s\in\supp(\zeta_\mu)\cap[q,1)$, and at $s=q$ by the law matching when
		$\int a^2\mu(\dd a)=q$,
		\begin{equation}\label{eq:TAP-contact-calibration}
			\dE[(M_s^\mu)^2]=s .
		\end{equation}
		For $q=0$ and $\mu=\delta_0$, this reduces to the usual first-variation and contact equations for the Parisi functional $P_\zeta$.
	\end{lemma}
	
	%------------------------------------------------------------------------------
	\subsection{Ward identities in the parent band}
	%------------------------------------------------------------------------------
	
	We record the identities needed for the conditional construction in a band
	above a critical TAP ancestor. Let \(0\le q_0<q<1\), assume that
	\(\zeta((q_0,q))=0\), and put
	\[
	u:=\zeta([0,q_0]),
	\qquad
	\xi_0'(q):=\xi'(q)-\xi'(q_0).
	\]
	Let $X_{q_0}=b$, $X_q=c$, and
	\[
	a:=\partial_x\Phi_\zeta(q_0,b),
	\qquad
	m:=\partial_x\Phi_\zeta(q,c),
	\qquad
	M_r:=\partial_x\Phi_\zeta(r,X_r)\quad(q_0\le r\le1).
	\]
	Assume the endpoint contact equations and the contact calibration above \(q\):
	\[
	\dE[a^2]=q_0,
	\qquad
	\dE[m^2]=q,
	\qquad
	\dE[M_r^2]=r
	\quad\text{for \(\zeta\)-a.e.\ \(r\in[q,1]\)}.
	\]
	Under the channel~\eqref{eq:band-AC-channel}, the martingale property and
	the endpoint equations, together with \eqref{eq:AC-basic-ward-a}, give
	\begin{equation}\label{eq:band-channel-self-overlap}
		\dE[m\mid b]=a,
		\qquad
		\dE[a^2]=q_0,
		\qquad
		\dE[m^2]=q,
		\qquad
		\dE[a(c-b)]=u q_0\xi_0'(q).
	\end{equation}
	Integration by parts in \eqref{eq:band-AC-channel} and
	Lemma~\ref{lem:AC-susceptibility} give
	\begin{equation}\label{eq:band-channel-ward1}
		\dE\bigl[m(c-b)\bigr]
		=\xi_0'(q)\left(
		u\dE[m^2]
		+\dE\left[\partial_{xx}\Phi_\zeta(q,c)\right]
		\right)
		=\xi_0'(q)\left(uq+\widehat\zeta(q)\right).
	\end{equation}
	Taking expectations in \eqref{eq:AC-ito-phi-integrated} gives
	\begin{equation}\label{eq:band-channel-potential-increment}
		\dE\bigl[\Phi_\zeta(q,c)-\Phi_\zeta(q_0,b)\bigr]
		=\frac12\int_{q_0}^q u\,\xi''(r)\dE[M_r^2]\,\dd r .
	\end{equation}
	When \(q_0=0\), \eqref{eq:band-channel-ward1} is the limiting identity
	behind \eqref{eq:ising-ward-inner}, while
	\eqref{eq:band-channel-potential-increment} is the corresponding
	one-point potential identity.
	
	%==============================================================================
	\section{Supersymmetric annealed complexity and Hessian estimates}
	\label{sec:one-point-susy-computations}
	%==============================================================================
	
	%==============================================================================
	\subsection{Kac--Rice formula}
	%==============================================================================
	
	We prove Proposition~\ref{prop:susy-annealed}.
	Fix a regular \(f\in(f_{\eq},f_1)\), let
	\(\theta=\theta_f\), \(\zeta=\zeta_f\), and let \(q\) be the first
	positive contact point of \(\zeta\). Let \(X\) be the corresponding
	Auffinger--Chen process. Thus
	\[
	\zeta(\{0\})=\theta,
	\qquad
	\zeta((0,q))=0,
	\qquad
	F_{\TAP,\zeta}(0)=NP_\zeta.
	\]
	Assume that
	\[
	1-\xi''(q)\dE\left[
	\left(
	\partial_{xx}\Phi_\zeta(q,X_q)
	\right)^2
	\right]>0.
	\]
	For \(\bfm\in(-1,1)^N\), define \(b_i\) by
	\[
	m_i=\partial_x\Phi_\zeta(q,b_i),
	\]
	and write \(\bfb=(b_i)_{i\le N}\).
	
	For \(\ve>0\) and \(f'\in[f-\ve,f+\ve]\), define
	\[
	\mathcal D_{N,\ve}(f')
	:=
	\left\{
	\begin{array}{l}
		\bfm\in(-1,1)^N:\ |q_\bfm-q|\le\ve,\\[1mm]
		W_2\left(
		N^{-1}\sum_{i=1}^N\delta_{b_i},
		\Law(X_q)
		\right)\le\ve,\\[1mm]
		\left|
		f'
		+N^{-1}\left(
		F_{\TAP}(\bfm)-F_{\TAP,\zeta}(\bfm)
		\right)
		-f
		\right|\le\ve
	\end{array}
	\right\}.
	\]
	On the level set \(F_{\TAP,\zeta}(\bfm)=Nf'\), the last condition is
	\[
	\left|N^{-1}F_{\TAP}(\bfm)-f\right|\le\ve.
	\]
	
	For \(v\in\dR^N\) and \(w\in\dR\), let
	\(\varphi_\bfm(v,w)\) denote the density of
	\[
	\left(
	\nabla F_{\TAP,\zeta}(\bfm),
	F_{\TAP,\zeta}(\bfm)-NP_\zeta
	\right)
	\]
	at \((v,w)\).
	
	The Kac--Rice formula gives
	\begin{multline}\label{eq:one-point-kac-rice-window}
		\dE\mathsf N_\ve
		=
		\int_{f-\ve}^{f+\ve}
		\int_{\bfm\in\mathcal D_{N,\ve}(f')}
		\dE\!\left[
		|\det\nabla^2 F_{\TAP,\zeta}(\bfm)|
		\,\middle|\,
		\substack{
			\nabla F_{\TAP,\zeta}(\bfm)=0,\\
			F_{\TAP,\zeta}(\bfm)=Nf'}
		\right]\\
		\times
		\varphi_\bfm\bigl(0,N(f'-P_\zeta)\bigr)
		\,N\,\dd\bfm\,\dd f'.
	\end{multline}
	
	%==============================================================================
	\subsection{Conditional Hessian estimates}
	%==============================================================================
	
	We prove the conditional Hessian and determinant estimates used in
	Section~\ref{sec:onshell-lb}.
	
	\begin{proposition}\label{prop:one-point-conditional-hessian}
		Let \(f\in(f_{\eq},f_1)\) be regular, let \(\zeta=\zeta_f\), and let
		\(q\) be the first positive contact point of \(\zeta\). Let \(X\) be the
		corresponding Auffinger--Chen process. Suppose that
		\[
		1-\xi''(q)\dE\left[
		\left(
		\partial_{xx}\Phi_\zeta(q,X_q)
		\right)^2
		\right]>0.
		\]
		Let \(\bfm\in (-1,1)^N\) satisfy \(q_\bfm=q\), and define \(b_i\) by
		\(m_i=\partial_x\Phi_\zeta(q,b_i)\). Assume that
		\[
		\frac1N\sum_{i=1}^N\delta_{b_i}
		\xrightarrow{W_2}\Law(X_q).
		\]
		Let \(f_N\to f\) and condition on
		\[
		\nabla F_{\TAP,\zeta}(\bfm)=0,
		\qquad
		F_{\TAP,\zeta}(\bfm)=Nf_N.
		\]
		Write \(\dP_{\bfm,f_N}\) and \(\dE_{\bfm,f_N}\) for the corresponding
		conditional probability and expectation. For
		\(\sigma\in\{-,+\}\), the following conclusions hold.
		\begin{enumerate}[label=\textup{(\roman*)}]
			\item There is \(c>0\) such that
			\begin{equation}
				\dP_{\bfm,f_N}\left(
				\nabla_\sigma^2F_{\TAP,\zeta}(\bfm)\preceq-cI_N
				\right)\longrightarrow1.
			\end{equation}
			
			\item
			\begingroup
			There are \(\delta_0,\kappa>0\) such that, for every
			\(0<\delta<\delta_0\),
			\begin{equation}\label{eq:one-point-local-spectral-gap}
				\dP_{\bfm,f_N}\left(
				\sup_{\substack{
						x\in(-1,1)^N,\\
						\|x-\bfm\|_2\le2\sqrt{\delta N}
				}}
				\lambda_{\max}\!\left(
				\nabla^2F_{\TAP,\zeta}(x)
				\right)
				\le-\kappa
				\right)\longrightarrow1.
			\end{equation}
			If \(q_x\) is an atom of \(\zeta\), the two one-sided Hessians may
			differ. In \eqref{eq:one-point-local-spectral-gap},
			\(\nabla^2F_{\TAP,\zeta}(x)\) denotes
			\(\nabla_-^2F_{\TAP,\zeta}(x)\), obtained by approaching \(q_x\)
			through smaller self-overlaps.
			\endgroup
		\end{enumerate}
		The conclusion in \textup{(ii)} is uniform under
		\[
		|q_\bfm-q|\le\ve,
		\qquad
		|f_N-f|\le\ve,
		\qquad
		W_2\left(N^{-1}\sum_{i=1}^N\delta_{b_i},\Law(X_q)\right)\le\ve.
		\]
		The probability of the complementary event in
		\eqref{eq:one-point-local-spectral-gap} is
		\(o_\ve(1)+o_N(1)\), with \(N\to\infty\) before \(\ve\downarrow0\).
	\end{proposition}
	
	\begin{remark}
		The Plefka condition places the limiting bulk in \((-\infty,0]\).
		The nonpositivity of the outliers requires a separate argument. This is
		consistent with the role of SUSY states as candidates for maxima of
		\(F_\TAP\).
	\end{remark}
	
	\begin{lemma}[Conditional determinant estimates]
		\label{lem:one-point-conditional-determinant}
		Under the hypotheses and notation of
		Proposition~\ref{prop:one-point-conditional-hessian}, the following
		hold for \(\sigma\in\{-,+\}\):
		\begin{enumerate}[label=\textup{(\roman*)}]
			\item
			\begin{equation}\label{eq:one-point-conditional-determinant}
				\begin{aligned}
					\frac1N
					\log\dE_{\bfm,f_N}
					\left|\det\nabla_\sigma^2F_{\TAP,\zeta}(\bfm)\right|
					={}&
					-\frac1N\sum_{i=1}^N
					\log\partial_{xx}\Phi_\zeta(q,b_i)
					+\frac{\xi''(q)}2
					\left(
					\frac1N\sum_{i=1}^N
					\partial_{xx}\Phi_\zeta(q,b_i)
					\right)^2
					+o(1)\\
					\longrightarrow{}&
					-\dE\log\partial_{xx}\Phi_\zeta(q,X_q)
					+\frac{\xi''(q)}2\widehat\zeta(q)^2 .
				\end{aligned}
			\end{equation}
			In the \(\ve\)-window above, the error in
			\eqref{eq:one-point-conditional-determinant} is
			\(o_\ve(1)+o_N(1)\).
			
			\item
			Suppose that, for some \(\kappa_0>0\) and \(L<\infty\),
			\[
			\frac1N\sum_{i=1}^N
			\exp\{(8+\kappa_0)|b_i|\}\le L.
			\]
			Set
			\[
			D_{\bfm,\sigma}
			:=
			\det\nabla_\sigma^2F_{\TAP,\zeta}(\bfm).
			\]
			There are \(c_L>0\) and a deterministic \(r_N=o(N)\) such that
			\begin{equation}\label{eq:one-point-determinant-lower-tail}
				\dP_{\bfm,f_N}\left(
				|D_{\bfm,\sigma}|
				<e^{-r_N}\dE_{\bfm,f_N}|D_{\bfm,\sigma}|
				\right)
				\le e^{-c_LN}.
			\end{equation}
			Moreover, the estimate is uniform over the same window, subject to the
			exponential-moment bound above. There is a function \(\rho_{\det}\),
			with \(\rho_{\det}(\ve)\downarrow0\), and there are nonnegative
			deterministic numbers \(r_{N,\ve}=o_N(N)\) such that
			\begin{equation}\label{eq:susy-one-point-determinant-lower-tail}
				\dP_{\bfm,f_N}\left(
				|D_{\bfm,\sigma}|
				<
				\exp\{-N\rho_{\det}(\ve)-r_{N,\ve}\}
				\dE_{\bfm,f_N}|D_{\bfm,\sigma}|
				\right)
				\le e^{-c_LN}.
			\end{equation}
			Here \(r_{N,\ve}/N\to0\) for fixed \(\ve\), and \(N\to\infty\)
			before \(\ve\downarrow0\).
		\end{enumerate}
	\end{lemma}
	
	Both proofs use the following conditional Gaussian regression formula.
	
	\begin{lemma}
		\label{lem:one-point-gaussian-regression}
		Under the hypotheses and notation of
		Proposition~\ref{prop:one-point-conditional-hessian}, whenever
		\(m=\partial_x\Phi_\zeta(r,x)\), write
		\[
		u(r,m):=\partial_{xx}\Phi_\zeta(r,x)
		=\frac1{\partial_{mm}h_\zeta(r,m)}.
		\]
		Set \(e=\bfm/\sqrt{Nq}\), and let \(P\) be the orthogonal
		projection onto \(\bfm^\perp\). Define
		\[
		z_i
		:=
		\xi''(q)\left(
		\frac{b_i}{\xi'(q)}-\partial_m u(q,m_i)
		-2\zeta(\{0\})m_i
		\right).
		\]
		There exist a GOE matrix \(W_N\), with off-diagonal entries of
		variance \(N^{-1}\), and a standard Gaussian vector \(g_N\) in
		\(\bfm^\perp\), independent of \(W_N\), such that, under
		\(\dP_{\bfm,f_N}\) and relative to
		\(\mathbb R e\oplus\bfm^\perp\),
		\begin{equation}\label{eq:one-point-hessian-blocks}
			\nabla_-^2F_{\TAP,\zeta}(\bfm)
			=
			\begin{pmatrix}
				a_N&y_N^{\mathsf T}P\\
				Py_N&C_N
			\end{pmatrix}
			+o(1),
		\end{equation}
		where the error converges to zero in operator norm in probability,
		and
		\begin{equation}\label{eq:one-point-hessian-tangent-block}
			C_N
			=
			P\left(
			\sqrt{\xi''(q)}\,W_N
			-\operatorname{diag}\left(
			\frac1{u(q,m_i)}
			+\frac{\xi''(q)}N\sum_{j=1}^Nu(q,m_j)
			\right)_{i\le N}
			\right)P\big|_{\bfm^\perp},
		\end{equation}
		\[
		\begin{aligned}
			y_N={}&
			\frac1{\sqrt N}\left(
			-\frac{m_i}{\sqrt q\,u(q,m_i)}+\sqrt q\,z_i
			+\sqrt q\,\zeta(\{0\})\xi''(q)m_i
			\right)_{i\le N}\\
			&+\left(
			\frac{\xi''(q)+q\xi'''(q)-q\xi''(q)^2/\xi'(q)}N
			\right)^{1/2}g_N,
		\end{aligned}
		\]
		and
		\[
		\begin{aligned}
			a_N={}&
			-\frac1{Nq}\sum_{i=1}^N\frac{m_i^2}{u(q,m_i)}
			-\frac{\xi''(q)}N\sum_{i=1}^Nu(q,m_i)
			+\frac2N\sum_{i=1}^Nm_iz_i\\
			&+q\left(
			-\frac{\xi'''(q)}N\sum_{i=1}^Nu(q,m_i)
			+\zeta(\{0\})\xi''(q)
			-\frac{2\xi''(q)}N\sum_{i=1}^N\partial_q u(q,m_i)
			\right).
		\end{aligned}
		\]
	\end{lemma}
	
	\begin{proof}
		Write \(M_q=\partial_x\Phi_\zeta(q,X_q)\). Since
		\(\zeta((0,q))=0\), integration by parts in the plateau
		representation of \(X_q\), the contact identity
		\(\dE[M_q^2]=q\), and
		Lemma~\ref{lem:AC-susceptibility} give
		\begin{align}
			\dE\left[\partial_{xx}\Phi_\zeta(q,X_q)\right]
			&=\widehat\zeta(q),
			\label{eq:one-point-hessian-susceptibility}\\
			\dE\left[X_q\partial_x\Phi_\zeta(q,X_q)\right]
			&=\xi'(q)\dE\left[
			\partial_{xx}\Phi_\zeta(q,X_q)
			+\zeta(\{0\})M_q^2
			\right]\notag\\
			&=\xi'(q)\left(
			\widehat\zeta(q)+\zeta(\{0\})q
			\right).
			\label{eq:one-point-hessian-mixed-ward}
		\end{align}
		Set \(\bfb=(b_i)_{i\le N}\).
		Under the value conditioning at the selected level,
		\begin{equation}\label{eq:one-point-hessian-energy}
			\frac1N H_N(\bfm)
			=
			\xi'(q)\widehat\zeta(q)+\zeta(\{0\})\xi(q)+o(1).
		\end{equation}
		Indeed, the definition of \(F_{\TAP,\zeta}\) and the conditioning
		\(F_{\TAP,\zeta}(\bfm)=Nf_N\), with \(f_N\to f\), give
		\[
		\frac1N H_N(\bfm)
		=
		f_N-\frac1N\sum_{i=1}^N\Phi_\zeta^*(q,m_i)
		+\frac12\int_q^1t\xi''(t)\zeta([0,t])\,\dd t.
		\]
		The Fenchel equality and the empirical-law hypothesis imply
		\[
		\frac1N\sum_{i=1}^N\Phi_\zeta^*(q,m_i)
		=
		\dE\left[
		\Phi_\zeta(q,X_q)
		-X_q\partial_x\Phi_\zeta(q,X_q)
		\right]+o(1).
		\]
		Substituting \eqref{eq:ising-selected-potential} and
		\eqref{eq:one-point-hessian-mixed-ward} into these identities yields
		\[
		\frac1N H_N(\bfm)
		=
		\xi'(q)\widehat\zeta(q)
		+\zeta(\{0\})\xi(q)+o(1),
		\]
		as claimed.
		
		Gaussian regression, using \eqref{eq:one-point-hessian-energy}, gives
		\[
		\begin{aligned}
			&\dE\left[
			\nabla^2H_N(\bfm)
			\,\middle|\,
			\nabla F_{\TAP,\zeta}(\bfm)=0,
			\ F_{\TAP,\zeta}(\bfm)=Nf_N
			\right]\\
			&\qquad=
			\frac{\xi''(q)}{N\xi'(q)}
			\left(\bfb\bfm^{\mathsf T}+\bfm\bfb^{\mathsf T}\right)
			+\frac{
				\xi'''(q)\widehat\zeta(q)
				-\zeta(\{0\})\xi''(q)
			}{N}\bfm\bfm^{\mathsf T}
			+o(1),
		\end{aligned}
		\]
		where \(o(1)\) is in operator norm. Denote the centered conditional
		Hessian of \(H_N\) by \(\widetilde H_N\). Covariance differentiation
		gives
		\[
		P\widetilde H_NP\big|_{\bfm^\perp}
		\stackrel{\mathrm d}=
		\sqrt{\xi''(q)}\,PW_NP\big|_{\bfm^\perp},
		\]
		and
		\[
		P\widetilde H_Ne
		\stackrel{\mathrm d}=
		\sqrt{\frac1N\left(
			\xi''(q)+q\xi'''(q)
			-\frac{q\xi''(q)^2}{\xi'(q)}
			\right)}\,g_N,
		\qquad
		e^{\mathsf T}\widetilde H_Ne=o(1),
		\]
		where \(o(1)\) is in probability.
		
		If \(\zeta\) has an atom at \(q\), the TAP correction is
		\(C^1\) at \(q_\bfm=q\), but its Hessian has a jump there. We first
		compute the Hessian \(\nabla_-^2F_{\TAP,\zeta}\) from the side
		\(q_\bfm<q\). Differentiating the TAP correction and using \(e\)
		as the first basis vector gives
		\eqref{eq:one-point-hessian-blocks} and
		\eqref{eq:one-point-hessian-tangent-block}, together with the
		displayed formulas for \(a_N\) and \(y_N\).
	\end{proof}

	\begin{proof}[Proof of Proposition~\ref{prop:one-point-conditional-hessian}]
		Use the notation and the conditional representation of
		Lemma~\ref{lem:one-point-gaussian-regression}.
		For \(t>0\), let \(\sigma_t\) denote the centered semicircle law
		of variance \(t\). We write \(\boxplus\) for free convolution.
		
		\medskip\noindent\emph{Step 1 (proving the free-convolution gap).}
		Set
		\[
		\mu_q
		:=
		\Law\left(
		\frac1{\partial_{xx}\Phi_\zeta(q,X_q)}
		+\xi''(q)\widehat\zeta(q)
		\right).
		\]
		Write \(G_{\mu_q}(z)=\int(z-x)^{-1}\,\mu_q(\dd x)\), and set
		\[
		\omega_0:=\xi''(q)\widehat\zeta(q).
		\]
		The susceptibility identity
		\eqref{eq:one-point-hessian-susceptibility} gives
		\[
		G_{\mu_q}(\omega_0)
		=-\dE\left[\partial_{xx}\Phi_\zeta(q,X_q)\right]
		=-\widehat\zeta(q),
		\qquad
		\omega_0+\xi''(q)G_{\mu_q}(\omega_0)=0.
		\]
		At this preimage, the strict Plefka condition reads
		\[
		1+\xi''(q)G_{\mu_q}'(\omega_0)
		=
		1-\xi''(q)\dE\left[
		\left(
		\partial_{xx}\Phi_\zeta(q,X_q)
		\right)^2
		\right]>0.
		\]
		
		By the left-edge formula \cite[Eq.~(C.16)]{TAPIsing}
		\textup{(}see also
		\cite[Lemmas~2 and~4 and Corollary~3]{Biane1997}\textup{)},
		the left edge is
		\(\omega_*+\xi''(q)G_{\mu_q}(\omega_*)\), where
		\[
		\omega_*
		:=
		\inf\left\{
		\omega\in\mathbb R:
		\int\frac{\mu_q(\dd x)}{(x-\omega)^2}
		\ge\frac1{\xi''(q)}
		\right\}.
		\]
		The defining integral is increasing and is smaller than
		\(1/\xi''(q)\) at \(\omega_0\), so \(\omega_*>\omega_0\).
		The characteristic
		\[
		\omega\longmapsto
		\omega+\xi''(q)G_{\mu_q}(\omega)
		\]
		is strictly increasing on \([\omega_0,\omega_*)\) and vanishes at
		\(\omega_0\). Hence
		\[
		\inf\operatorname{supp}
		\left(\mu_q\boxplus\sigma_{\xi''(q)}\right)>0.
		\]
		
		Since \(0<\partial_{xx}\Phi_\zeta\le1\),
		\(\inf\operatorname{supp}\mu_q\ge\omega_0+1\). Choose
		\(a\in(0,1)\) such that the characteristic has positive derivative on
		\([\omega_0,\omega_0+a]\). For \(M<\infty\), let
		\(\mu_{q,M}\) be the law of
		\[
		\left(
		\frac1{\partial_{xx}\Phi_\zeta(q,X_q)}
		+\xi''(q)\widehat\zeta(q)
		\right)\wedge M.
		\]
		The Stieltjes transforms of \(\mu_{q,M}\) and their derivatives
		converge uniformly on \([\omega_0,\omega_0+a]\). For all sufficiently
		large \(M\), the truncated characteristic is increasing there and is
		positive at \(\omega_0+a\). The left-edge formula then gives
		\[
		\inf\operatorname{supp}
		\left(\mu_{q,M}\boxplus\sigma_{\xi''(q)}\right)
		\ge
		\omega_0+a+\xi''(q)G_{\mu_{q,M}}(\omega_0+a)>0.
		\]
		Fix such an \(M\), and then \(c_0>0\) so that this edge is larger than
		\(4c_0\).
		
		\medskip\noindent\emph{Step 2 (controlling the transverse block).}
		Put
		\[
		d_{i,N}
		:=
		\frac1{u(q,m_i)}
		+\frac{\xi''(q)}N\sum_{j=1}^Nu(q,m_j).
		\]
		The empirical-law hypothesis gives
		\[
		\frac1N\sum_{j=1}^Nu(q,m_j)
		\longrightarrow\widehat\zeta(q),
		\qquad
		\frac1N\sum_{i=1}^N\delta_{d_{i,N}\wedge M}
		\Longrightarrow\mu_{q,M}.
		\]
		The full support of \(X_q\) and the continuity of
		\(x\mapsto\partial_{xx}\Phi_\zeta(q,x)\) give
		\[
		\max_{i\le N}
		\operatorname{dist}\left(
		d_{i,N}\wedge M,\operatorname{supp}\mu_{q,M}
		\right)
		\le
		\xi''(q)\left|
		\frac1N\sum_{j=1}^Nu(q,m_j)-\widehat\zeta(q)
		\right|
		=o(1).
		\]
		Thus \cite[Proposition~8.1]{CapitaineDonatiMartinFeralFevrier2011}
		applies. Since \(-W_N\) is again a GOE,
		\[
		\lambda_{\min}\left(
		\operatorname{diag}(d_{i,N}\wedge M)
		-\sqrt{\xi''(q)}\,W_N
		\right)
		\xrightarrow{\dP}
		\inf\operatorname{supp}
		\left(\mu_{q,M}\boxplus\sigma_{\xi''(q)}\right).
		\]
		The limit is larger than \(4c_0\). Since
		\(d_{i,N}\ge d_{i,N}\wedge M\), for every
		\(v\in\bfm^\perp\), \eqref{eq:one-point-hessian-tangent-block} gives
		\[
		-v^{\mathsf T}C_Nv
		\ge
		v^{\mathsf T}
		\left(
		\operatorname{diag}(d_{i,N}\wedge M)
		-\sqrt{\xi''(q)}\,W_N
		\right)v
		\ge2c_0\|v\|_2^2
		\]
		with probability tending to one. Therefore
		\begin{equation}\label{eq:one-point-hessian-bulk-gap}
			C_N\preceq-c_0I_{N-1}
		\end{equation}
		with probability tending to one.
		Set
		\[
		\widetilde C_N
		:=
		\sqrt{\xi''(q)}\,W_N
		-\operatorname{diag}(d_{i,N})_{i\le N},
		\]
		so that \(C_N=P\widetilde C_NP|_{\bfm^\perp}\).
		
		\medskip\noindent\emph{Step 3 (computing the resolvent).}
		The same comparison gives
		\(\widetilde C_N\preceq-2c_0I_N\) with probability tending to one.
		For \(\widetilde C_N\), the vector Dyson equation is solved by
		\(-\operatorname{diag}(u(q,m_i))\), whose normalized trace cancels the
		scalar shift in \(\widetilde C_N\). Its stability parameter is
		\[
		\frac{\xi''(q)}N\sum_{i=1}^Nu(q,m_i)^2
		\longrightarrow
		\xi''(q)\dE\left[
		\left(\partial_{xx}\Phi_\zeta(q,X_q)\right)^2
		\right]<1.
		\]
		Apply the anisotropic local law of \cite{KnowlesYin} first to a
		truncated diagonal. The spectral gap removes the truncation and gives
		\begin{equation}\label{eq:one-point-anisotropic-resolvent}
			\frac1N v^{\mathsf T}\widetilde C_N^{-1}w
			=-\frac1N\sum_{i=1}^Nu(q,m_i)v_iw_i+o(1)
			\qquad\text{with high probability}.
		\end{equation}
		The estimate holds for deterministic arrays \(v,w\) with bounded
		empirical second moments. It also holds conditionally on \(v,w\) when
		they are independent of \(W_N\).
		
		\medskip\noindent\emph{Step 4 (computing the scalar Schur complement).}
		For an invertible \(C\),
		\begin{equation}
			\begin{pmatrix}a&y^{\mathsf T}\\y&C\end{pmatrix}
			=
			\begin{pmatrix}1&y^{\mathsf T}C^{-1}\\0&I\end{pmatrix}
			\begin{pmatrix}a-y^{\mathsf T}C^{-1}y&0\\0&C\end{pmatrix}
			\begin{pmatrix}1&0\\C^{-1}y&I\end{pmatrix}.
		\end{equation}
		By \eqref{eq:one-point-hessian-bulk-gap}, it is enough to show that
		\(a_N-(Py_N)^{\mathsf T}C_N^{-1}(Py_N)\) is strictly negative.
		On this event, the inverse of the restriction is
		\[
		PC_N^{-1}P
		=
		\widetilde C_N^{-1}
		-\frac{
			\widetilde C_N^{-1}ee^{\mathsf T}\widetilde C_N^{-1}
		}{
			e^{\mathsf T}\widetilde C_N^{-1}e
		}.
		\]
		Consequently, the Schur complement has the exact decomposition
		\begin{equation}\label{eq:one-point-schur-decomposition}
			a_N-(Py_N)^{\mathsf T}C_N^{-1}(Py_N)
			=
			a_N-y_N^{\mathsf T}\widetilde C_N^{-1}y_N
			+\frac{
				\left(y_N^{\mathsf T}\widetilde C_N^{-1}e\right)^2
			}{
				e^{\mathsf T}\widetilde C_N^{-1}e
			}.
		\end{equation}
		
		Set
		\[
		\Delta_{\mathrm{Pl}}:=1-\xi''(q)\dE\left[
		\left(
		\partial_{xx}\Phi_\zeta(q,X_q)
		\right)^2
		\right]>0,
		\qquad
		\gamma_q:=\dE\left[
		\partial_{xx}\Phi_\zeta(q,X_q)
		\left(\partial_x\Phi_\zeta(q,X_q)\right)^2
		\right]>0.
		\]
		The three terms in \eqref{eq:one-point-schur-decomposition} converge as
		follows:
		\begin{align}
			e^{\mathsf T}\widetilde C_N^{-1}e
			&\xrightarrow{\dP}-\frac{\gamma_q}{q},
			\label{eq:one-point-schur-denominator}\\
			\left(y_N^{\mathsf T}\widetilde C_N^{-1}e\right)^2
			&\xrightarrow{\dP}\Delta_{\mathrm{Pl}}^2,
			\label{eq:one-point-schur-numerator}\\
			a_N-y_N^{\mathsf T}\widetilde C_N^{-1}y_N
			&\xrightarrow{\dP}
			-q\zeta(\{0\})\xi''(q)\Delta_{\mathrm{Pl}}.
			\label{eq:one-point-schur-first-term}
		\end{align}
		Since \(\zeta((0,q))=0\), integration by parts under the law of \(X_q\)
		reads
		\begin{equation}\label{eq:one-point-integration-by-parts}
			\dE\left[
			\psi(X_q)\left(
			\frac{X_q}{\xi'(q)}
			-\zeta(\{0\})\partial_x\Phi_\zeta(q,X_q)
			\right)
			\right]
			=\dE[\psi'(X_q)].
		\end{equation}
		Equation~\eqref{eq:one-point-schur-denominator} follows from
		\eqref{eq:one-point-anisotropic-resolvent} with \(v=w=\bfm\), since
		\(e=\bfm/\sqrt{Nq}\).
		
		For the numerator, use
		\[
		\begin{aligned}
			z_i+\zeta(\{0\})\xi''(q)m_i
			&=
			\xi''(q)\left(
			\frac{b_i}{\xi'(q)}
			-\partial_m u(q,m_i)
			-\zeta(\{0\})m_i
			\right),\\
			\partial_m u(q,m_i)
			&=
			\frac{\partial_{xxx}\Phi_\zeta(q,b_i)}
			{\partial_{xx}\Phi_\zeta(q,b_i)}.
		\end{aligned}
		\]
		Applying \(\widetilde C_N\widetilde C_N^{-1}=I_N\) to the
		\(m_i/u(q,m_i)\) term and
		\eqref{eq:one-point-anisotropic-resolvent} to the remaining terms gives
		\[
		y_N^{\mathsf T}\widetilde C_N^{-1}e
		\xrightarrow{\dP}
		1-\xi''(q)\dE\left[
		\partial_{xx}\Phi_\zeta(q,X_q)
		\partial_x\Phi_\zeta(q,X_q)
		\left(
		\frac{X_q}{\xi'(q)}
		-\zeta(\{0\})\partial_x\Phi_\zeta(q,X_q)
		-\frac{\partial_{xxx}\Phi_\zeta(q,X_q)}
		{\partial_{xx}\Phi_\zeta(q,X_q)}
		\right)
		\right].
		\]
		The independent \(g_N\) term is \(o(1)\) in probability.
		Taking
		\(\psi=\partial_x\Phi_\zeta(q,\cdot)
		\partial_{xx}\Phi_\zeta(q,\cdot)\) in
		\eqref{eq:one-point-integration-by-parts} identifies the expectation as
		\(\dE[(\partial_{xx}\Phi_\zeta(q,X_q))^2]\), proving
		\eqref{eq:one-point-schur-numerator}.
		
		To prove \eqref{eq:one-point-schur-first-term}, apply
		\(\widetilde C_N\widetilde C_N^{-1}=I_N\) to each factor
		\(m_i/u(q,m_i)\), and apply
		\eqref{eq:one-point-anisotropic-resolvent} to the remaining quadratic
		forms, conditionally on \(g_N\) when needed. All unmarked derivatives
		below are evaluated at \((q,X_q)\). This gives
		\[
		\begin{aligned}
			a_N-y_N^{\mathsf T}\widetilde C_N^{-1}y_N
			\longrightarrow q\xi''(q)\Bigg[&-\zeta(\{0\})
			-2\dE\,\partial_q u(q,\partial_x\Phi_\zeta(q,X_q))\\
			&+\xi''(q)\dE\left[
			\partial_{xx}\Phi_\zeta
			\left(
			\frac{X_q}{\xi'(q)}-\zeta(\{0\})\partial_x\Phi_\zeta
			-\frac{\partial_{xxx}\Phi_\zeta}{\partial_{xx}\Phi_\zeta}
			\right)^2\right]
			-\frac{\xi''(q)}{\xi'(q)}\widehat\zeta(q)\Bigg].
		\end{aligned}
		\]
		Differentiating the Parisi equation at fixed \(m\) gives
		\[
		\begin{aligned}
			-2\partial_q u(q,\partial_x\Phi_\zeta(q,X_q))
			=\xi''(q)\Bigg(&\partial_{xxxx}\Phi_\zeta
			-\frac{(\partial_{xxx}\Phi_\zeta)^2}
			{\partial_{xx}\Phi_\zeta}
			+2\zeta(\{0\})(\partial_{xx}\Phi_\zeta)^2\Bigg).
		\end{aligned}
		\]
		Two applications of
		\eqref{eq:one-point-integration-by-parts} give
		\[
		\begin{aligned}
			&\dE\left[\partial_{xx}\Phi_\zeta
			\left(
			\frac{X_q}{\xi'(q)}
			-\zeta(\{0\})\partial_x\Phi_\zeta
			-\frac{\partial_{xxx}\Phi_\zeta}
			{\partial_{xx}\Phi_\zeta}
			\right)^2\right]\\
			&\qquad=
			\dE\left[
			-\partial_{xxxx}\Phi_\zeta
			+\frac{(\partial_{xxx}\Phi_\zeta)^2}
			{\partial_{xx}\Phi_\zeta}
			+\frac{\partial_{xx}\Phi_\zeta}{\xi'(q)}
			-\zeta(\{0\})(\partial_{xx}\Phi_\zeta)^2
			\right].
		\end{aligned}
		\]
		Substitution reduces the bracket to
		\[
		-\zeta(\{0\})
		+\zeta(\{0\})\xi''(q)
		\dE\left[
		\left(\partial_{xx}\Phi_\zeta(q,X_q)\right)^2
		\right]
		=-\zeta(\{0\})\Delta_{\mathrm{Pl}},
		\]
		which proves \eqref{eq:one-point-schur-first-term}. Substituting the three
		limits into \eqref{eq:one-point-schur-decomposition} yields
		\begin{equation}\label{eq:one-point-negative-schur-complement}
			a_N-(Py_N)^{\mathsf T}C_N^{-1}(Py_N)
			\xrightarrow{\dP}
			-q\zeta(\{0\})\xi''(q)\Delta_{\mathrm{Pl}}
			-\frac{q\Delta_{\mathrm{Pl}}^2}{\gamma_q}<0.
		\end{equation}
		
		The resolvent estimate also gives
		\(\|C_N^{-1}Py_N\|_2=O(1)\) with high probability. The block
		factorization, \eqref{eq:one-point-hessian-bulk-gap}, and
		\eqref{eq:one-point-negative-schur-complement} give some \(c_1>0\)
		such that
		\[
		\nabla_-^2F_{\TAP,\zeta}(\bfm)\preceq-c_1I_N
		\]
		with probability tending to one. If \(\zeta\) has an atom at \(q\),
		comparing the two one-sided second derivatives at \(q_\bfm=q\) gives
		\begin{equation}\label{eq:one-point-hessian-jump}
			\begin{aligned}
				&\nabla_+^2F_{\TAP,\zeta}(\bfm)
				-\nabla_-^2F_{\TAP,\zeta}(\bfm)\\
				&\qquad=
				-\frac{2\zeta(\{q\})\xi''(q)}N
				\left(
				1-\frac{\xi''(q)}N\sum_{i=1}^N
				\bigl(\partial_{xx}\Phi_\zeta(q,b_i)\bigr)^2
				\right)\bfm\bfm^{\mathsf T}.
			\end{aligned}
		\end{equation}
		The empirical Plefka factor is positive for all large \(N\). Hence the
		jump is negative semidefinite, and the same gap holds for
		\(\nabla_+^2F_{\TAP,\zeta}(\bfm)\).
		
		\medskip\noindent\emph{Step 5 (extending the gap locally).}
		If \(\|x-\bfm\|_2\le2\sqrt{\delta N}\), then
		\(|q_x-q|\le C\sqrt\delta\). The compact-set continuity used in
		\cite[Proof of Lemma~C.4]{TAPIsing}, together with the \(W_2\)-hypothesis,
		gives
		\begin{equation}\label{eq:one-point-local-susceptibility-continuity}
			\lim_{\delta\downarrow0}\limsup_{N\to\infty}
			\sup_{\|x-\bfm\|_2\le2\sqrt{\delta N}}
			\frac1N\sum_{i=1}^N
			\left|
			u(q_x,x_i)-u(q,m_i)
			\right|^2
			=0.
		\end{equation}
		For all sufficiently small \(\delta\), this implies
		\[
		1-\frac{\xi''(q_x)}N\sum_{i=1}^N
		u(q_x,x_i)^2
		\ge\frac{\Delta_{\mathrm{Pl}}}{2}
		\]
		uniformly in the ball for all sufficiently large \(N\).
		
		Replace every occurrence of
		\(u(q_x,x_i)^{-1}\) in the differentiated TAP correction by
		\(u(q_x,x_i)^{-1}\wedge K\), and denote the resulting matrix by
		\(\mathcal H_{N,K}(x)\). Choose \(K\) large enough that the
		free-convolution edge and the scalar Schur complement at \(\bfm\)
		still have fixed gaps. For this fixed \(K\),
		\eqref{eq:one-point-local-susceptibility-continuity} controls the
		deterministic coefficients. Covariance differentiation gives the Gaussian
		increment bound. The edge estimate and Gaussian concentration give
		\(\kappa_*,c_K>0\) such that, once \(\delta\) is small enough, for
		every deterministic \(x\) in the ball,
		\[
		\dP_{\bfm,f_N}\left(
		\lambda_{\max}\!\left(
		\mathcal H_{N,K}(x)
		\right)>-3\kappa_*
		\right)
		\le e^{-c_KN}
		\]
		for all sufficiently large \(N\).
		
		Choose an \(\ell_\infty\)-net \(\mathcal Q_N\) of the ball with
		\[
		\frac1N\log|\mathcal Q_N|=o_\delta(1).
		\]
		The increment bound and a union bound over \(\mathcal Q_N\) yield
		\[
		\sup_{\substack{
				x\in(-1,1)^N,\\
				\|x-\bfm\|_2\le2\sqrt{\delta N}
		}}
		\lambda_{\max}\bigl(\mathcal H_{N,K}(x)\bigr)
		\le-2\kappa_*
		\]
		with probability tending to one.
		
		To remove the cap, absorb the mixed terms into the discarded negative
		diagonal using
		\[
		yr^{\mathsf T}+ry^{\mathsf T}-D
		\preceq
		(r^{\mathsf T}D^{-1}r)\,yy^{\mathsf T},
		\qquad D\succ0.
		\]
		Using \(|\partial_m u|\le2\) and the \(W_2\)-hypothesis, this gives,
		uniformly over the ball,
		\[
		\nabla^2F_{\TAP,\zeta}(x)
		\preceq
		\mathcal H_{N,K}(x)
		+\bigl(\varepsilon_K+\omega_K(\delta)+o(1)\bigr)I_N,
		\qquad
		\varepsilon_K\longrightarrow0,
		\quad
		\omega_K(\delta)\longrightarrow0
		\ \text{as }\delta\downarrow0.
		\]
		Choose \(K\) large and then \(\delta\) small so that the error is below
		\(\kappa_*\). This proves
		\eqref{eq:one-point-local-spectral-gap} with \(\kappa=\kappa_*\).
		
		The resulting \(\delta_0\) and \(\kappa_*\) are uniform in the
		\(\ve\)-window.
		If \(\zeta\) has an atom at \(q\), the uniform Plefka bound and
		\eqref{eq:one-point-hessian-jump} reduce the claim to
		\(\nabla_-^2F_{\TAP,\zeta}\).
	\end{proof}
	
	\begin{proof}[Proof of Lemma~\ref{lem:one-point-conditional-determinant}]
		We use the notation of Lemma~\ref{lem:one-point-gaussian-regression}
		and begin with \textup{(i)}.
		
		Set
		\[
		U_N:=\operatorname{diag}\bigl(u(q,m_i)\bigr)_{i\le N},
		\qquad
		\chi_N:=\frac1N\sum_{i=1}^Nu(q,m_i),
		\qquad
		\mu_N:=\frac1N\sum_{i=1}^N
		\delta_{u(q,m_i)^{-1}+\xi''(q)\chi_N}.
		\]
		To reduce to a bounded deformation, fix \(T<\infty\) and set
		\[
		b_i^{(T)}:=(-T)\vee b_i\wedge T,
		\qquad
		u_i^{(T)}:=\partial_{xx}\Phi_\zeta(q,b_i^{(T)}),
		\qquad
		U_N^{(T)}:=\operatorname{diag}(u_i^{(T)}),
		\qquad
		\chi_N^{(T)}:=\frac1N\sum_{i=1}^Nu_i^{(T)}.
		\]
		The corresponding deformed-GOE matrices are
		\[
		\begin{aligned}
			\mathsf M_N
			&:=U_N^{-1}+\xi''(q)\chi_NI_N
			-\sqrt{\xi''(q)}W_N,\\
			\mathsf M_N^{(T)}
			&:=(U_N^{(T)})^{-1}+\xi''(q)\chi_N^{(T)}I_N
			-\sqrt{\xi''(q)}W_N.
		\end{aligned}
		\]
		The deformed-GOE argument in \cite[Proof of Lemma~C.4]{TAPIsing} applies to
		\(\mathsf M_N^{(T)}\). We use
		\[
		-\log\partial_{xx}\Phi_\zeta(q,x)\le C(1+|x|)
		\]
		and the factorization
		\[
		\begin{aligned}
			&\det\mathsf M_N\\
			&\quad=
			\left(\prod_{i=1}^Nu(q,m_i)^{-1}\right)
			\det\left(
			I_N+\xi''(q)\chi_NU_N
			-\sqrt{\xi''(q)}U_N^{1/2}W_NU_N^{1/2}
			\right).
		\end{aligned}
		\]
		The same factorization holds for \(\mathsf M_N^{(T)}\). The determinant
		comparison in \cite[Proof of Lemma~C.4]{TAPIsing} gives
		\[
		\left|
		\log\dE|\det\mathsf M_N|
		-
		\log\dE|\det\mathsf M_N^{(T)}|
		\right|
		\le
		C\sum_{\lvert b_i\rvert>T}(1+\lvert b_i\rvert)+o(N).
		\]
		The same comparison shows that the finite-rank terms
		in Lemma~\ref{lem:one-point-gaussian-regression} and the rank-one jump
		\eqref{eq:one-point-hessian-jump} contribute \(o(N)\).
		The \(W_2\)-hypothesis gives
		\[
		\lim_{T\to\infty}\limsup_{N\to\infty}
		\frac1N\sum_{\lvert b_i\rvert>T}(1+\lvert b_i\rvert)=0.
		\]
		Letting \(T\to\infty\) gives, for \(\sigma\in\{-,+\}\),
		\[
		\frac1N
		\log\dE_{\bfm,f_N}
		\left|\det\nabla_\sigma^2F_{\TAP,\zeta}(\bfm)\right|
		=
		\int\log|x|\,\dd
		\left(\mu_N\boxplus\sigma_{\xi''(q)}\right)(x)+o(1).
		\]
		
		It remains to evaluate the logarithmic potential. For large \(N\),
		strict Plefka gives
		\[
		0<\inf\operatorname{supp}
		\left(\mu_N\boxplus\sigma_{\xi''(q)}\right).
		\]
		The subordination preimage of zero is \(\xi''(q)\chi_N\).
		The identity in \cite[Lemma~C.3]{TAPIsing} gives
		\[
		\int\log|x|\,\dd
		\left(\mu_N\boxplus\sigma_{\xi''(q)}\right)(x)
		=
		-\frac1N\sum_{i=1}^N\log u(q,m_i)
		+\frac{\xi''(q)}2\chi_N^2.
		\]
		Together with the \(W_2\)-convergence, this proves
		\eqref{eq:one-point-conditional-determinant}.
		For the windowed estimate, fix \(T\). On \(\{|b_i|\le T\}\), the
		compact-set continuity used in \cite[Proof of Lemma~C.4]{TAPIsing}
		controls the diagonal and regression coefficients, while strict Plefka
		preserves the gap. The displayed tail bound handles the remaining
		coordinates. We then let \(T\to\infty\).
		
		For \textup{(ii)}, assume the exponential-moment bound. The pointwise
		estimate in Step~5 of the proof of
		Proposition~\ref{prop:one-point-conditional-hessian}, followed by removal
		of the cap, gives
		\(c_*,c_L>0\) such that, uniformly on the same \(\ve\)-window,
		\begin{equation}\label{eq:one-point-exponential-gap}
			\dP_{\bfm,f_N}\left(
			\nabla_\sigma^2F_{\TAP,\zeta}(\bfm)
			\not\preceq-c_*I_N
			\right)
			\le e^{-c_LN}.
		\end{equation}
		Fix \(0<\kappa<c_*/4\), where \(c_*\) is from
		\eqref{eq:one-point-exponential-gap}. Let
		\((\lambda_i,v_i)_{i\le N}\) be an orthonormal eigendecomposition of
		\(\nabla_\sigma^2F_{\TAP,\zeta}(\bfm)\). Define
		\[
		\ell_{N,\sigma}
		:=
		\sum_{i=1}^N
		\log\max\{-\lambda_i,\kappa\},
		\qquad
		R_N
		:=
		\sum_{i:\,\lambda_i<-\kappa}
		\frac1{\lambda_i}v_iv_i^{\mathsf T}.
		\]
		Let \(G=(G_\gamma)_\gamma\) be standard Gaussian coordinates generating
		the conditional Hessian.
		The exponential-moment bound gives a uniform \(K_L\)-Lipschitz estimate
		for \(\ell_{N,\sigma}\) as a function of \(G\). Covariance
		differentiation gives
		\[
		\sum_\gamma
		\left[
		\operatorname{Tr}\left(
		R_N\partial_{G_\gamma}
		\left[\nabla_\sigma^2F_{\TAP,\zeta}(\bfm)\right]
		\right)
		\right]^2
		\le
		\frac{C_L}{N}\operatorname{Tr}(R_N^2)+C_L
		\le K_L^2,
		\]
		since \(\|R_N\|_{\mathrm{op}}\le\kappa^{-1}\).
		
		We compare \(\ell_{N,\sigma}\) with
		\[
		\mathcal L_N
		:=
		-\sum_{i=1}^N\log u(q,m_i)
		+\frac{N\xi''(q)}2
		\left(\frac1N\sum_{i=1}^Nu(q,m_i)\right)^2.
		\]
		The deformed-GOE spectral convergence and
		Proposition~\ref{prop:one-point-conditional-hessian}\textup{(i)} give a
		deterministic \(s_N=o(N)\) such that
		\[
		\dP_{\bfm,f_N}\left(
		|\ell_{N,\sigma}-\mathcal L_N|>s_N
		\right)=o(1).
		\]
		After increasing \(s_N\), \eqref{eq:one-point-conditional-determinant}
		also gives
		\[
		\left|
		\log\dE_{\bfm,f_N}|D_{\bfm,\sigma}|
		-\mathcal L_N
		\right|
		\le s_N.
		\]
		Every median of \(\ell_{N,\sigma}\) is at least
		\(\mathcal L_N-s_N\). For fixed \(a>0\), Gaussian concentration and
		\eqref{eq:one-point-exponential-gap} give
		\[
		\dP_{\bfm,f_N}\left(
		|D_{\bfm,\sigma}|
		<
		e^{-2s_N-a\sqrt N}
		\dE_{\bfm,f_N}|D_{\bfm,\sigma}|
		\right)
		\le e^{-c_LN}.
		\]
		This proves \eqref{eq:one-point-determinant-lower-tail} with
		\(r_N=2s_N+a\sqrt N=o(N)\).
		
		Uniformly over the \(\ve\)-window, we may choose \(\rho_{\det}\) and
		\(r_{N,\ve}\) so that
		\[
		2s_N+a\sqrt N
		\le N\rho_{\det}(\ve)+r_{N,\ve},
		\qquad
		\rho_{\det}(\ve)\downarrow0,
		\qquad
		r_{N,\ve}=o_N(N).
		\]
		This proves
		\eqref{eq:susy-one-point-determinant-lower-tail}.
	\end{proof}
	
	%==============================================================================
	\subsection{The joint density}
	%==============================================================================
	
	The remaining term in the Kac--Rice formula is the joint density of the
	value and gradient.
	
	Set
	\[
	B(\bfm)
	:=
	\sum_{i=1}^N
	\left[
	-h_\zeta(q,m_i)-\Phi_\zeta(0,0)
	\right]
	+\frac{N\theta}{2}\int_0^q t\xi''(t)\,\dd t .
	\]
	Uniformly for \(|q_\bfm-q|\le\ve\),
	\begin{equation}\label{eq:one-point-decomposition}
		F_{\TAP,\zeta}(\bfm)-NP_\zeta
		=
		H_N(\bfm)+B(\bfm)+N o_\ve(1)+o_N(N).
	\end{equation}
	Define the corresponding gradient and value targets by
	\begin{equation}\label{eq:one-point-gE}
		g_\bfm:=\bfb+\xi''(q)\widehat\zeta(q)\,\bfm,
		\qquad
		E_\bfm(f')
		:=
		N(f'-P_\zeta)-B(\bfm).
	\end{equation}
	
	\begin{lemma}
		\label{lem:one-point-joint-density}
		Uniformly over \(f'\in[f-\ve,f+\ve]\) and
		\(\bfm\in\mathcal D_{N,\ve}(f')\),
		\begin{equation}\label{eq:one-point-joint-density-bound}
			\begin{aligned}
				\log \varphi_\bfm\bigl(0,N(f'-P_\zeta)\bigr)
				\le{}&
				-\frac N2\log(2\pi \xi'(q))
				-\theta E_\bfm(f')\\
				&-\frac1{2\xi'(q)}
				\left\|\bfb-\theta\xi'(q)\bfm\right\|_2^2
				+\frac N2\theta^2\xi(q)\\
				&-\frac N2\xi''(q)\widehat\zeta(q)^2
				+N o_\ve(1)+o_N(N).
			\end{aligned}
		\end{equation}
		If, in addition, \(f'=f+o_\ve(1)+o_N(1)\) and
		\begin{equation}\label{eq:ward2}
			\frac1N\sum_{i=1}^N
			\bigl[\Phi_\zeta(q,b_i)-\Phi_\zeta(0,0)\bigr]
			=
			f-P_\zeta
			+\frac\theta2\int_0^q t\xi''(t)\,\dd t
			+o_\ve(1)+o_N(1),
		\end{equation}
		then equality holds in
		\eqref{eq:one-point-joint-density-bound} up to
		\(N o_\ve(1)+o_N(N)\).
	\end{lemma}
	
	\begin{proof}
		Set
		\[
		\widetilde g_\bfm
		:=
		-\nabla\bigl(F_{\TAP,\zeta}-H_N\bigr)(\bfm),
		\qquad
		\widetilde E_\bfm
		:=
		Nf'-\bigl(F_{\TAP,\zeta}-H_N\bigr)(\bfm).
		\]
		Thus the conditioning in \eqref{eq:one-point-kac-rice-window}
		prescribes
		\[
		\bigl(\nabla H_N(\bfm),H_N(\bfm)\bigr)
		=
		(\widetilde g_\bfm,\widetilde E_\bfm).
		\]
		Equations~\eqref{eq:one-point-decomposition} and
		\eqref{eq:one-point-gE} give
		\begin{equation}\label{eq:one-point-exact-target}
			\begin{aligned}
				\|\widetilde g_\bfm-g_\bfm\|_2
				&=\sqrt N\,o_\ve(1)+o_N(\sqrt N),\\
				|\widetilde E_\bfm-E_\bfm(f')|
				&=N o_\ve(1)+o_N(N).
			\end{aligned}
		\end{equation}
		
		Covariance differentiation gives
		\begin{equation}\label{eq:one-point-field-covariance}
			\begin{aligned}
				\dE H_N(\bfm)^2
				&=N\xi(q_\bfm),\\
				\dE\bigl[\nabla H_N(\bfm)H_N(\bfm)\bigr]
				&=\xi'(q_\bfm)\bfm,\\
				\Cov\bigl(\nabla H_N(\bfm)\bigr)
				&=\xi'(q_\bfm)I_N
				+\frac{\xi''(q_\bfm)}N\bfm\bfm^{\mathsf T}.
			\end{aligned}
		\end{equation}
		For every \(q>0\),
		\begin{equation}\label{eq:one-point-covariance-nondegeneracy}
			\begin{aligned}
				&\xi(q)\bigl(\xi'(q)+q\xi''(q)\bigr)
				-q\xi'(q)^2\\
				&\qquad=
				\frac12\sum_{p,r\ge2}
				(p-r)^2\beta_p^2\beta_r^2q^{p+r-1}
				\ge0.
			\end{aligned}
		\end{equation}
		If \(\xi\) has at least two nonzero coefficients, the inequality is
		strict.
		
		Set
		\[
		\widehat Z_\bfm
		:=
		\bigl(\nabla H_N(\bfm),N^{-1/2}H_N(\bfm)\bigr).
		\]
		By \eqref{eq:one-point-field-covariance}, uniformly on
		\(\mathcal D_{N,\ve}(f')\),
		\begin{equation}\label{eq:one-point-rescaled-covariance}
			\left\|
			\Cov(\widehat Z_\bfm)
			-
			\begin{pmatrix}
				\xi'(q)I_N+\dfrac{\xi''(q)}N\bfm\bfm^{\mathsf T}
				&\xi'(q)\bfm/\sqrt N\\[4pt]
				\xi'(q)\bfm^{\mathsf T}/\sqrt N
				&\xi(q)
			\end{pmatrix}
			\right\|_{\mathrm{op}}
			=o_\ve(1)+o_N(1).
		\end{equation}
		Since \(\xi'(q)>0\) and
		\(\lvert q_\bfm-q\rvert\le\ve\) on
		\(\mathcal D_{N,\ve}(f')\), continuity and
		\eqref{eq:one-point-covariance-nondegeneracy} show that the matrix
		in \eqref{eq:one-point-rescaled-covariance} is uniformly
		nondegenerate for all sufficiently small \(\ve\). The matrix
		determinant lemma and the Schur complement give
		\[
		\log\det\Cov(\widehat Z_\bfm)
		=
		N\log\xi'(q)+N o_\ve(1)+o_N(N).
		\]
		Returning from \(\widehat Z_\bfm\) to
		\((\nabla H_N(\bfm),H_N(\bfm))\) changes the log-density by
		\(O(\log N)\). Therefore
		\begin{equation}\label{eq:one-point-maximal-density}
			\log\sup_z p_{(\nabla H_N(\bfm),H_N(\bfm))}(z)
			=
			-\frac N2\log(2\pi\xi'(q))
			+N o_\ve(1)+o_N(N).
		\end{equation}
		
		Set
		\[
		x_\bfm
		:=
		\frac{\bfb-\theta\xi'(q)\bfm}{\xi'(q)},
		\qquad
		L_\bfm
		:=
		\theta H_N(\bfm)
		+\langle x_\bfm,\nabla H_N(\bfm)\rangle.
		\]
		Define
		\[
		\frac{\dd\mathbb P^\bfm}{\dd\mathbb P}
		=
		\exp\left\{
		L_\bfm-\frac12\Var(L_\bfm)
		\right\}.
		\]
		Let \(p^{\,\bfm}\) denote the density of
		\((\nabla H_N(\bfm),H_N(\bfm))\) under \(\mathbb P^\bfm\).
		Under \(\mathbb P^\bfm\), the value and gradient have their original
		covariance and mean
		\[
		\dE^\bfm
		\bigl(\nabla H_N(\bfm),H_N(\bfm)\bigr)
		=
		\Cov\left(
		\bigl(\nabla H_N(\bfm),H_N(\bfm)\bigr),L_\bfm
		\right).
		\]
		The \(W_2\) condition gives
		\(\|x_\bfm\|_2=O(\sqrt N)\). Hence
		\eqref{eq:one-point-exact-target} and the Cameron--Martin formula give
		\begin{align}
			\log \varphi_\bfm\bigl(0,N(f'-P_\zeta)\bigr)
			={}&
			-\theta E_\bfm(f')
			-\langle x_\bfm,g_\bfm\rangle
			+\frac12\Var(L_\bfm)\notag\\
			&+
			\log p^{\,\bfm}
			(\widetilde g_\bfm,\widetilde E_\bfm)
			+N o_\ve(1)+o_N(N),
			\label{eq:one-point-CM-density}
		\end{align}
		Since the tilt does not change the covariance,
		\eqref{eq:one-point-maximal-density} gives
		\[
		\log p^{\,\bfm}
		(\widetilde g_\bfm,\widetilde E_\bfm)
		\le
		-\frac N2\log(2\pi\xi'(q))
		+N o_\ve(1)+o_N(N).
		\]
		
		Using \eqref{eq:one-point-field-covariance}, we obtain
		\begin{align*}
			-\langle x_\bfm,g_\bfm\rangle
			+\frac12\Var(L_\bfm)
			={}&
			-\frac1{2\xi'(q)}
			\left\|\bfb-\theta\xi'(q)\bfm\right\|_2^2
			+\frac N2\theta^2\xi(q)\\
			&-\frac N2\xi''(q)\widehat\zeta(q)^2
			+\frac N2\xi''(q)
			\left(
			\frac1N\langle\bfm,x_\bfm\rangle
			-\widehat\zeta(q)
			\right)^2\\
			&+N o_\ve(1)+o_N(N).
		\end{align*}
		The \(W_2\) condition in \(\mathcal D_{N,\ve}(f')\) and
		\eqref{eq:ising-ward-inner} give
		\[
		\frac1N\langle\bfm,x_\bfm\rangle
		=
		\frac1{\xi'(q)}
		\frac1N\langle\bfm,\bfb\rangle
		-\theta q_\bfm
		=
		\widehat\zeta(q)+o_\ve(1)+o_N(1).
		\]
		This proves \eqref{eq:one-point-joint-density-bound}.
		
		Assume now the two additional conditions in the lemma.
		Equations~\eqref{eq:one-point-field-covariance} and
		\eqref{eq:ising-ward-inner} give
		\begin{equation}\label{eq:one-point-tilted-gradient-mean}
			\frac1N
			\left\|
			\dE^\bfm\nabla H_N(\bfm)-\widetilde g_\bfm
			\right\|_2^2
			=o_\ve(1)+o_N(1).
		\end{equation}
		The value component of the same covariance calculation is
		\begin{equation}\label{eq:one-point-tilted-value-mean}
			\dE^\bfm H_N(\bfm)
			=
			N\left(
			\theta\xi(q)+\xi'(q)\widehat\zeta(q)
			\right)
			+N o_\ve(1)+o_N(N).
		\end{equation}
		Using \eqref{eq:one-point-gE},
		\eqref{eq:ising-fenchel-identity}, \eqref{eq:ward2},
		\eqref{eq:ising-ward-inner}, and
		\[
		\int_0^q t\xi''(t)\,\dd t
		=q\xi'(q)-\xi(q),
		\]
		we obtain
		\[
		E_\bfm(f')
		=
		N\left(
		\xi'(q)\widehat\zeta(q)+\theta\xi(q)
		\right)
		+N o_\ve(1)+o_N(N).
		\]
		Equations~\eqref{eq:one-point-exact-target},
		\eqref{eq:one-point-tilted-gradient-mean}, and
		\eqref{eq:one-point-tilted-value-mean}, together with the Gaussian
		density formula, give
		\[
		\log p^{\,\bfm}
		(\widetilde g_\bfm,\widetilde E_\bfm)
		=
		-\frac N2\log(2\pi\xi'(q))
		+N o_\ve(1)+o_N(N).
		\]
		Substitution into \eqref{eq:one-point-CM-density} proves the equality
		statement.
	\end{proof}
	
	%==============================================================================
	\subsection{Proof of Proposition~\ref{prop:susy-annealed}}
	%==============================================================================
	
	We combine the determinant and density estimates.
	
	\begin{proof}
		Fix \(f'\in[f-\ve,f+\ve]\). The \(W_2\) condition in
		\(\mathcal D_{N,\ve}(f')\) and
		\eqref{eq:one-point-hessian-susceptibility} give
		\[
		\frac1N\sum_{i=1}^N
		\partial_{xx}\Phi_\zeta(q,b_i)
		=
		\widehat\zeta(q)+o_\ve(1).
		\]
		Substitute the uniform-window estimate in
		Lemma~\ref{lem:one-point-conditional-determinant}\textup{(i)} and
		Lemma~\ref{lem:one-point-joint-density} into
		\eqref{eq:one-point-kac-rice-window}. Equations
		\eqref{eq:one-point-gE} and
		\eqref{eq:ising-fenchel-identity}, together with
		\[
		\int_0^q t\xi''(t)\,\dd t
		=q\xi'(q)-\xi(q)
		\]
		give the following upper bound for the Kac--Rice integrand:
		\begin{equation}\label{eq:one-point-coordinate-factorization}
			\begin{aligned}
				&\exp\left\{
				-\theta N(f'-P_\zeta)
				+N o_\ve(1)+o_N(N)
				\right\}
				(2\pi\xi'(q))^{-N/2}
				\prod_{i=1}^N\partial_{mm}h_\zeta(q,m_i)\\
				&\qquad\times
				\exp\left\{
				\theta\sum_{i=1}^N
				\bigl[
				\Phi_\zeta(q,b_i)-\Phi_\zeta(0,0)
				\bigr]
				-\frac1{2\xi'(q)}\sum_{i=1}^Nb_i^2
				\right\}.
			\end{aligned}
		\end{equation}
		
		Enlarge the domain to \([-1,1]^N\) and set
		\[
		b_i=\partial_m h_\zeta(q,m_i),
		\qquad
		\dd b_i
		=\partial_{mm}h_\zeta(q,m_i)\,\dd m_i.
		\]
		Since \(\zeta((0,q))=0\), the plateau transition formula
		\eqref{eq:AC-plateau-transition}, with lower endpoint \(0\), gives
		\[
		\frac1{\sqrt{2\pi\xi'(q)}}
		\int_{\mathbb R}
		\exp\left\{
		\theta\bigl[\Phi_\zeta(q,b)-\Phi_\zeta(0,0)\bigr]
		-\frac{b^2}{2\xi'(q)}
		\right\}\dd b
		=1.
		\]
		Every coordinate integral in
		\eqref{eq:one-point-coordinate-factorization} is therefore one, and
		\begin{equation}\label{eq:one-point-fixed-window-upper}
			\frac1N\log
			\dE\mathsf N_\ve
			\le
			\theta(P_\zeta-f+\ve)
			+o_\ve(1)+o_N(1).
		\end{equation}
		
		For the lower bound, restrict the \(f'\)-integral to
		\[
		|f'-f|\le N^{-2}
		\]
		and use the product measure whose one-coordinate density is
		\[
		\frac1{\sqrt{2\pi\xi'(q)}}
		\exp\left\{
		\theta\bigl[\Phi_\zeta(q,b)-\Phi_\zeta(0,0)\bigr]
		-\frac{b^2}{2\xi'(q)}
		\right\}\dd b.
		\]
		This is the law of \(X_q\). Set
		\(m_i=\partial_x\Phi_\zeta(q,b_i)\). The contact identity
		\eqref{eq:ising-contact} and the law of large numbers give
		\[
		\frac1N\|\bfm\|_2^2\longrightarrow q.
		\]
		The \(W_2\) window follows from the law of large numbers.
		The selected-potential identity
		\eqref{eq:ising-selected-potential}, the definition of \(P_\zeta\),
		and
		\(\int_0^q t\xi''(t)\,\dd t=q\xi'(q)-\xi(q)\) give
		\[
		\dE\bigl[
		\Phi_\zeta(q,X_q)-\Phi_\zeta(0,0)
		\bigr]
		=
		f-P_\zeta
		+\frac\theta2\int_0^q t\xi''(t)\,\dd t.
		\]
		The law of large numbers now gives \eqref{eq:ward2}. The optimality of
		\(\Shift_q\zeta\) at the limiting
		empirical law, together with continuity in the self-overlap and
		empirical law, gives
		\[
		0\le
		F_{\TAP,\zeta}(\bfm)-F_{\TAP}(\bfm)
		\le \frac{N\ve}{2}
		\]
		with probability tending to one under the product measure. On this event,
		\(|f'-f|\le N^{-2}\) implies the last condition defining
		\(\mathcal D_{N,\ve}(f')\) for all sufficiently large \(N\).
		
		Since
		\[
		N\int_{|f'-f|\le N^{-2}}\dd f'=\frac2N
		\]
		the determinant estimate and the equality case of
		Lemma~\ref{lem:one-point-joint-density} give the lower bound below.
		Equation~\eqref{eq:one-point-fixed-window-upper} gives the upper bound:
		\[
		\theta(P_\zeta-f)
		\le
		\liminf_{\ve\downarrow0}
		\liminf_{N\to\infty}
		\frac1N\log
		\dE\mathsf N_\ve
		\le
		\limsup_{\ve\downarrow0}
		\limsup_{N\to\infty}
		\frac1N\log
		\dE\mathsf N_\ve
		\le
		\theta(P_\zeta-f).
		\]
		This proves the proposition.
	\end{proof}
	
	%==============================================================================
	\section{Spherical examples}
	\label{sec:spherical-examples}
	%==============================================================================
	
	We work in the spherical model, where the relevant Kac--Rice exponents can be
	evaluated explicitly. Let \(H_N\) be the centered Gaussian Hamiltonian on
	\(\mathbb R^N\) with covariance
	\[
	\dE H_N(\bfm)H_N(\bfm')
	=N\xi\bigl(R(\bfm,\bfm')\bigr),
	\qquad
	R(\bfm,\bfm')=\frac1N\langle\bfm,\bfm'\rangle.
	\]
	The spins lie on the sphere \(\|\sigma\|_2^2=N\), while the TAP
	magnetizations lie in
	\[
	B_N^\circ:=\{\bfm\in\mathbb R^N:\|\bfm\|_2^2<N\}.
	\]
	For a radial band correction \(\mathcal V\), the TAP free energy is
	\[
	F_{\TAP}(\bfm)=H_N(\bfm)+N\mathcal V(q_\bfm).
	\]
	For \(\zeta\in\mc P([q,1])\), recall that
	\[
	\zhat(s)=\int_s^1\zeta([0,t])\,\dd t.
	\]
	A critical point \(\bfm\) with \(q_\bfm=q\), associated with \(\zeta\), is
	SUSY when it satisfies the Ward identity
	\begin{equation}\label{eq:ward1}
		\xi'(q)\zhat(0)\zhat(q)=q.
	\end{equation}
	
	\subsection{Positive complexity with a negative transverse Hessian bulk}
	
	Let \(\Sigma_{\mathrm{tot}}(q)\) denote the fixed-radius annealed exponent.
	
	\begin{fact}
		For \(\xi(x)=x^4+x^2/10\) and \(q=1/2\), the band minimizer is
		\(\delta_{1/2}\), the Plefka inequality is strict, and
		\(\Sigma_{\mathrm{tot}}(1/2)>0\).
		The conditioned Hessian has a negative transverse bulk and one positive
		radial outlier.
	\end{fact}
	
	Take
	\begin{equation}\label{eq:model}
		\xi(x)=x^4+\tfrac1{10}x^2,
		\qquad
		q=\tfrac12.
	\end{equation}
	
	For \(\zeta\in\mc P([q,1])\), the band correction is
	\begin{equation}\label{eq:Vzeta-def}
		\mathcal V_\zeta(q)
		=
		\frac12\int_q^1
		\left[
		\xi''(s)\zhat(s)+\frac1{\zhat(s)}-\frac1{1-s}
		\right]\dd s
		+\frac12\log(1-q).
	\end{equation}
	For \(\zeta=\delta_q\), one has \(\zhat(s)=1-s\) on \([q,1]\). To verify
	that \(\delta_q\) minimizes \(\mathcal V_\zeta(q)\), fix
	\(\zeta\in\mc P([q,1])\) and consider
	\[
	\zeta_\varepsilon=(1-\varepsilon)\delta_q+\varepsilon\zeta,
	\qquad 0\le\varepsilon\le1.
	\]
	Writing
	\[
	\eta(s)=\zhat(s)-(1-s),
	\qquad
	\zhat_\varepsilon(s)=(1-s)+\varepsilon\eta(s),
	\]
	we have \(\eta(s)\le0\).
	Differentiating \eqref{eq:Vzeta-def} at \(\varepsilon=0\) gives
	\begin{equation}\label{eq:gateaux}
		\left.\frac{\dd}{\dd\varepsilon}
		\mathcal V_{\zeta_\varepsilon}(q)\right|_{\varepsilon=0}
		=
		\frac12\int_q^1
		\left[\xi''(s)-\frac1{(1-s)^2}\right]\eta(s)\,\dd s.
	\end{equation}
	Since \(\eta\le0\), this derivative is nonnegative whenever
	\[
	\xi''(s)(1-s)^2\le1,
	\qquad q\le s\le1.
	\]
	For the mixture \eqref{eq:model}, the function
	\[
	\left(12s^2+\tfrac15\right)(1-s)^2
	\]
	is decreasing on \([1/2,1]\), and \(\xi''(q)(1-q)^2<1\). Hence
	\eqref{eq:gateaux} is nonnegative. The map
	\(\varepsilon\mapsto\mathcal V_{\zeta_\varepsilon}(q)\) is convex because
	\(\xi''\zhat_\varepsilon\) is affine and
	\(\zhat_\varepsilon^{-1}\) is convex. Hence
	\[
	\mathcal V_\zeta(q)\ge\mathcal V_{\delta_q}(q).
	\]
	The band correction is therefore minimized by \(\delta_{1/2}\). The Plefka
	inequality is strict. The same inequalities hold for \(r\) in a
	neighborhood of \(q=1/2\), so the radial correction there is
	\[
	\mathcal V(r):=\mathcal V_{\delta_r}(r).
	\]
	
	For this band correction, the TAP gradient condition is
	\(\nabla H_N(\bfm)=\lambda\bfm\), where
	\[
	\lambda=-2\mathcal V'(q)
	=\xi''(q)(1-q)+\frac1{1-q}.
	\]
	The fixed-radius Kac--Rice formula, including the shell volume, gradient
	density, and tangential determinant, gives
	\[
	\Sigma_{\mathrm{tot}}(q)
	=\frac12\log\frac{eq}{\xi'(q)}
	-
	\frac{q\lambda^2}{2(\xi'(q)+q\xi''(q))}
	+\frac{\xi''(q)(1-q)^2}{2}
	-\log(1-q).
	\]
	The radial Schur complement does not change the exponential rate. To
	determine its sign, set
	\[
	\begin{aligned}
		s_q&:=\xi''(q),\\
		\tau_q
		&:=
		\xi''(q)+q\xi'''(q)
		-\frac{q\xi''(q)^2}{\xi'(q)},\\
		\alpha_q
		&:=
		-\lambda+4q\mathcal V''(q)
		+
		\frac{
			q\lambda\bigl(q\xi'''(q)+2\xi''(q)\bigr)
		}{
			\xi'(q)+q\xi''(q)
		}.
	\end{aligned}
	\]
	Under the conditioning \(\nabla F_{\TAP}(\bfm)=0\), Gaussian regression in
	the decomposition
	\(\dR^N=\operatorname{span}(\bfm)\oplus\bfm^\perp\) gives
	\[
	\nabla^2F_{\TAP}(\bfm)
	\stackrel{\mathrm d}=
	\begin{pmatrix}
		\alpha_q+o_{\dP}(1)&y_N^{\mathsf T}\\
		y_N&\sqrt{s_q}\,W_{N-1}-\lambda I_{N-1}
	\end{pmatrix},
	\]
	where \(W_{N-1}\) is a GOE matrix, \(y_N\) is centered Gaussian with
	covariance \(N^{-1}\tau_qI_{N-1}\), and the two are independent. Write
	\[
	G_q(t):=\frac{t-\sqrt{t^2-4s_q}}{2s_q},
	\qquad t>2\sqrt{s_q}.
	\]
	The radial outlier \(z_\star\) is determined by
	\begin{equation}\label{eq:spherical-radial-outlier}
		z_\star
		=
		\alpha_q+\tau_qG_q(\lambda+z_\star).
	\end{equation}
	For the mixture \eqref{eq:model}, at \(q=1/2\),
	\eqref{eq:spherical-radial-outlier} gives
	\begin{align*}
		\Sigma_{\mathrm{tot}}(q)
		&=0.0292591294\ldots>0,\\
		-\lambda+2\sqrt{\xi''(q)}
		&=-0.0222912\ldots<0,\\
		z_\star
		&=0.0502581\ldots>0.
	\end{align*}
	Thus the fixed-radius annealed exponent is positive. The transverse bulk is
	strictly negative, while the radial outlier is positive. The corresponding
	critical points have index \(N-1\) with probability tending to one.
	
	\subsection{Cancellation of the SUSY complexity}
	
	We next verify the cancellation of the signed SUSY complexity at a fixed
	radius.
	
	\begin{fact}
		For an even mixed spherical model, let \(q\in(0,1)\) and
		\(\zeta\in\mc P([q,1])\). Assume that
		\[
		\xi'(q)\zhat(0)\zhat(q)=q,
		\qquad
		\xi''(q)\zhat(q)^2\le1.
		\]
		For \(\bfm\in B_N^\circ\) with \(\|\bfm\|_2^2=Nq\), let
		\(\dE_\bfm\) denote conditional expectation given
		\(\nabla F_{\TAP}(\bfm)=0\). Assume that, as \(N\to\infty\),
		\[
		\frac1N\log\left|
		\dE_\bfm\det\nabla^2F_{\TAP}(\bfm)
		\right|
		-
		\frac1N\log\left|
		\dE_\bfm\det\left(
		\nabla^2F_{\TAP}(\bfm)|_{\bfm^\perp}
		\right)
		\right|
		=o(1).
		\]
		Then the signed SUSY exponent at radius \(q\), denoted by
		\(\Sigma_{\mathrm{SUSY}}^{\mathrm{shell}}(q)\), is zero.
	\end{fact}
	
	Since \(\zeta\) is supported on \([q,1]\), the function \(\zhat\) is
	constant on \([0,q]\). Therefore \eqref{eq:ward1} is equivalent to
	\begin{equation}\label{eq:susy-radius-identity}
		\xi'(q)\zhat(q)^2=q.
	\end{equation}
	We first record the gradient density.
	\begin{lemma}
		Let \(q\in(0,1)\) satisfy \(\xi'(q)>0\), let
		\(\zeta\in\mc P([q,1])\), and let \(\bfm\in B_N^\circ\) satisfy
		\(\|\bfm\|_2^2=Nq\). Let
		\[
		\lambda=-2\mathcal V'_\zeta(q)
		=\xi''(q)\zhat(q)+\frac1{\zhat(q)}.
		\]
		Then
		\begin{equation}\label{eq:prob0}
			\log p_{\nabla F_{\TAP}(\bfm)}(0)
			=-\frac{Nq\lambda^2}{2(\xi'(q)+q\xi''(q))}
			-\frac{N-1}{2}\log(2\pi\xi'(q))
			-\frac12\log(2\pi(\xi'(q)+q\xi''(q))).
		\end{equation}
		If in addition \(\xi'(q)\zhat(q)^2=q\), then
		\begin{equation}\label{eq:prob0-ward}
			\frac1N\log p_{\nabla F_{\TAP}(\bfm)}(0)
			=-\frac12
			-\frac{q\xi''(q)}{2\xi'(q)}
			-\frac12\log(2\pi\xi'(q))
			+O\left(\frac1N\right).
		\end{equation}
	\end{lemma}
	
	\begin{proof}
		Since
		\[
		\nabla F_{\TAP}(\bfm)
		=\nabla H_N(\bfm)+2\mathcal V'_\zeta(q)\bfm,
		\]
		the event \(\nabla F_{\TAP}(\bfm)=0\) is
		\(\nabla H_N(\bfm)=\lambda\bfm\). The tangent components of
		\(\nabla H_N(\bfm)\) are independent centered Gaussian variables with
		variance \(\xi'(q)\). The radial component has variance
		\(\xi'(q)+q\xi''(q)\) and is evaluated at \(\lambda\sqrt{Nq}\).
		This gives \eqref{eq:prob0}. Under \eqref{eq:susy-radius-identity},
		\[
		\lambda^2
		=\frac{(\xi'(q)+q\xi''(q))^2}{q\xi'(q)}.
		\]
		Substitution into \eqref{eq:prob0} proves \eqref{eq:prob0-ward}.
	\end{proof}
	
	\begin{proof}[Proof of the fact]
		The Kac--Rice formula gives
		\[
		\Sigma_{\mathrm{SUSY}}^{\mathrm{shell}}(q)
		=\lim_{N\to\infty}
		\left\{
		\frac12\log(2\pi e q)
		+\frac1N\log p_{\nabla F_{\TAP}(\bfm)}(0)
		+\frac1N\log\left|
		\dE_\bfm\det\nabla^2F_{\TAP}(\bfm)
		\right|
		\right\}.
		\]
		The equation
		\[
		\lambda=\xi''(q)a+\frac1a
		\]
		has two roots, possibly equal. Since \(\zhat(q)\) is a root, the non-strict
		Plefka inequality \(\xi''(q)\zhat(q)^2\le1\) forces it to be the smaller
		root:
		\[
		a
		=
		\frac{
			\lambda-\sqrt{\lambda^2-4\xi''(q)}
		}{
			2\xi''(q)
		}
		=
		\zhat(q).
		\]
		For this root, the signed tangential determinant contributes
		\begin{equation}\label{eq:susy-det}
			\frac1N\log\left|
			\dE_\bfm\det\left(
			\nabla^2F_{\TAP}(\bfm)|_{\bfm^\perp}
			\right)
			\right|
			=\frac{\xi''(q)\zhat(q)^2}{2}-\log\zhat(q)+o(1).
		\end{equation}
		By the determinant hypothesis, the full Hessian has the same signed
		determinant exponent.
		
		Using \eqref{eq:susy-radius-identity} in \eqref{eq:prob0-ward} and
		\eqref{eq:susy-det}, the three terms are
		\begin{align*}
			\text{shell volume}
			&=\frac12\log(2\pi)+\frac12+\frac12\log q,\\
			\text{gradient density}
			&=-\frac12-\frac{q\xi''(q)}{2\xi'(q)}
			-\frac12\log(2\pi)-\frac12\log\xi'(q),\\
			\text{determinant}
			&=\frac{q\xi''(q)}{2\xi'(q)}
			-\frac12\log q+\frac12\log\xi'(q).
		\end{align*}
		Their sum is zero:
		\[
		\Sigma_{\mathrm{SUSY}}^{\mathrm{shell}}(q)=0.
		\]
	\end{proof}
	
	\subsection{Annealed and SUSY complexities above equilibrium}
	
	Let \(\Sigma_{\mathrm{ann}}(q,f)\) denote the fixed-radius,
	fixed-level annealed exponent, and set
	\[
	\Sigma_{\mathrm{ann}}(f)
	:=
	\sup_{0<q'<1}\Sigma_{\mathrm{ann}}(q',f).
	\]
	Define the spherical SUSY complexity by
	\begin{equation}\label{eq:cex-susy-def}
		\Sigma_{\mathrm{SUSY}}(f)
		=\inf_{0<\theta<1}\bigl[\Lambda_{\mathrm{sph}}(\theta)-\theta f\bigr],
		\qquad
		\Lambda_{\mathrm{sph}}(\theta)
		=\theta
		\inf_{\substack{\zeta\in\mc P([0,1])\\
				\zeta(\{0\})\ge\theta}}
		\mathcal V_\zeta(0).
	\end{equation}
	
	\begin{fact}
		There exist an even mixed spherical model and a level
		\(f>f_{\eq}^{\mathrm{sph}}\) such that
		\[
		\Sigma_{\mathrm{ann}}(f)>0>\Sigma_{\mathrm{SUSY}}(f).
		\]
	\end{fact}
	
	For the counterexample, set
	\[
	\xi(x)=x^8+2x^2,
	\qquad
	f=1.402,
	\qquad
	q=0.744.
	\]
	The spherical equilibrium free energy is
	\[
	f_{\eq}^{\mathrm{sph}}
	=
	\inf_{\zeta\in\mc P([0,1])}\mathcal V_\zeta(0).
	\]
	Define
	\[
	\theta_0=0.0926,
	\qquad
	r=0.5356,
	\qquad
	\zeta_0=\theta_0\delta_0+(1-\theta_0)\delta_r.
	\]
	Its cumulant is
	\[
	\zhat_0(s)=
	\begin{cases}
		(1-r)+\theta_0(r-s),&0\le s<r,\\
		1-s,&r\le s\le1,
	\end{cases}.
	\]
	Direct evaluation of \eqref{eq:Vzeta-def} gives
	\[
	f_{\eq}^{\mathrm{sph}}\le\mathcal V_{\zeta_0}(0)<f.
	\]
	
	The convexity calculation leading to \eqref{eq:gateaux} shows that the band
	minimizer at \(q\) is \(\delta_q\), since
	\[
	\xi''(s)(1-s)^2=(56s^6+4)(1-s)^2
	\]
	is decreasing on \([q,1]\) and \(\xi''(q)(1-q)^2<1\).
	The fixed-radius, fixed-level annealed Kac--Rice exponent is
	\begin{equation}\label{eq:cex-Sigma}
		\begin{split}
			\Sigma_{\mathrm{ann}}(q,f)
			={}&\frac12\log\frac{eq}{\xi'(q)}
			+\frac{\xi''(q)(1-q)^2}{2}
			-\log(1-q)\\
			&-
			\frac{
				\xi(q)q\lambda^2
				-2\xi'(q)q\lambda E
				+(\xi'(q)+q\xi''(q))E^2
			}{2\Delta},
		\end{split}
	\end{equation}
	where
	\[
	\lambda=\xi''(q)(1-q)+\frac1{1-q},
	\qquad
	E=f-\mathcal V_{\delta_q}(q),
	\]
	and
	\[
	\Delta
	=\xi(q)(\xi'(q)+q\xi''(q))-q\xi'(q)^2.
	\]
	
	Since \(\zeta_0(\{0\})=\theta_0\),
	\eqref{eq:cex-susy-def} gives
	\[
	\Sigma_{\mathrm{SUSY}}(f)
	\le\theta_0\bigl(\mathcal V_{\zeta_0}(0)-f\bigr).
	\]
	
	Direct evaluation of \eqref{eq:Vzeta-def} and
	\eqref{eq:cex-Sigma} gives
	\begin{equation*}
		\begin{aligned}
			\mathcal V_{\zeta_0}(0)
			&=1.401015445\ldots<f,\\
			\xi''(q)(1-q)^2
			&=0.884595576\ldots<1,\\
			\Sigma_{\mathrm{ann}}(q,f)
			&=0.002039886\ldots>0,\\
			\theta_0\bigl(\mathcal V_{\zeta_0}(0)-f\bigr)
			&=-9.116\ldots\times10^{-5}<0,\\
			\xi'(q)(1-q)^2
			&=0.261192846\ldots\ne q.
		\end{aligned}
	\end{equation*}
	It follows that
	\[
	\Sigma_{\mathrm{ann}}(f)
	\ge\Sigma_{\mathrm{ann}}(q,f)
	>0
	>\Sigma_{\mathrm{SUSY}}(f).
	\]
	The positive contribution at radius \(q\) is not SUSY, since
	\(\xi'(q)(1-q)^2\ne q\) violates \eqref{eq:ward1}.
	
	%==============================================================================
	\bibliographystyle{amsplain}
	\bibliography{references}
	
\end{document}